\documentclass[11pt]{article}
\usepackage[a4paper]{geometry}
\usepackage{mathrsfs}
\usepackage{amsmath}
\usepackage{amsthm}
\usepackage{amssymb}
\usepackage[hidelinks]{hyperref}
\usepackage{manfnt}\makeatletter
\theoremstyle{plain}
\newtheorem{thm}{\protect\theoremname}[section]
\theoremstyle{plain}
\newtheorem{lem}[thm]{\protect\lemmaname}
\theoremstyle{plain}
\newtheorem{prop}[thm]{\protect\propositionname}
\newtheorem{cor}[thm]{Corollary}
\theoremstyle{definition}

\newtheorem{remark}[thm]{Remark}

\usepackage{color}
\usepackage{todonotes}

\newcommand*{\bbP}{\mathbb{P}}

\newcommand*{\scrF}{\mathscr{F}}

\definecolor{electricultramarine}{rgb}{0.25, 0.0, 1.0}

\newcommand*{\fzcst}[1]{\relax\ifmmode\text{\textcolor{electricultramarine}{\sout{\ensuremath{#1}}}}\else\textcolor{electricultramarine}{\sout{#1}}\fi}

\definecolor{darkgreen}{RGB}{30,130,80}

\usepackage[normalem]{ulem}
\newcommand*{\ykcst}[1]{\relax\ifmmode\text{\textcolor{darkgreen}{\sout{\ensuremath{#1}}}}\else\textcolor{darkgreen}{\sout{#1}}\fi}

\numberwithin{equation}{section}

\makeatletter
\def\renewtheorem#1{\expandafter\let\csname#1\endcsname\relax
  \expandafter\let\csname c@#1\endcsname\relax
  \gdef\renewtheorem@envname{#1}
  \renewtheorem@secpar
}
\def\renewtheorem@secpar{\@ifnextchar[{\renewtheorem@numberedlike}{\renewtheorem@nonumberedlike}}
\def\renewtheorem@numberedlike[#1]#2{\newtheorem{\renewtheorem@envname}[#1]{#2}}
\def\renewtheorem@nonumberedlike#1{  
\def\renewtheorem@caption{#1}
\edef\renewtheorem@nowithin{\noexpand\newtheorem{\renewtheorem@envname}{\renewtheorem@caption}}
\renewtheorem@thirdpar
}
\def\renewtheorem@thirdpar{\@ifnextchar[{\renewtheorem@within}{\renewtheorem@nowithin}}
\def\renewtheorem@within[#1]{\renewtheorem@nowithin[#1]}
\makeatother

\theoremstyle{definition}
\renewtheorem{defn}[thm]{Definition}
\renewtheorem{assumption}{Assumption}
\newtheorem{Remark}[thm]{Remark}

\usepackage[hang,flushmargin]{footmisc} 

\makeatother

\usepackage[maxbibnames=99]{biblatex}
\providecommand{\examplename}{Example}
\providecommand{\lemmaname}{Lemma}
\providecommand{\propositionname}{Proposition}
\providecommand{\theoremname}{Theorem}

\usepackage{authblk}

\begin{document}
\title{Dynamical Low-Rank Approximation for\\ Stochastic Differential Equations}
\author[1]{Yoshihito Kazashi}
\author[2]{Fabio Nobile}
\author[3]{Fabio Zoccolan}

\affil[1]{Department of Mathematics \& Statistics, University of Strathclyde, 26 Richmond St., Glasgow, G1 1XH, UK. email: y.kazashi@strath.ac.uk}
\affil[2]{Institut de Mathématiques, École Polytechnique Fédérale de Lausanne, 1015 Lausanne, Switzerland. email: fabio.nobile@epfl.ch}
\affil[3]{Institut de Mathématiques, École Polytechnique Fédérale de Lausanne, 1015 Lausanne, Switzerland. email: fabio.zoccolan@epfl.ch}

\date{} \maketitle

\begin{abstract}
    \noindent	In this paper, we set the mathematical foundations of the Dynamical Low-Rank Approximation (DLRA) method for stochastic differential equations (SDEs). 
	DLRA aims at approximating the solution as a linear combination of a small number of basis vectors with random coefficients (low rank format) with the peculiarity that both the basis vectors and the random coefficients vary in time.

	While the formulation and properties of DLRA are now well understood for random/parametric equations, the same cannot be said for SDEs and this work aims to fill this gap. We start by rigorously formulating a Dynamically Orthogonal (DO) approximation (an instance of DLRA successfully used in applications) for SDEs, which we then generalize to define a parametrization independent DLRA for SDEs. We show local well-posedness of the DO equations and their equivalence with the DLRA formulation. 
	We also characterize the explosion time of the DO solution by a loss of linear independence of the random coefficients defining the solution expansion and give sufficient conditions for global existence.
\end{abstract}

\maketitle

\section{Introduction}
This paper is concerned with the theoretical foundation of dynamical low-rank methods for stochastic differential equations (SDEs).  
Low-rank methods aim to approximate solutions of high-dimensional differential equations in a well chosen low dimensional subspace.
Such methods are widely used in computational science and industrial applications  \cite{Bachmayr.M_2023_LowrankTensorMethods,lubich2008quantum,quarteroni2014reduced,schilders2008model}. 
Our focus here is on low-rank methods for high-dimensional SDEs, which is a case of primary interest, for instance 
in finance \cite{redmann2021low,redmann2022solving}, or in various applications such as biology \cite{adamer2020coloured} or machine learning \cite{muzellec2022dimension} where SDEs often appear as discretizations of Stochastic Partial Differential Equations (SPDEs).

{Among such low-rank methods, the so-called Dynamical Low-Rank Approximation (DLRA) has gained notable attention and has been successfully applied in many applications in recent years \cite{koch2007dynamical,kieri2016discretized,nonnenmacher2008dynamical}. 
	Unlike traditional reduced basis methods, however, DLRA allows the low-dimensional subspace to evolve over time, enabling it to adapt to possibly rapid temporal variations of the solution.}
    Thus DLRA evolves in a low-rank manifold, not in a fixed linear space of low dimension.  
	In its original form \cite{koch2007dynamical}, such evolution in a manifold is achieved by enforcing the temporal derivative of the DLRA to be in the corresponding tangent space, which requires the differentiability of the approximate solution in time.

{An equivalent formulation of the DLRA
is the Dynamically Orthogonal (DO) approximation \cite{sapsis2009dynamically}. This strategy can approximate random time-dependent equations by using a sum of a small number of products between spatial and stochastic basis functions, thus leading to a low-rank approximation with explicit parametrization of the low-rank factors. 
In the context of 
random partial differential equations (RPDEs), the DO/DLRA methodology has been successfully applied in several problems, showing promising
computational results  \cite{choi2014equivalence,cheng2013dynamically,cheng2013dynamicallyII,feppon2018dynamically, musharbash2015error,musharbash2018dual,musharbash2020symplectic,kazashi2021stability, sapsis2009dynamically}. }

{The goal of this work is to set a theoretical framework of DLRA for SDEs. 
	This is a challenging task, because, unlike the RPDEs setting, the differentiability in time of the solution cannot be exploited due to the presence of the rough diffusion term. 
	Our approach is to first define DO dynamics for a general SDE, specifying conditions for the well-posedness of this framework. 
	From this DO system, we derive an equation that is independent of the parametrization. We define DLRA for SDEs as the solution to this parameter-independent equation, and show that these DO and DLRA methodologies are equivalent. 
}

The DO approximation is given by the solution of a system of equations, which provide the evolution of the deterministic and stochastic basis functions. This system is however highly nonlinear, even if the original dynamics was linear. In addition, its well-posedness is not obvious due to the presence of the inverse of a Gramian, which is not guaranteed to exist at all times. Moreover, in the specific case of SDEs, the DO solution depends on the law of the whole process at every instant of time. These features make the well-posedness study challenging and not-standard, questioning also whether the existence of solutions is global or only local and, in the latter case, what happens to the solution at the explosion time. Besides being of interest in itself, a well-posedness study of the DO equations is also important to derive consistent and stable time discretization schemes, which are eventually needed to apply these techniques in real-life problems.

For RPDEs, theoretical results on the existence and
uniqueness of the solution of the DO equations \cite{kazashi2021existence}, as well as
stability and error estimates of time discretization schemes \cite{kazashi2021stability,musharbash2015error,vidlickova2022dynamical} are available. 
An existence and uniqueness result is also available for the DLR approximation of two-dimensional deterministic PDEs \cite{bachmayr2021existence}, where the sum involves products of basis functions for the different spatial variables. 
In contrast, for SDEs well-posedness results remain largely unexplored. {This lack of theory for SDEs is unfortunate, since DO approximations are very appealing and, despite no results on its well-posedness, have been also used in several problems modeled by SDEs (see e.g. \cite{sapsis2009dynamically,sapsis2013blended,sapsis2013blending,ueckermann2013numerical}). The reader can find other approaches of initiating a DLRA framework starting from the Fokker-Plank equation in \cite{sapsis2009dynamically,cao2018stochastic}.}

This paper aims at filling this gap between application and theory by:
\begin{itemize}
	\item  (Re-)deriving the DO equations for SDEs. We also derive the corresponding parametrization independent DLRA equation for SDEs.
	\item Showing the equivalence between DO approximation and DLRA. By exploiting this equivalence, we show that the DO equations and the DLRA equation are well-posed.
	\item Characterising the finite explosion time in terms of the linear independence of the stochastic DO basis.
	\item Discussing the extension of DLRA beyond the explosion time. We also provide a sufficient condition under which the DO/DLRA solution exists globally.
\end{itemize}

More precisely, we start by revisiting the DO equations for SDEs, which were originally derived in  \cite{sapsis2009dynamically} relying on the formal assumption that the DO solution is time differentiable, a property that SDE solutions do not possess. 
Deriving DO equations without using the time derivatives is a fundamental challenge, since the essence of the usual DLRA methodology consists in projecting the time derivative onto the tangent space of the low-rank manifold. 
We overcome this challenge by pushing the differentiability to the spatial basis, while the stochastic basis remains non-differentiable.  
It turns out that this alternative approach leads to the same DO equations as in \cite{sapsis2009dynamically}, supporting their validity as a correct dynamical low-rank approximation.

In the DLRA literature \cite{koch2007dynamical,kieri2016discretized}, it is well known by now that, when applied to RPDEs, the DO formulation is just a specific parametrization of the DLRA, the parameters being the deterministic and stochastic time-dependent bases defining the DO solution \cite{musharbash2015error,kazashi2021existence}.
In this paper, we derive an analogous result for SDEs. In particular, starting from the DO equations, we derive a parameter independent low-rank approximation which we name DLRA for SDEs, and show its equivalence with the DO formulation. The main result of this work is to prove local existence and uniqueness to both DO and DLRA equations.
As a part of our existence result, we give a characterisation of the interval $[0,T_e)$ on which the solution exists. 
More specifically, we show that, if the \emph{explosion time} $T_e$ is finite, then the stochastic DO basis must become linearly dependent at $T_e$. 

Although the characterisation of the explosion time provides us a valuable insight into the DO solution, it is not satisfactory from the practical point of view; if the true solution still exists at $T_e$ and beyond, we expect any sensible approximation to exist as well. Under some mild integrability condition on the initial datum, we show that, while the DO solution ceases to exist at the explosion time $T_e$, the DLRA can be continuously extended up to $T_e$ and beyond. 
Thus, our findings offer insights on how to continue a DO approximation \emph{beyond} the explosion time $T_e$. 
As a final result, we show that a sufficient condition for global existence of a DO solution (i.e.\ $T_e =  + \infty$) is that the noise in the SDE is non-degenerate.

{In passing, we point out that the DO/DLRA equation derived in this work cannot be regarded as a trivial application of the standard theory of SDEs on manifolds (see e.g.\ \cite{gliklikh2010global,hsu2002stochastic}). Indeed, the projection operator employed in the DO decomposition is not standard, as it also involves a purely stochastic term that depends on the average of all paths. }

The rest of the paper is organised as follows. 
Section~\ref{sec: DO eq} introduces the DO equations, together with theoretical justifications for why they are sensible. Furthermore, it introduces a DLRA formulation that does not depend on the parametrization.
Section~\ref{sec: well-posed} concerns the local well-posedness of the DO equations, where we also show the continuity of the solution with respect to the initial data. In Section~\ref{sec: maximality}, we study the behaviour of the solution up to and beyond the explosion time $T_e$. In Section 5 we draw some conclusions and perspectives.

\section{The DO equations}\label{sec: DO eq}
Let $(\Omega,\mathscr{F},\bbP;(\scrF_{t})_{t\geq0})$ be a filtered complete 
probability space with the usual conditions; see for example \cite[Remark 6.24]{schilling2021brownian}.
Consider the stochastic
differential equation (SDE)
\begin{equation}
	X^{\mathrm{true}}(t)=X_{0}^{\mathrm{true}}+\int_{0}^{t}a(s,X^{\mathrm{true}}(s))\,ds+\int_{0}^{t}b(s,X^{\mathrm{true}}(s))\,
	dW_s
	,\label{eq:true-eq}
\end{equation}
where 
$W_t=(W^{1}_t,\dots,W^{m}_t)^{\top}$ 
is a standard $m$-dimensional
$(\scrF_{t})$-Brownian motion.
Here, we used the notation $X^{\mathrm{true}}(t,\omega)=(X_{1}^{\mathrm{true}}(t,\omega),\dots,X_{d}^{\mathrm{true}}(t,\omega))^{\top}\in\mathbb{R}^{d}$ for $t \geq 0$ and $\omega \in \Omega$. Let  $|\cdot|$ and $\|\cdot\|_{\mathrm{F}}$ denote 
the Euclidean norm and the Frobenius norm, respectively.
We work with the following general assumptions {on the coefficients of the SDE \eqref{eq:true-eq}}. 
\begin{assumption}\label{assump:Lip}
	The drift coefficient  $a\colon[0,\infty)\times\mathbb{R}^{d}\to\mathbb{R}^{d}$ 
	and the diffusion coefficient  $b\colon[0,\infty)\times\mathbb{R}^{d}\to\mathbb{R}^{d\times m}$
    {are jointly measurable and} Lipschitz continuous with respect to the second variable, uniformly in time, i.e.
	\begin{equation}
		\begin{cases}
			\!\!\!\! & |a(s,x)-a(s,y)|\leq C_{\mathrm{Lip}}|x-y|\\
			\!\!\!\! & \|b(s,x)-b(s,y)\|_{\mathrm{F}}\leq C_{\mathrm{Lip}}|x-y|,
		\end{cases}\label{eq:lip}
	\end{equation}
	for some constant $C_{\mathrm{Lip}}>0$. 
\end{assumption}
\begin{assumption}\label{assump:lin-growth}
	The drift $a$
	and the diffusion $b$ satisfy the following linear-growth bound condition {uniformly in time:}
	\begin{equation}
		|a(s,x)|^{2}+\|b(s,x)\|_{\mathrm{F}}^{2}\leq C_{\mathrm{lgb}}(1+|x|^{2})\label{eq:lin-growth}
	\end{equation}
	for some constant $C_{\mathrm{lgb}}>0$.
\end{assumption}
{The choice of the Frobenius norm in Assumption~\ref{assump:Lip} is made primarily for convenience when utilizing Ito's isometry; this is standard practice. 
Since we are in a finite-dimensional setting, the results presented in this work remain valid under Lipschitz continuity with respect to other norms, with a modified constant. 
It is also worth mentioning that in the infinite-dimensional case, Lipschitz continuity is typically considered with respect to the Hilbert–Schmidt norm between appropriate spaces; see for example \cite{da2014stochastic}; hence our choice of the Frobenius norm is a first step in this direction.}
Furthermore, we assume that the initial condition $X_{0}^{\mathrm{true}}$ in \eqref{eq:true-eq} satisfies the following:
\begin{assumption}
	\begin{equation}\label{eq:initial value}
		X^{\mathrm{true}}_0 \mbox{ is } \mathcal{F}_0\mbox{-measurable and satisfies } \mathbb{E}[|X^{\mathrm{true}}_0|^2] < +\infty.
	\end{equation}
\end{assumption}
Under these assumptions, equation \eqref{eq:true-eq} has a unique strong solution; see for example \cite[Theorem 21.13]{schilling2021brownian}.

Let us consider a positive integer $R$ such that $R\leq d$. To numerically approximate \eqref{eq:true-eq}, in this work we consider dynamically orthogonal approximations of the form
\begin{equation} X^{\mathrm{true}}\approx
	X:=\boldsymbol{U}^{\top}\boldsymbol{Y}:=\sum_{j=1}^{R}U^{j}Y^{j}\in\mathbb{R}^d,\label{eq:DO-form}
\end{equation}
where $\boldsymbol{U} = \left(\boldsymbol{U}_t\right)_{t\in[0,T]}$ is a deterministic absolutely continuous matrix-valued function that gives a matrix in $\mathbb{R}^{R \times d}$ {with orthogonal rows} for all $t$, whereas $\boldsymbol{Y}=(\boldsymbol{Y}_{t})_{t\in[0,T]}$ is an Itô process with values in $\mathbb{R}^{R}$ with linearly independent components for all $t$. 
It is worth pointing out that for all $t \in [0,T]$ and $\omega \in \Omega$, the approximate process $X_t(\omega)$ belongs to an $R$ dimensional vector space $\mathrm{span}\{U^{1}_t, \dots, U^{R}_t\}$ spanned by the rows $U^{1}_t,\dots, U^{R}_t$ of $\boldsymbol{U}_{t}$, whereas each component $X^j_t, j=1,\dots, d$, belongs to  $\mathrm{span}\{Y_{t}^{1},\dots,Y_{t}^{R}\}$. 

Approximations of the form \eqref{eq:DO-form} for SDEs have been considered already in \cite{sapsis2009dynamically}, where the following evolution equations, hereafter called \textit{DO equations} were derived for the factors $(\boldsymbol{U},\boldsymbol{Y})$ by {using the same formal procedure as for ODEs via treating the process $\boldsymbol{Y}$ as differentiable in time:}
\begin{align}
	{C}_{\boldsymbol{Y}_{t}}\dot{\boldsymbol{U}}_{t} & =\mathbb{E}[\boldsymbol{Y}_{t}a(t,\boldsymbol{U}_{t}^{\top}\boldsymbol{Y}_{t})^{\top}](I_{d\times d}-P_{\boldsymbol{U}_{t}}^{\mathrm{row}}),\label{eq:DLR-eq-U}\\
	d\boldsymbol{Y}_{t} & =\boldsymbol{U}_{t}a(t,\boldsymbol{U}_{t}^{\top}\boldsymbol{Y}_{t})\,dt+\boldsymbol{U}_{t}b(t,\boldsymbol{U}_{t}^{\top}\boldsymbol{Y}_{t})dW_{t}.\label{eq:DLR-eq-Y}
\end{align}
Here, we let  ${C}_{\boldsymbol{Y}_{t}}:=\mathbb{E}[\boldsymbol{Y}_{t}\boldsymbol{Y}_{t}{}^{\top}]$ 
whereas $P_{\boldsymbol{U}_{t}}^{\mathrm{row}} \in\mathbb{R}^{d\times d}$ denotes the projection-matrix onto the row space $\mathrm{span}\{U_{t}^{1},\dots,U_{t}^{R}\}\subset\mathbb{R}^{d}$
of $\boldsymbol{U}_{t}$; when $\boldsymbol{U}_{t}$ has orthonormal rows, one has $P_{\boldsymbol{U}_{t}}^{\mathrm{row}}= \boldsymbol{U}_{t}^{\top}\boldsymbol{U}_{t} $. 
In Section \ref{sec: Consistency}, we will give a rigorous justification for \eqref{eq:DLR-eq-U} and \eqref{eq:DLR-eq-Y}. 

Having stated these equations, we now define {our notion of} strong DO solution and DO approximation for an SDE problem of the type \eqref{eq:true-eq}.

\begin{defn}[Strong DO solution]\label{def: D0 sol}
A function $(\boldsymbol{U},\boldsymbol{Y}) : [0,T] \to \mathbb{R}^{R\times d} \times L^{2}(\Omega;\mathbb{R}^{R})$ is called a \textit{strong DO solution} for \eqref{eq:true-eq} if the following conditions are satisfied:
	\begin{enumerate}
		\item {the initial conditions $(\boldsymbol{U}_0,\boldsymbol{Y}_0 )$ are such that $\boldsymbol{U}_0  \in \mathbb{R}^{R\times d}$ is a matrix with orthonormal rows and $\boldsymbol{Y}_0  \in L^{2}(\Omega;\mathbb{R}^{R})$ has linearly independent components};
		\item the curve $t \to \boldsymbol{U}_t \in \mathbb{R}^{R\times d}$ is absolutely continuous on $[0,T]$ and $\boldsymbol{U}_t\dot{\boldsymbol{U}}^{\top}_t = 0 \in \mathbb{R}^{R \times R}$ for a.e. $t \in [0,T]$;
\item the curve $t \to \boldsymbol{Y}_t(\omega) \in  \mathbb{R}^{R}$ has almost surely continuous paths on $[0,T]$ and it is $\scrF_{t}$-measurable for all $t \in [0,T]$. Moreover,  for any $t \in [0,T]$ the components $Y^{1}_t,\dots,Y^{R}_t$ are linearly independent in $L^2(\Omega)$;
\item $\boldsymbol{U}$ satisfies equation \eqref{eq:DLR-eq-U} for a.e.~$t \in [0,T]$ and $\boldsymbol{Y}$ is a strong solution of \eqref{eq:DLR-eq-Y} 
on $[0,T]$.
	\end{enumerate}
\end{defn}
{We remark that points (1) and (2) in Definition \ref{def: D0 sol} imply that $\boldsymbol{U}$ has orthonormal rows for all $t \in [0,T]$, i.e.\  $\boldsymbol{U}_t\boldsymbol{U}^{\top}_t = I \in \mathbb{R}^{R \times R}$ for all $t \in [0,T]$. 
Furthermore, the term ``strong'' in Definition \ref{def: D0 sol} follows the standard terminology for the solution of SDEs: we seek a solution on the given probability space and Brownian motion. See for example \cite{karatzas2012brownian} for more details.}
For convenience, given a DO solution $(\boldsymbol{U},\boldsymbol{Y})$ we call the product $\boldsymbol{U}^{\top}\boldsymbol{Y}$ a DO approximation.
\begin{defn}[DO approximation]\label{def: DLR sol}
	We call a process $X : [0,T] \to L^{2}(\Omega;\mathbb{R}^{d})$ a \emph{DO approximation} of \eqref{eq:true-eq} if there exists a strong DO solution $(\boldsymbol{U},\boldsymbol{Y})$  such that $X_t := \boldsymbol{U}_t^{\top}\boldsymbol{Y}_t$ for all $t \in [0,T]$.
\end{defn}
Given a DO approximation $X$, the corresponding DO solution $(\boldsymbol{U},\boldsymbol{Y})$ is determined only up to a (process of) rotation matrix.
Indeed, let $(\boldsymbol{U},\boldsymbol{Y})$ and $(\tilde{\boldsymbol{U}}, \tilde{\boldsymbol{Y}})$ be two strong DO solutions such that  $\boldsymbol{U}_t^{\top}\boldsymbol{Y}_t = \tilde{\boldsymbol{U}}^{\top}_t\tilde{\boldsymbol{Y}}_t = X_t$  for all $t\ge0$. Then {the orthogonality of the rows }of $\boldsymbol{U}_t$ implies $\tilde{\boldsymbol{Y}}_t = O_t \boldsymbol{Y}_t$ with $O_t = \tilde{\boldsymbol{U}}_t\boldsymbol{U}_t^{\top}$. 
The matrix $O_t$ is orthogonal for every $t \in [0,T]$, but not necessarily an identity.
See Section~\ref{sec:equiv-DO-DLR}, in particular Proposition~\ref{prop: equiv of param}, for more details.
\subsection{Consistency of the DO equations}\label{sec: Consistency}
In this section, we rigorously show the consistency of the DO equations \eqref{eq:DLR-eq-U} and \eqref{eq:DLR-eq-Y} in the sense described hereafter. 
Assume that the exact solution of \eqref{eq:true-eq} is of the form $X=\sum_{j=1}^{R}U^{j}Y^{j}$
with deterministic function $\boldsymbol{U}_{t}=(U_{t}^{1},\dots,U_{t}^{R})^{\top}\in\mathbb{R}^{R\times d}$
and an Itô process $\boldsymbol{Y}_{t}(\omega)=(Y_{t}^{1}(\omega),\dots,Y_{t}^{R}(\omega))^{\top}\in[L^{2}(\Omega)]^{R}$
for some $R\le d$ that fulfil the following properties:
\begin{enumerate}
	\item the function $[0,T]\ni t\mapsto\boldsymbol{U}_{t}\in\mathbb{R}^{R\times d}$,
	$d\geq R$, is absolutely continuous on $[0,T]$ and satisfies $\boldsymbol{U}_{t}\dot{\boldsymbol{U}}_{t}^{\top}=0\in\mathbb{R}^{R\times R}$
	for almost every $t\in [0,T]$;
	moreover, $\boldsymbol{U}_{t}\boldsymbol{U}_{t}^{\top}=I\in\mathbb{R}^{R\times R}$
	for almost every $t \in [0,T]$, where $I$ is the identity matrix;
	\item $\boldsymbol{Y}=(\boldsymbol{Y}_{t})_{t\in[0,T]}$ is an Itô process
	\begin{equation}
		d\boldsymbol{Y}_{t}=\alpha_{t}dt+\beta_{t}dW_{t},\label{eq:anzats-dY}
	\end{equation}
	with coefficients $\alpha_{t}\in\mathbb{R}^{R}$ and $\beta_{t}\in\mathbb{R}^{R\times m}$,
	where $\alpha_{t}$ and $\beta_{t}$ are progressively measurable
	and have a continuous path almost surely.
\end{enumerate}
Then $(\boldsymbol{U}_t,\boldsymbol{Y}_t)$ must satisfy equations \eqref{eq:DLR-eq-U} and \eqref{eq:DLR-eq-Y}.

Indeed, since $\alpha_{t}$ and $\beta_{t}$ in \eqref{eq:anzats-dY}
are progressively measurable, by applying Itô's formula \cite{ito1951formula,kunita1967square}
we have
\begin{align}\label{eq:ito applied to trueeq}
	dX_{t} & =  (d\boldsymbol{U}_{t}^{\top})\boldsymbol{Y}_{t}+\boldsymbol{U}_{t}^{\top}d\boldsymbol{Y}_{t}+\sum_{j=1}^{R}d\langle U^{j},Y^{j}\rangle(t)\\
	& =  \dot{\boldsymbol{U}}_{t}^{\top}\boldsymbol{Y}_{t}\,dt+\boldsymbol{U}_{t}^{\top}\bigl(\alpha_{t}dt+\beta_{t}dW_{t}\bigr)+0\cdot\beta_{t}dt\\
	& =  \bigl(\dot{\boldsymbol{U}}_{t}^{\top}\boldsymbol{Y}_{t}+\boldsymbol{U}_{t}^{\top}\alpha_{t}\bigr)\,dt+\boldsymbol{U}_{t}^{\top}\beta_{t}\,dW_{t},
\end{align}
where $\langle U^{j},Y^{j}\rangle$, $j=1,\dots,R$ is the quadratic
covariation of $U^{j}$ and $Y^{j}$. Since $X$ is assumed to satisfy
\eqref{eq:true-eq}, the uniqueness of the representation of Itô processes
implies
\begin{align}
	\dot{\boldsymbol{U}}_{t}{}^{\top}\boldsymbol{Y}_{t}+\boldsymbol{U}_{t}^{\top}\alpha_{t} & =a(t,\boldsymbol{U}_{t}^{\top}\boldsymbol{Y}_{t})\label{eq:eq-bias}\\
	\boldsymbol{U}_{t}^{\top}\beta_{t} & =b(t,\boldsymbol{U}_{t}^{\top}\boldsymbol{Y}_{t}).\label{eq:eq-cov}
\end{align}
Now, using $\boldsymbol{U}_{t}\dot{\boldsymbol{U}}_{t}^{\top}=0\in\mathbb{R}^{R\times R}$
and $\boldsymbol{U}_{t}\boldsymbol{U}_{t}^{\top}=I\in\mathbb{R}^{R\times R}$ for almost every $t \in [0,T]$,
these equalities imply
\begin{equation}\label{eq: alpha and beta}
	\alpha_{t}=\boldsymbol{U}_{t}a(t,\boldsymbol{U}_{t}^{\top}\boldsymbol{Y}_{t})\quad\text{and}\quad\beta_{t}=\boldsymbol{U}_{t}b(t,\boldsymbol{U}_{t}^{\top}\boldsymbol{Y}_{t}).
\end{equation}
Moreover, $\boldsymbol{U}_t$ is absolutely continuous and by assumption $\boldsymbol{Y}_{t}$ is a.s.\ continuous. 
This implies that $\alpha_t$ and $\beta_t$ have a continuous path almost surely, and hence \eqref{eq:DLR-eq-Y} follows.
In turn, from \eqref{eq:eq-bias} we find
\begin{align*}
	{C}_{\boldsymbol{Y}_{t}}\dot{\boldsymbol{U}}_{t} & =\mathbb{E}[\boldsymbol{Y}_{t}a(t,\boldsymbol{U}_{t}^{\top}\boldsymbol{Y}_{t})^{\top}]-\mathbb{E}[\boldsymbol{Y}_t a(t,\boldsymbol{U}_{t}^{\top}\boldsymbol{Y}_{t})^{\top}]\boldsymbol{U}_{t}^{\top}\boldsymbol{U}_{t}\\
	& =\mathbb{E}[\boldsymbol{Y}_{t}a(t,\boldsymbol{U}_{t}^{\top}\boldsymbol{Y}_{t})^{\top}](I_{d\times d}-\boldsymbol{U}_{t}^{\top}\boldsymbol{U}_{t}),
\end{align*}
with ${C}_{\boldsymbol{Y}_{t}}:=\mathbb{E}[\boldsymbol{Y}_{t}\boldsymbol{Y}_{t}{}^{\top}]$.
Hence, from the orthogonality assumption {for the rows} of $\boldsymbol{U}$ we
have
\[
{C}_{\boldsymbol{Y}_{t}}\dot{\boldsymbol{U}}_{t}=\mathbb{E}[\boldsymbol{Y}_{t}a(t,\boldsymbol{U}_{t}^{\top}\boldsymbol{Y}_{t})^{\top}](I_{d\times d}-P_{\boldsymbol{U}_{t}}^{\mathrm{row}}).
\]
This completes our consistency argument.

\subsection{DO equations interpreted as a projected dynamics}\label{sec: proj dyn}

The DO equations \eqref{eq:DLR-eq-U}--\eqref{eq:DLR-eq-Y} are posed as a system of equations for the separate factors $\boldsymbol{U}$  and $\boldsymbol{Y}$.  
We now discuss what equation the DO approximation $X=\sum_{j=1}^{R}U^{j}Y^{j}$ should satisfy. 
In other words, we aim to derive an equation for $X$ in the ambient space,  independent of the parametrization $(\boldsymbol{U},\boldsymbol{Y})$.  

For this purpose, we substitute \eqref{eq:DLR-eq-U} and \eqref{eq:DLR-eq-Y} into \eqref{eq:ito applied to trueeq}; we obtain
\begin{equation}\label{eq:eq-manifold0}
	\begin{aligned}
		dX_{t} = &\bigl(\bigl(I_{d\times d}-P_{\boldsymbol{U}_{t}}^{\mathrm{row}}\bigr)\mathbb{E}[a(t,\boldsymbol{U}_{t}^{\top}\boldsymbol{Y}_{t})\boldsymbol{Y}_{t}^{\top}]{C}_{\boldsymbol{Y}_{t}}^{-1}\boldsymbol{Y}_{t}+P_{\boldsymbol{U}_{t}}^{\mathrm{row}}a(t,\boldsymbol{U}_{t}^{\top}\boldsymbol{Y}_{t})\bigr)\,dt \\
		&+P_{\boldsymbol{U}_{t}}^{\mathrm{row}}b(t,\boldsymbol{U}_{t}^{\top}\boldsymbol{Y}_{t})\,dW_{t}\\
		 =& \bigl(\bigl(I_{d\times d}-P_{\boldsymbol{U}_{t}}^{\mathrm{row}}\bigr)[P_{\boldsymbol{Y}_{t}}a(t,\boldsymbol{U}_{t}^{\top}\boldsymbol{Y}_{t})]+P_{\boldsymbol{U}_{t}}^{\mathrm{row}}a(t,\boldsymbol{U}_{t}^{\top}\boldsymbol{Y}_{t})\bigr)\,dt \\
		 &+P_{\boldsymbol{U}_{t}}^{\mathrm{row}}b(t,\boldsymbol{U}_{t}^{\top}\boldsymbol{Y}_{t})\,dW_{t},
	\end{aligned}
\end{equation}
\sloppy
where $P_{\boldsymbol{Y}_{t}}a(t,X_{t})\in\mathbb{R}^{d}$ is the
application of the $L^{2}(\Omega)$-orthogonal projection $P_{\boldsymbol{Y}_{t}}\colon L^{2}(\Omega)$ $\to\text{span}\{Y_{t}^{1},\dots,Y_{t}^{R}\}$
to each component of $a(t,X_{t})\in\mathbb{R}^{d}$. 
To derive a parameter-independent equation, we seek a parameter-free expression of the projections $P_{\boldsymbol{U}_{t}}$ and $P_{\boldsymbol{Y}_{t}}$.
\sloppy
Given $\mathcal{X}_{t}\in L^{2}(\Omega;\mathbb{R}^{d})$, $t\in[0,T]$,
let $\mathscr{P}_{\mathcal{U}(\mathcal{X}_t)}\colon\mathbb{R}^{d}\to\mathrm{Im}(\mathbb{E}[\mathcal{X}_t\,\cdot\,])\subset\mathbb{R}^{d}$
be the orthogonal projection matrix to the image of the mapping $\mathbb{E}[\mathcal{X}_t\,\cdot\,]\colon L^{2}(\Omega)\to\mathbb{R}^{d}$,
and let 
$\mathscr{P}_{\mathcal{Y}(\mathcal{X}_t)}\colon L^{2}(\Omega)\to\mathrm{Im}\bigl((\mathcal{X}^{\top}_t\cdot\,)\bigr)\subset L^{2}(\Omega)$
be the $L^{2}(\Omega)$-orthogonal projection to the image of the
mapping $\mathcal{X}^{\top}_t\colon\mathbb{R}^{d}\to L^{2}(\Omega)$.
In Lemma~\ref{lem:same-rank} we will show  $\mathrm{dim}(\mathrm{Im}\bigl((\mathcal{X}^{\top}_t\cdot\,))=\mathrm{dim}(\mathrm{Im}(\mathbb{E}[\mathcal{X}_t\,\cdot\,]))$, i.e.,  $\mathrm{Im}\bigl((\mathcal{X}^{\top}_t\cdot\,)$ is a finite-dimensional linear subspace of $L^{2}(\Omega)$, and thus $\mathscr{P}_{\mathcal{Y}(\mathcal{X}_t)}$ is well defined.

\sloppy
In the case of DO approximation $X_t = \boldsymbol{U}_{t}^{\top}\boldsymbol{Y}_{t}$, we have $\mathrm{Im}(\mathbb{E}[X_t\,\cdot\,]) = \mathrm{span}\{U^{1}_t, \dots, U^{R}_t\}$ and $\mathrm{Im}\bigl((X_t^{\top}\cdot\,)\bigr) = \mathrm{span}\{Y^1_{t}, \dots, Y^R_{t}\}$, we conclude that $\mathscr{P}_{\mathcal{U}(X_t)} v= \boldsymbol{U}_{t}^{\top}\boldsymbol{U}_{t} v =P_{\boldsymbol{U}_{t}}^{\mathrm{row}}v$, $v \in \mathbb{R}^d$ and $\mathscr{P}_{\mathcal{Y}(X_t)}w = \mathbb{E}[w \boldsymbol{Y}^{\top}_{t}]C^{-1}_{\boldsymbol{Y}_{t}}\boldsymbol{Y}_{t}=P_{\boldsymbol{Y}_{t}} w$, $w \in L^2(\Omega)$, so that \eqref{eq:eq-manifold0} can be rewritten as
\begin{equation}\label{eq:eq-manifold DO} 
	\begin{aligned}
	dX_{t} = &\bigl(\bigl(I_{d\times d}-\mathscr{P}_{\mathcal{U}( X_t)}\bigr)[\mathscr{P}_{\mathcal{Y}(X_t)}a(t, \boldsymbol{U}_{t}^{\top}\boldsymbol{Y}_{t})]+\mathscr{P}_{\mathcal{U}(X_t)}a(t, \boldsymbol{U}_{t}^{\top}\boldsymbol{Y}_{t})\bigr)\,dt\\
	&+\mathscr{P}_{\mathcal{U}(X_t)}b(t, \boldsymbol{U}_{t}^{\top}\boldsymbol{Y}_{t})\,dW_{t}.
 \end{aligned}
\end{equation}
This equation, derived from \eqref{eq:DLR-eq-U} and \eqref{eq:DLR-eq-Y}, does not depend on the parametrization of the pair $(\boldsymbol{U},\boldsymbol{Y})$ and could be taken as an alternative definition of DO approximation. 
More precisely, given any process $\mathcal{X} = (\mathcal{X}_t)_{t \in [0,T]}$ with $\mathcal{X}_{t}\in L^{2}(\Omega;\mathbb{R}^{d})$ for $t\in[0,T]$, we can define the following stochastic process:
\begin{equation}
	d\mathcal{\mathcal{X}}_{t}=\bigl(\bigl(I_{d\times d}-\mathscr{P}_{\mathcal{U}(\mathcal{X}_t)}\bigr)[\mathscr{P}_{\mathcal{Y}(\mathcal{X}_t)}a(t,\mathcal{\mathcal{X}}_{t})]+\mathscr{P}_{\mathcal{U}(\mathcal{X}_t)}a(t,\mathcal{X}_{t})\bigr)\,dt+\mathscr{P}_{\mathcal{U}(\mathcal{X}_t)}b(t,\mathcal{X}_{t})\,dW_{t}.
	\label{eq:eq-manifold DLR}
\end{equation}
It is worth noticing that \eqref{eq:eq-manifold DLR} is a McKean-Vlasov-type SDE since the evolution of $\mathcal{X}_{t}$ depends on the law of the process.

Finally, to speak of the \emph{rank} of the solution $\mathcal{X}$ to \eqref{eq:eq-manifold DLR} we note the following. 

Given $\tilde{\mathcal{X}}\in L^{2}(\Omega;\mathbb{R}^{d})$
the mapping $K_{\tilde{\mathcal{X}}}\colon\mathbb{R}^{d}\to L^{2}(\Omega)$
defined by $K_{\tilde{\mathcal{X}}}\boldsymbol{v}:=\tilde{\mathcal{X}}^{\top}\boldsymbol{v}$
for $\boldsymbol{v}\in\mathbb{R}^{d}$ is by definition finite rank.
Moreover, the operator $K_{\tilde{\mathcal{X}}}^{*}:=\mathbb{E}[\tilde{\mathcal{X}}\cdot]\colon L^{2}(\Omega)\to\mathbb{R}^{d}$
is the adjoint of $K_{\tilde{\mathcal{X}}}$:
\[
\mathbb{E}[y(K_{\tilde{\mathcal{X}}}\boldsymbol{v})]=\mathbb{E}[y(\tilde{\mathcal{X}}^{\top}\boldsymbol{v})]=\mathbb{E}[\boldsymbol{v}^{\top}\tilde{\mathcal{X}}y]=\boldsymbol{v}^{\top}K_{\tilde{\mathcal{X}}}^{*}y, \quad \forall y \in L^2(\Omega), \ \forall v \in \mathbb{R}^d
\]
The operator $K_{\tilde{\mathcal{X}}}^{*}K_{\tilde{\mathcal{X}}}\colon\mathbb{R}^{d}\to\mathbb{R}^{d}$
is given by $K_{\tilde{\mathcal{X}}}^{*}K_{\tilde{\mathcal{X}}}\boldsymbol{v}=\mathbb{E}[\tilde{\mathcal{X}}\tilde{\mathcal{X}}^{\top}]\boldsymbol{v}$
for $\boldsymbol{v}\in\mathbb{R}^{d}$. 
The following lemma characterises the rank of these  operators.
\begin{lem}
	\label{lem:same-rank}Given $\tilde{\mathcal{X}}\in L^{2}(\Omega;\mathbb{R}^{d})$,
	we have $\mathrm{rank}(K_{\tilde{\mathcal{X}}}^{*}K_{\tilde{\mathcal{X}}})=\mathrm{rank}(K_{\tilde{\mathcal{X}}})=\mathrm{rank}(K_{\tilde{\mathcal{X}}}^{*})$.
\end{lem}
\begin{proof}
	First, notice that we have $\mathrm{ker}(K_{\tilde{\mathcal{X}}}^{*})=\mathrm{Im}(K_{\tilde{\mathcal{X}}})^{\perp}$,
	where $\perp$ is the orthogonal complement with respect to the Euclidean
	inner product. Hence, we have $\mathrm{ker}(K_{\tilde{\mathcal{X}}}^{*}|_{\mathrm{Im}(K_{\tilde{\mathcal{X}}})})=\{0\}$.
	Thus, the rank-nullity theorem implies
	\[
	\mathrm{dim}(\mathrm{Im}(K_{\tilde{\mathcal{X}}}^{*}|_{\mathrm{Im}(K_{\tilde{\mathcal{X}}})}))=\mathrm{dim}(\mathrm{Im}(K_{\tilde{\mathcal{X}}})).
	\]
	Therefore $\mathrm{rank}(K_{\tilde{\mathcal{X}}}^{*}K_{\tilde{\mathcal{X}}})=\mathrm{rank}(K_{\tilde{\mathcal{X}}})$.
	Moreover, from $\mathrm{rank}(K_{\tilde{\mathcal{X}}})=\mathrm{rank}(K_{\tilde{\mathcal{X}}}^{*})$
	(see for example \cite[Theorem III.4.13]{Kato.T_book_1995_reprint}),
	the proof is complete.
\end{proof}
In view of the lemma above, we call  $\mathrm{dim}(\mathrm{Im}(\mathbb{E}[\mathcal{X}_{t}\,\cdot\,]))$, equivalently $\mathrm{dim}(\mathrm{Im}((\mathcal{X}_{t}^{\top}\cdot\,)))$ and $\mathrm{rank}(\mathbb{E}[\mathcal{X}_{t}\mathcal{X}_{t}^\top]])$, rank of $\mathcal{X}_{t}$.
\begin{defn}[DLR solution of rank $R$]\label{def:DLRA-rank-R}
	A process $\mathcal{X} : [0,T] \to L^{2}(\Omega;\mathbb{R}^{d})$ is called a \emph{DLR solution of rank} $R$  to \eqref{eq:true-eq} for an initial datum $\mathcal{X}_0 \in L^2(\Omega;\mathbb{R}^{d})$ if it satisfies \eqref{eq:eq-manifold DLR}, 
	$\mathrm{dim}(\mathrm{Im}(\mathbb{E}[\mathcal{X}_t\,\cdot\,])) = R$ for some $R\in\mathbb{N}$ for all $t\in[0,T]$, 
	and $\mathcal{X}$ has almost surely continuous paths.
\end{defn}

The parameter-independent formulation \eqref{eq:eq-manifold DLR} corresponds to the projected dynamics in the DLRA literature; see \cite{koch2007dynamical}, also \cite[Proposition2]{kazashi2021stability}. 
Notice however that in our formulation, only the projector $\mathscr{P}_{\mathcal{U}(\mathcal{X}_t)}$ is applied to the diffusion term $b(t,\mathcal{X}_{t})\,dW_{t}$, instead of the full projector 
\[
\mathscr{P}_{\mathcal{U}(\mathcal{X}_t)}+\mathscr{P}_{\mathcal{Y}(\mathcal{X}_t)} -\mathscr{P}_{\mathcal{U}(\mathcal{X}_t)}\mathscr{P}_{\mathcal{Y}(\mathcal{X}_t)}.
\]
If we naively wrote the projected dynamics following the standard DLRA approach, we would end up with the alternative formal expression
\begin{equation}
	d\mathcal{\mathcal{X}}_{t}=\left(\mathscr{P}_{\mathcal{U}(\mathcal{X}_t)}+\mathscr{P}_{\mathcal{Y}(\mathcal{X}_t)} -\mathscr{P}_{\mathcal{U}(\mathcal{X}_t)}\mathscr{P}_{\mathcal{Y}(\mathcal{X}_t)}\right)\left[a(t,\mathcal{\mathcal{X}}_{t})\,dt + b(t,\mathcal{X}_{t})\,dW_{t} \right],\label{eq:eq-manifold2}
\end{equation}
which would coincide with \eqref{eq:eq-manifold DLR} if one could prove that ``$\mathscr{P}_{\mathcal{Y}(\mathcal{X}_t)}[b(t,\mathcal{X}_{t})\,dW_{t}]=0$'' for a.e. $t >0$. However, it is not obvious how to give a rigorous meaning to the term $\mathscr{P}_{\mathcal{Y}(\mathcal{X}_t)}[b(t,\mathcal{X}_{t})\,dW_{t}]$ as an Itô integral and for this reason we do not pursue the formulation \eqref{eq:eq-manifold2} further. 

To confirm that \eqref{eq:eq-manifold DLR} is nevertheless a sensible projected dynamics, let us ask ourselves the following question: if the solution $\mathcal{\mathcal{X}}$ of  \eqref{eq:eq-manifold DLR} satisfies
\[\mathrm{dim}(\mathrm{Im}((\mathcal{X}_{t}^{\top}\cdot\,)))=\mathrm{dim}(\mathrm{Im}(\mathbb{E}[\mathcal{X}_{t}\,\cdot\,]))= R\text{ around }t,\]
is the right hand side of \eqref{eq:eq-manifold DLR}  consistent with this structure of $\mathcal{\mathcal{X}}$? In the following, we will see that the answer is affirmative, which supports the validity of \eqref{eq:eq-manifold DLR}, and thus \eqref{eq:DLR-eq-U} and \eqref{eq:DLR-eq-Y}, as a correct DLRA formulation for SDEs. 

In view of Lemma~\ref{lem:same-rank}, we study the rank of the matrix $\mathbb{E}[\mathcal{X}_{t}\mathcal{X}_{t}^{\top}]$,
where $\mathcal{X}_{t}$ is given by \eqref{eq:eq-manifold DLR}. To do
this, we will show that $\frac{d}{dt}\big(\mathbb{E}[\mathcal{X}_{t}\mathcal{X}_{t}^{\top}]\big)$
is in a tangent space of the rank-$R$ manifold, i.e., the manifold of $d \times d$ matrices with rank equal to $R$, at $\mathbb{E}[\mathcal{X}_{t}\mathcal{X}_{t}^{\top}]$.
For $j,k=1,\dots,d$, Itô's formula implies
\begin{align*}
	d(\mathcal{X}_{t}^{j}\mathcal{X}_{t}^{k})= & \mathcal{X}_{t}^{j}\bigl[\mu_{t}^{k}\,dt+\sum_{\ell=1}^{m}\Sigma_{t}^{k\ell}\,dW_{t}^{\ell}\bigr]\\
	& +\bigl[\mu_{t}^{j}\,dt+\sum_{\ell=1}^{m}\Sigma_{t}^{j
		\ell}\,dW_{t}^{\ell}\bigr]\mathcal{X}_{t}^{k}+\sum_{\ell=1}^{m}\Sigma_{t}^{j\ell}\Sigma_{t}^{k\ell}\,dt,
\end{align*}
where $\mu_{t}^{k}:=[(I_{d\times d}-\mathscr{P}_{\mathcal{U}(\mathcal{X}_t)})[\mathscr{P}_{\mathcal{Y}(\mathcal{X}_t)}a(t,\mathcal{X}_{t})]+\mathscr{P}_{\mathcal{U}(\mathcal{X}_t)}a(t,\mathcal{X}_{t})]_{k}$
and $\Sigma_{t}^{k\ell}:=[\mathscr{P}_{\mathcal{U}(\mathcal{X}_t)}b(t,\mathcal{X}_{t})]_{k\ell}$.
Thus, we have
\begin{align}\label{eq:covariance}
	\frac{d}{dt}\Bigl(\mathbb{E}[\mathcal{X}_{t}\mathcal{X}_{t}^{\top}]\Bigr)=\mathbb{E}\bigl[\mathcal{X}_{t}\mu_{t}^{\top}+\mu_{t}\mathcal{X}_{t}^{\top}+\Sigma_{t}\Sigma_{t}^{\top}\bigr]\quad\text{for almost every }t\in[0,T],
\end{align}
with $\mu_{t}:=[\mu_{t}^{k}]_{k=1,\dots,d}\in\mathbb{R}^{d}$ and
$\Sigma_{t}:=\big[\Sigma_{t}^{k\ell}\big]_{\substack{k=1,\dots,d \\\ell=1,\dots m}}\in\mathbb{R}^{d\times m}$.

On the other hand, from \cite[Proposition 2.1]{vandereycken2013low}
we know that the tangent space of a rank $R$ manifold at $\mathbb{E}[\mathcal{X}\mathcal{X}^{\top}]=Q\mathrm{diag}(\gamma_{1},\dots,\gamma_{R})Q^{\top}$
can be characterized as
\[
T_{\mathbb{E}[\mathcal{X}\mathcal{X}^{\top}]}\mathcal{M}:=\left\{ QV_{1}{}^{\top}+V_{2}Q^{\top}+QAQ^{\top}\,\left\vert \,\begin{array}{l}
	V_{\mathrm{1}}^{\top}Q=0,V_{\mathrm{1}}\in\mathbb{R}^{d\times R},\\
	V_{2}^{\top}Q=0,V_{\mathrm{2}}\in\mathbb{R}^{d\times R},\\
	\text{ and }A\in\mathbb{\mathbb{R}}^{d\times d}
\end{array}\right\} \right..
\]
To conclude  
$\frac{d}{dt}(\mathbb{E}[\mathcal{X}_{t}\mathcal{X}_{t}^{\top}])\in T_{\mathbb{E}[\mathcal{X}_t\mathcal{X}_t^{\top}]}\mathcal{M}$ we will use the following result. 
\begin{lem}
	\label{lem:SVD-finite-rank-op}Let the singular value decomposition
	$\mathbb{E}[\mathcal{X}\mathcal{X}^{\top}]=Q\mathrm{diag}(\gamma_{1},\dots,\gamma_{R})Q^{\top}$,
	with $\gamma_{1},\dots,\gamma_{R}>0$ be given, where $Q\in\mathbb{R}^{d\times R}$
	is a matrix consisting of $R$ orthogonal columns $\boldsymbol{q}_{1},\dots,\boldsymbol{q}_{R}\in\mathbb{R}^{d}$.
	Then, the canonical expansion of the finite rank operator $\mathbb{E}[\mathcal{X}\,\cdot\,]\colon L^{2}(\Omega)\to\mathbb{R}^{d}$
	is given by
	\[
	\mathbb{E}[\mathcal{X}y]=\sum_{k=1}^{R}\gamma_{k}^{1/2}\mathbb{E}\bigl[y\,\varphi_{k}\bigr]\boldsymbol{q}_{k}=Q\begin{bmatrix}\gamma_{1}^{1/2}\mathbb{E}[y\,\varphi_{1}]\\
		\vdots\\
		\gamma_{R}^{1/2}\mathbb{E}[y\,\varphi_{R}]
	\end{bmatrix},\text{ for }y\in L^{2}(\Omega),
	\]
	where $\varphi_{k}:=\gamma_{k}^{-1/2}\mathcal{X}^{\top}\boldsymbol{q}_{k}$,
	$k=1,\dots,R$ is an orthonormal basis of {$\mathrm{Im}\bigl((\mathcal{X}^{\top}\cdot\,)\bigr)\subset L^{2}(\Omega)$.}
\end{lem}

\begin{proof}
From
	\[
	\mathbb{E}[\mathcal{X}\mathcal{X}^{\top}\gamma_{k}^{-1}\boldsymbol{q}_{k}]
	=\gamma_{k}^{-1}Q\mathrm{diag}(\gamma_{1},\dots,\gamma_{R})Q^{\top}\boldsymbol{q}_{k}=\boldsymbol{q}_{k},
	\]
	we have $\{\boldsymbol{q}_{k}\}_{k=1}^{R}\subset\mathrm{Im}(\mathbb{E}[\mathcal{X}\,\cdot\,])$,
	and thus $\{\boldsymbol{q}_{k}\}_{k=1}^{R}$ is an orthonormal basis
	of $\mathrm{Im}(\mathbb{E}[\mathcal{X}\,\cdot\,])\subset\mathbb{R}^{d}$.

	Thus, with some coefficients $\{c_{k}\}_{k=1}^{R}\subset\mathbb{R}$
	we have for any $y \in L^2(\Omega)$ a representation $\sum_{k=1}^{R}c_{k}\boldsymbol{q}_{k}=\mathbb{E}[\mathcal{X}y]$,
	which implies
	\[
	c_{k}=\boldsymbol{q}_{k}^{\top}\mathbb{E}[\mathcal{X}y]=\mathbb{E}[\mathcal{X}^{\top}\boldsymbol{q}_{k}y]=\gamma_{k}^{1/2}\mathbb{E}[\gamma_{k}^{-1/2}\mathcal{X}^{\top}\boldsymbol{q}_{k}y] = \gamma_{k}^{1/2}\mathbb{E}[\varphi_{k}y].
	\]
	The functions $\varphi_{k}=\gamma_{k}^{-1/2}\mathcal{X}^{\top}\boldsymbol{q}_{k}$,
	$k=1,\dots,R$ are orthonormal in $L^{2}(\Omega)$, and thus form an
	orthonormal basis in the $R$-dimensional subspace $\mathrm{Im}\bigl((\mathcal{X}^{\top}\cdot\,)\bigr)$.
\end{proof}
{\begin{Remark}\label{rem: rank of span X}
Lemmata~\ref{lem:same-rank} and~\ref{lem:SVD-finite-rank-op}
imply $\mathrm{rank}\left(\mathcal{X}\right)=\mathrm{dim}(\mathrm{span}(\mathcal{X}))$,
where 
\[
\mathrm{span}(\mathcal{X}):=\mathrm{span}(\{\mathcal{X}(\omega)\in\mathbb{R}^{d}\mid\omega\in\Omega\})=\biggl\{\sum_{j=1}^{n}c_{j}\mathcal{X}(\omega_{j})\in\mathbb{R}^{d}\,\bigg|\,n\in\mathbb{N},\,c_{j}\in\mathbb{R},\,\omega_{j}\in\Omega\biggr\}\subset\mathbb{R}^{d},
\]
i.e.~the rank of $\mathcal{X}\in L^{2}(\Omega;\mathbb{R}^{d})$ as
in Definition~\ref{def:DLRA-rank-R} gives the dimension of the space spanned by all the
realizations of $\mathcal{X}$. To see this, take $\mathcal{X}\in L^{2}(\Omega;\mathbb{R}^{d})$
such that $\mathrm{dim}(\mathrm{Im}(\mathbb{E}[\mathcal{X}\,\cdot\,]))=R$.
Lemma~\ref{lem:same-rank} implies that $\mathbb{E}[\mathcal{X}\mathcal{X}^{\top}]=Q\mathrm{diag}(\gamma_{1},\dots,\gamma_{R})Q^{\top}$
holds, where $\gamma_{1},\dots,\gamma_{R}>0$ and $Q\in\mathbb{R}^{d\times R}$
is a matrix consisting of $R$ orthogonal columns $\boldsymbol{q}_{1},\dots,\boldsymbol{q}_{R}\in\mathbb{R}^{d}$,
while Lemma \ref{lem:SVD-finite-rank-op} tells us that $\varphi_{k}:=\gamma_{k}^{-1/2}\mathcal{X}^{\top}\boldsymbol{q}_{k}$,
for $k=1,\dots,R$, is an orthonormal basis of $\mathrm{Im}\bigl((\mathcal{X}^{\top}\cdot\,)\bigr)\subset L^{2}(\Omega)$.
Therefore, $\mathcal{X}$ admits the expansion
\begin{equation}
	\mathcal{X}(\omega)=\sum_{k=1}^{R}\mathbb{E}[\mathcal{X}\varphi_{k}]\varphi_{k}(\omega),\label{eq: exp of X tran}
\end{equation}
where $\mathbb{E}[\mathcal{X}\varphi_{k}]:=(\mathbb{E}[\mathcal{X}_{\ell}\varphi_{k}])_{\ell=1,\dots,d}$
is the component-wise $L^{2}(\Omega)$-inner product. From $\varphi_{k}=\gamma_{k}^{-1/2}\boldsymbol{q}_{k}^{\top}\mathcal{X}$
we have 
\begin{equation}
	\mathbb{E}[\mathcal{X}\varphi_{k}]=\gamma_{k}^{-1/2}\mathbb{E}[\mathcal{X}\mathcal{X}^{\top}\boldsymbol{q}_{k}]=\gamma_{k}^{-1/2}Q\mathrm{diag}(\gamma_{1},\dots,\gamma_{R})Q^{\top}\boldsymbol{q}_{k}=\gamma_{k}^{1/2}\boldsymbol{q}_{k}.\label{eq: exp varphi}
\end{equation}
By substituting \eqref{eq: exp varphi} into \eqref{eq: exp of X tran},
we obtain $
\mathcal{X}(\omega)=\sum_{k=1}^{R}\gamma_{k}^{1/2}\varphi_{k}(\omega)\boldsymbol{q}_{k}$ for any $\omega\in\Omega$,. This implies that $\sum_{j=1}^{n}c_{j}\mathcal{X}(\omega_j)= \sum_{k=1}^{R}\gamma_{k}^{1/2}c_{j}\varphi_{k}(\omega_j)\boldsymbol{q}_{k}$ for any $n\in\mathbb{N}$, $c_{j}\in\mathbb{R},$ and $\omega_{j}\in\Omega$, i.e.
the space spanned by all the realizations of $\mathcal{X}(\omega)$ is spanned by
$\boldsymbol{q}_{1},\dots,\boldsymbol{q}_{R}$. Therefore, $\mathrm{dim}(\mathrm{span}(\mathcal{X}))=R$.
\end{Remark}}

The equation \eqref{eq:covariance} together with Lemma \ref{lem:SVD-finite-rank-op} shows 
$\frac{d}{dt}(\mathbb{E}[\mathcal{X}_{t}\mathcal{X}_{t}^{\top}])\in T_{\mathbb{E}[\mathcal{X}_t\mathcal{X}_t^{\top}]}\mathcal{M}$. Indeed, the first term of the right hand side of \eqref{eq:covariance} can be written as 
\begin{align*}
	\mathbb{E}\bigl[\mathcal{X}_{t}\mu_{t}^{\top}\bigr] & =\mathbb{E}\bigl[\mathcal{X}_{t}(\mathscr{P}_{\mathcal{Y}(\mathcal{X}_t)}a(t,\mathcal{X}_{t}))^{\top}\bigr](I_{d\times d}-\mathscr{P}_{\mathcal{U}(\mathcal{X}_t)})+\mathbb{E}\bigl[\mathcal{X}_{t}a(t,\mathcal{X}_{t})^{\top}\bigr]\mathscr{P}_{\mathcal{U}(\mathcal{X}_t)}\\
	& =Q_{t}B_{t}(I_{d\times d}-Q_{t}Q_{t}^{\top})+Q_{t}C_{t}Q_{t}Q_{t}^{\top},
\end{align*}
where, using Lemma~\ref{lem:SVD-finite-rank-op}, we may write $Q_{t}B_{t}:=\mathbb{E}\bigl[\mathcal{X}_{t}(\mathscr{P}_{\mathcal{Y}(\mathcal{X}_t)}a(t,\mathcal{X}_{t}))^{\top}\bigr]$
and $Q_{t}C_{t}:=\mathbb{E}\bigl[\mathcal{X}_{t}a(t,\mathcal{X}_{t})^{\top}\bigr]$
with some $B_{t},C_{t}\in\mathbb{R}^{d\times d}$. Then, with $\tilde{V}_{t}:=(B_{t}(I_{d\times d}-Q_{t}Q_{t}^{\top}))^{\top}$
and $\tilde{A}_{t}:=C_{t}Q_{t}$ we have
\[
\mathbb{E}\bigl[\mathcal{X}_{t}\mu_{t}^{\top}\bigr]=Q_{t}\tilde{V}_{t}^{\top}+Q_{t}\tilde{A}_{t}Q_{t}^{\top}.
\]
Similarly, we have $\mathbb{E}[\mu_{t}\mathcal{X}_{t}^{\top}]=\tilde{V}_{t}Q_{t}^{\top}+Q_{t}\tilde{A}_{t}^{\top}Q_{t}^{\top}$
and $\mathbb{E}[\Sigma_{t}\Sigma_{t}^{\top}]=Q_{t}Q_{t}^{\top}\mathbb{E}[b(t,\mathcal{X}_{t})b(t,\mathcal{X}_{t})^{\top}]Q_{t}Q_{t}^{\top}$.
Hence, if $\mathbb{E}[\mathcal{X}_t\mathcal{X}_t^{\top}]$ is of rank $R$ around $t$, then its derivative \eqref{eq:covariance}, which was derived from the projected dynamics \eqref{eq:eq-manifold DLR}, lies indeed in 
$T_{\mathbb{E}[\mathcal{X}_t\mathcal{X}_t^{\top}]}\mathcal{M}$. 
This consistency supports the validity of the formulation \eqref{eq:eq-manifold DLR}, and in turn that of the DO equations \eqref{eq:DLR-eq-U} and \eqref{eq:DLR-eq-Y}.

\subsection{Equivalence of DO and DLR formulations} \label{sec:equiv-DO-DLR}
In the previous section, we showed that  if there exists a strong DO solution $(\boldsymbol{U}_t, \boldsymbol{Y}_t)$ to~\eqref{eq:true-eq}, then the corresponding DO approximation $X_t = \boldsymbol{U}^{\top}_t \boldsymbol{Y}_t$ satisfies  \eqref{eq:eq-manifold DLR}. We now investigate the reverse question: if $\mathcal{X}_t$ is a rank-$R$ solution of \eqref{eq:eq-manifold DLR} (DLR solution of rank $R$ of \eqref{eq:true-eq}) does there exist a DO solution $(\boldsymbol{U}, \boldsymbol{Y})$ such that $\mathcal{X} = \boldsymbol{U}^{\top}\boldsymbol{Y}$?

First, we need the following bound for the DLR solution.
\begin{lem} \label{lem: gronwall for X}
	Let a rank-$R$ DLR solution $\mathcal{X}_t$, $t \in [0,T]$, to \eqref{eq:true-eq} with $\mathcal{X}_0 \in L^2(\Omega;\mathbb{R}^{d})$ be given. For all $t \in [0,T]$, $\mathcal{X}_t$ satisfies
	\begin{equation}\label{eq: gronwall for X}
		\mathbb{E}[|\mathcal{X}_{t}|^2] \leq 3 \big( \mathbb{E}[|\mathcal{X}_0|^2] + C_{\mathrm{lgb}}T(T + 1) \big)\mathrm{exp}\big( 3C_{\mathrm{lgb}}T(T + 1) \big) =: M(T)
	\end{equation}
	\begin{proof}
		Taking the squared $L^2(\Omega)$-norm of $\mathcal{X}_{t}$ and using \eqref{eq:eq-manifold DLR} in integral form, together with Itô's isometry, Jensen's inequality, and Assumption~\ref{assump:lin-growth}, we have
		\begin{equation*}
			\begin{aligned}
				\mathbb{E}[|\mathcal{X}_{t}|^2]  =&\mathbb{E}[|\mathcal{X}_0 + \int^{t}_{0}\bigl(\bigl(I_{d\times d}-\mathscr{P}_{\mathcal{U}(\mathcal{X}_s)}\bigr)[\mathscr{P}_{\mathcal{Y}(\mathcal{X}_s)}a(s,\mathcal{X}_{s})]+\mathscr{P}_{\mathcal{U}(X_s)}a(s,\mathcal{X}_{s})\bigr)\,ds\\
				&+\mathscr{P}_{\mathcal{U}(\mathcal{X}_s)}b(s,\mathcal{X}_{s})\,dW_{s}|^2] \\
			\leq &3 \mathbb{E}[|\mathcal{X}_0|^2] + 3 \mathbb{E}[|\int^{t}_{0}\left(\bigl(I_{d\times d}-\mathscr{P}_{\mathcal{U}(\mathcal{X}_s)}\bigr)\mathscr{P}_{\mathcal{Y}(\mathcal{X}_s)}+ \mathscr{P}_{\mathcal{U}(\mathcal{X}_s)}\right)a(s,\mathcal{X}_{s})\,ds|^2] \\ & +3\mathbb{E}[|\int^{t}_{0}\mathscr{P}_{\mathcal{U}(\mathcal{X}_s)}b(s,\mathcal{X}_{s})\,dW_{s}|^2] \\
				  \leq & 3 \mathbb{E}[|\mathcal{X}_0|^2] + 3T\int^{t}_{0}\mathbb{E}[|\left(\bigl(I_{d\times d}-\mathscr{P}_{\mathcal{U}(\mathcal{X}_s)}\bigr)\mathscr{P}_{\mathcal{Y}(\mathcal{X}_s)}+\mathscr{P}_{\mathcal{U}(\mathcal{X}_s)}\right)a(s,\mathcal{X}_{s})\,|^2]ds \\ &+3\int^{t}_{0}\mathbb{E}[|\mathscr{P}_{\mathcal{U}(\mathcal{X}_s)}b(s,\mathcal{X}_{s})\,|^2]ds \\
				 \leq & 3 \mathbb{E}[|\mathcal{X}_0|^2] + 3C_{\mathrm{lgb}}(T + 1)\int^{t}_{0} (1+ \mathbb{E}[|\mathcal{X}_{s}|^2]) ds.
			\end{aligned}
		\end{equation*}
		Hence, Gronwall's lemma implies the statement.
	\end{proof}
\end{lem}
The following theorem gives the uniqueness of the DLR solution.
\begin{thm}\label{thm:uniqueness of DLR}
	Let $\mathcal{X}_0 \in L^2(\Omega;\mathbb{R}^{d})$ be such that $\mathrm{dim}(\mathrm{Im}(\mathbb{E}[\mathcal{X}_0\,\cdot\,])) = R$. Suppose that two DLR solutions $\mathcal{X}_t$ and $\mathcal{Z}_t$ of rank $R$ to \eqref{eq:true-eq} with initial datum $\mathcal{X}_0$ exist on $[0,T]$.
	Then $\mathcal{X}$ and $\mathcal{Z}$ are indistinguishable.	
	\begin{proof}
		The processes $\mathcal{X}_t$ and $\mathcal{Z}_t$ are assumed to satisfy
		\begin{equation*}
			d\mathcal{X}_{t}=\underbrace{\bigl(\bigl(I_{d\times d}-\mathscr{P}_{\mathcal{U}(\mathcal{X}_t)}\bigr)[\mathscr{P}_{\mathcal{Y}(\mathcal{X}_t)}a(t,\mathcal{\mathcal{X}}_{t})]+\mathscr{P}_{\mathcal{U}(\mathcal{X}_t)}a(t,\mathcal{X}_{t})\bigr)}_{=:\mathfrak{P}_{\mathcal{X}_t}a(t,\mathcal{\mathcal{X}}_{t})}\,dt+\mathscr{P}_{\mathcal{U}(\mathcal{X}_t)}b(t,\mathcal{X}_{t})\,dW_{t} 
		\end{equation*}
		and
		\begin{equation*}
			d\mathcal{Z}_{t}=\underbrace{\bigl(\bigl(I_{d\times d}-\mathscr{P}_{\mathcal{U}( \mathcal{Z}_t)}\bigr)[\mathscr{P}_{\mathcal{Y}(\mathcal{Z}_t)}a(t, \mathcal{Z}_{t})]+\mathscr{P}_{\mathcal{U}(\mathcal{Z}_t)}a(t, \mathcal{Z}_{t})\bigr)}_{=:\mathfrak{P}_{\mathcal{Z}_t}a(t,\mathcal{Z}_{t})}\,dt+\mathscr{P}_{\mathcal{U}(\mathcal{Z}_t)}b(t, \mathcal{Z}_{t})\,dW_{t},
		\end{equation*}
		respectively, for the same initial datum $\mathcal{X}_0$. This implies that
		\begin{equation}\label{eq: lip diffe DLR et DO}
			\begin{aligned}
				\mathbb{E}\left[\left|\mathcal{X}_{t}-\mathcal{Z}_{t}\right|^2\right] & \leq \mathbb{E}\Bigg[\Bigg| \int_{0}^{t}\bigl(\bigl(I_{d\times d}-\mathscr{P}_{\mathcal{U}(\mathcal{X}_s)}\bigr)\mathscr{P}_{\mathcal{Y}(\mathcal{X}_s)}+\mathscr{P}_{\mathcal{U}(\mathcal{X}_s)} \bigr) a(s,\mathcal{X}_{s}) \\
				& - \bigl(\bigl(I_{d\times d}-\mathscr{P}_{\mathcal{U}(\mathcal{Z}_s)}\bigr)\mathscr{P}_{\mathcal{Y}(\mathcal{Z}_s)}+\mathscr{P}_{\mathcal{U}(\mathcal{Z}_s)} \bigr) a(s, \mathcal{Z}_{s}) ds \\
				& + \int_{0}^{t} \mathscr{P}_{\mathcal{U}(\mathcal{X}_s)}b(s,\mathcal{X}_{s}) - \mathscr{P}_{\mathcal{U}(\mathcal{Z}_s)}b(s, \mathcal{Z}_{s}) dW_{s}\Bigg|^2\Bigg] \\
				& \leq 4\mathbb{E}\left[\left| \int_{0}^{t}\bigl(\mathfrak{P}_{\mathcal{X}_s}  -\mathfrak{P}_{\mathcal{Z}_s}\bigr)a(s,\mathcal{X}_{s})ds\right|^2\right] + 4\mathbb{E}\left[\left|\int_{0}^{t}\mathfrak{P}_{\mathcal{Z}_s} \bigl( a(s, \mathcal{Z}_{s}) - a(s,\mathcal{X}_{s})  \bigr) ds\right|^2\right] \\
				& + 4\mathbb{E}\left[\left|\int_{0}^{t}\bigl(\mathscr{P}_{\mathcal{U}(\mathcal{X}_s)} - \mathscr{P}_{\mathcal{U}(\mathcal{Z}_s)}\bigr)b(s,\mathcal{X}_{s})dW_{s}\right|^2\right] \\
				& + 4\mathbb{E}\left[\left|\int_{0}^{t}\mathscr{P}_{\mathcal{U}(\mathcal{Z}_s)}\bigl(b(s,\mathcal{X}_{s}) - b(s, \mathcal{Z}_{s})\bigr)dW_{s}\right|^2\right] \\
			\end{aligned}
		\end{equation}		

		{Denote by 
		$$\tilde{\gamma} := 
		\inf_{t \in [0,T]} \max \{\sigma_R \!\!\;\bigl(\mathbb{E}[\mathcal{X}_t\mathcal{X}_t^{\top}]\bigr), \sigma_R \!\!\;\bigl(\mathbb{E}[\mathcal{Z}_t\mathcal{Z}_t^{\top}]\bigr) \}$$
		 the infimum over time of the maximum between the $R$-th singular value of the matrices $\mathbb{E}[\mathcal{X}_t\mathcal{X}_t^{\top}]\in\mathbb{R}^{R\times R}$ and $\mathbb{E}[\mathcal{Z}_t\mathcal{Z}_t^{\top}]\in\mathbb{R}^{R\times R}$} , where for all $t$ we are considering the singular values in a decreasing order.  
		We have $\tilde{\gamma} >0$. 
		To see this, first note that, from Lemma~\ref{lem:same-rank}  the rank of $\mathbb{E}[\mathcal{X}_t\mathcal{X}_t^{\top}]$ is the same as the rank of $\mathcal{X}_t$. 
		But from the definition of DLR solution, $\mathcal{X}_t$ has always rank $R$, and moreover, the continuity of $\mathcal{X}_t$ implies continuity of $\sigma_R \!\!\;\bigl(\!\:\mathbb{E}[\mathcal{X}_t\mathcal{X}_t^{\top}])$ on $[0,T]$. Hence $\tilde{\gamma} >0$ follows; see also \cite[Lemma~2.1]{kazashi2021existence}. 
		{Then, Proposition~\ref{prop: lip of the proj} gives us the following Lipschitz bounds for the projectors for any $s \in [0,T]$:
		\begin{equation}
			\begin{aligned}
				&\mathbb{E}[|\bigl( \mathfrak{P}_{\mathcal{X}_s}-\mathfrak{P}_{\mathcal{Z}_s} \bigr) v|^2] 
				 \leq \left(\frac{2}{\tilde{\gamma}}\right)^2  \mathbb{E}[|\mathcal{X}_{s}-\mathcal{Z}_{s}|^2]\|v\|^2_{L^2(\Omega;\mathbb{R}^d)}  \quad \text{ for any }\ v \in L^2(\Omega; \mathbb{R}^d),\\
				&\|\mathscr{P}_{\mathcal{U}(\mathcal{X}_s)} - \mathscr{P}_{\mathcal{U}(\mathcal{Z}_s)}\|^2_{2}  \leq \frac{\mathbb{E}[|\mathcal{X}_{s}-\mathcal{Z}_{s}|^2]}{\tilde{\gamma}^2}.
			\end{aligned}
		\end{equation}
	Hence \eqref{eq: lip diffe DLR et DO}, It\^o's isometry and Jensen's inequality lead to
	\begin{equation}
		\begin{aligned}
			\mathbb{E}[|\mathcal{X}_{t}-\mathcal{Z}_{t}|^2] \leq & 4T \int_{0}^{t} \mathbb{E}[|\bigl(\mathfrak{P}_{\mathcal{X}_s}  - \mathfrak{P}_{\mathcal{Z}_s} \bigr) a(s,\mathcal{X}_{s})|^2] ds \\
			&+ 4T\int_{0}^{t}\mathbb{E}[|\mathfrak{P}_{\mathcal{Z}_s} \bigl( a(s, \mathcal{Z}_{s}) - a(s,\mathcal{X}_{s})  \bigr)|^2] ds \\
			& + 4\int_{0}^{t}\mathbb{E}[|\bigl(\mathscr{P}_{\mathcal{U}(\mathcal{X}_s)} - \mathscr{P}_{\mathcal{U}(\mathcal{Z}_s)}\bigr)b(s,\mathcal{X}_{s})|^2]ds \\
			& + 4\int_{0}^{t}\mathbb{E}[|\mathscr{P}_{\mathcal{U}(\mathcal{Z}_s)}\bigl(b(s,\mathcal{X}_{s}) - b(s, \mathcal{Z}_{s})\bigr)|^2] d{s}\\
			\leq & \overline{C} \int_{0}^{t} \mathbb{E}[|\mathcal{X}_{s}-\mathcal{Z}_{s}|^2] ds, \\    		 
		\end{aligned}
	\end{equation}
	for all $t$ in $[0,T]$ with $\overline{C}:=4 \left( \left(4T+1\right)\tilde{\gamma}^{-2} C_{\mathrm{lgb}}(1+M(T)) + C^2_{\mathrm{Lip}} (T+1) \right)$ and $M(T)$ as in \eqref{eq: gronwall for X}. Gronwall's lemma yields $\mathbb{E}[|\mathcal{X}_{t}-\mathcal{Z}_{t}|^2] = 0$ for all $t \in [0,T]$ and therefore 
	$(\mathcal{Z}_t)_{t \in [0,T]}$ and $(\mathcal{X}_t)_{t \in [0,T]}$ are versions of each other.	
	Moreover, because they both have a.s.~continuous paths, they are indistinguishable; see for example~\cite[Theorem 2]{Protter.P.E_2005_book_StochasticIntegral_2nd}.
	}
	\end{proof}
\end{thm}

Given a strong DO solution $(\boldsymbol{U},\boldsymbol{Y})$ to \eqref{eq:true-eq}, the corresponding DO approximation $X_t$ satisfies the equation  \eqref{eq:eq-manifold DO}, and thus \eqref{eq:eq-manifold DLR}. Hence, Theorem~\ref{thm:uniqueness of DLR} gives the sought equivalence between the DLR equation~\eqref{eq:eq-manifold DLR} and the DO equation~\eqref{eq:DLR-eq-U}--\eqref{eq:DLR-eq-Y}, which we state as a corollary.
\begin{cor}\label{cor:equivalence DO and DLR}
	Let $\mathcal{X}_0 \in L^2(\Omega;\mathbb{R}^{d})$ with $\mathrm{dim}(\mathrm{Im}(\mathbb{E}[\mathcal{X}_0\,\cdot\,])) = R$. Suppose that a DLR solution $\mathcal{X}_t$ of rank $R$ to \eqref{eq:true-eq} and a strong DO solution $(\boldsymbol{U},\boldsymbol{Y})$ exist on $[0,T]$, both with initial datum $\mathcal{X}_0$. Then, $\mathcal{X}_t$ and the DO approximation $X_t = \boldsymbol{U}^{\top}_t\boldsymbol{Y}_t$ are indistinguishable on $[0,T]$.
\end{cor}
Moreover, the DO solution giving the same DLR solution is unique up to a rotation matrix. See also \cite[Section 2.2]{kazashi2021existence} for a similar result.
\begin{prop}
	\label{prop: equiv of param} 
	Assume that a DLR solution $\mathcal{X}_{t}$
	of rank $R$ to \eqref{eq:true-eq} with initial datum $\mathcal{X}_{0}\in[L^{2}(\Omega)]^{d}$
	exists for all $t\in[0,T]$. Suppose there exist two strong DO solutions
	$(\boldsymbol{U}_{t},\boldsymbol{Y}_{t})$ and $(\boldsymbol{V}_{t},\boldsymbol{Z}_{t})$.
	Then there exists a unique orthogonal matrix $\Theta\in\mathbb{R}^{R\times R}$
	such that  $(\boldsymbol{V}_{t},\boldsymbol{Z}_{t})=(\Theta\boldsymbol{U}_{t},\Theta\boldsymbol{Y}_{t})$. 
\end{prop}
\begin{proof}
	The proof follows closely the arguments in \cite[Lemma 2.3, Corollary 2.4, and Lemma 2.5]{kazashi2021existence}.
	We will show that there exists a unique absolutely continuous curve
	$t\to\Theta(t)$ with orthogonal matrix $\Theta(t)\in\mathbb{R}^{R\times R}$
	for all $t$ such that $(\boldsymbol{V}_{t},\boldsymbol{Z}_{t})=(\Theta(t)\boldsymbol{U}_{t},\Theta(t)\boldsymbol{Y}_{t})$,
	and that such $\Theta(t)$ is a constant in $t$. First, we will derive
	an equation that $\Theta(t)$ must satisfy. Suppose that sought $\Theta(t)$
	exists. Note that, since both $\boldsymbol{U}_{t}^{\top}\boldsymbol{Y}_{t}$
	and $\boldsymbol{V}_{t}^{\top}\boldsymbol{Z}_{t}$ satisfy \eqref{eq:eq-manifold DO},
	Theorem~\ref{thm:uniqueness of DLR} implies $\mathcal{X}_{t}=\boldsymbol{U}_{t}^{\top}\boldsymbol{Y}_{t}=\boldsymbol{V}_{t}^{\top}\boldsymbol{Z}_{t}$
	for all $t$ and a.s. Then since $(\boldsymbol{V}_{t},\boldsymbol{Z}_{t})$
	is a strong DO solution, one must have 
	\begin{equation}
		\dot{\boldsymbol{U}}_{t}=\dot{(\Theta(t)\boldsymbol{V}_{t})}=\dot{\Theta}^{\top}(t)\boldsymbol{V}_{t}+\Theta^{\top}(t)\dot{\boldsymbol{V}}_{t}\quad\text{ for a.e. }\quad t\in[0,T].\label{eq:der-UThetaV}
	\end{equation}
	As $\boldsymbol{U}_{t}\dot{\boldsymbol{U}}_{t}^{\top}=0$ and $\boldsymbol{U}_{t}\boldsymbol{U}_{t}^{\top}=I_{R\times R}$,
	from $\boldsymbol{U}_{t}=\Theta^{\top}(t)\boldsymbol{V}_{t}$ it follows
	\begin{equation}
		0=\Theta^{\top}(t)\boldsymbol{V}_{t}\left(\dot{\Theta}^{\top}(t)\boldsymbol{V}_{t}+\Theta^{\top}(t)\dot{\boldsymbol{V}}_{t}\right)^{\top}=\Theta^{\top}(t)\left(\dot{\Theta}(t)+\boldsymbol{V}_{t}\dot{\boldsymbol{V}}_{t}^{\top}\Theta(t)\right).
	\end{equation}
	Using orthogonality of $\Theta(t)$ and $\boldsymbol{V}_{t}\dot{\boldsymbol{V}}_{t}^{\top}=0$,
	we obtain the following differential equation that $\Theta(t)$ has
	to satisfy: 
	\begin{equation}
		\dot{\Theta}(t)=0\ \ \text{ for a.e.~\ensuremath{t\in[0,T]} with }\Theta(0)=\Theta_{*},\label{eq: theta eq}
	\end{equation}
	where $\Theta_{*}$ is an orthogonal matrix. But this equation has
	a unique solution $\Theta(t)\equiv\Theta_{*}$, which is an orthogonal
	matrix. Hence, going back the argument above, $\Theta$ satisfies
	\eqref{eq:der-UThetaV}. The absolute continuity of $\boldsymbol{U}$
	and $\Theta\boldsymbol{V}$ yields $\boldsymbol{U}_{t}=\Theta(t)\boldsymbol{V}_{t}+C$
	for all $t\in[0,T]$, with some matrix $C\in\mathbb{R}^{R\times R}$.

	Such $\Theta_{*}$ that makes $C=0$ can be constructed explicitly.
	Since $\boldsymbol{U}_{t}$, $\boldsymbol{V}_{t}$ are deterministic,
	from $\boldsymbol{U}_{0}^{\top}\boldsymbol{Y}_{0}=\boldsymbol{V}_{0}^{\top}\boldsymbol{Z}_{0}$ we
	must have 
	\[
	\boldsymbol{U}_{0}^{\top}=\boldsymbol{V}_{0}^{\top}\mathbb{E}[\boldsymbol{Z}_{0}\boldsymbol{Y}_{0}^{T}]{C}_{\boldsymbol{Y}_{0}}^{-1}.
	\]
	Take $\Theta_{*}=\mathbb{E}[\boldsymbol{Z}_{0}\boldsymbol{Y}_{0}^{T}]{C}_{\boldsymbol{Y}_{0}}^{-1}$. Then, from the argument above $\boldsymbol{U}_{t}=\Theta\boldsymbol{V}_{t}$
	holds for all $t\in[0,T]$. We conclude the proof by noting the orthogonality
	of $\Theta_{*}$: 
	\[
	I_{R\times R}=\boldsymbol{U}_{0}\boldsymbol{U}_{0}^{\top}=\Theta_{*}^{\top}\boldsymbol{V}_{0}\boldsymbol{V}_{0}^{\top}\Theta_{*}=\Theta_{*}^{\top}\Theta_{*}.
	\]
\end{proof}

\begin{remark}\label{rmk: equivalence DO and DLR}
In Section \ref{sec: well-posed}, under Assumptions~1--3 we will show that there exists a unique strong DO solution $(\boldsymbol{U},\boldsymbol{Y})$ in a certain interval $[0,T]$. Hence, in view of Theorem~\ref{thm:uniqueness of DLR} and Corollary~\ref{cor:equivalence DO and DLR},  a unique DLR solution $\mathcal{X}_t$ exists on $[0,T]$ and moreover, given $\mathcal{X}_t$, we can always find a strong DO solution $(\boldsymbol{U},\boldsymbol{Y})$ such that $\mathcal{X}_t = \boldsymbol{U}_t^{\top}\boldsymbol{Y}_t$.
\end{remark}
\section{Local existence and uniqueness}\label{sec: well-posed}

The equations \eqref{eq:DLR-eq-U} and \eqref{eq:DLR-eq-Y} define
a non-standard system of stochastic differential equations. Notice
that in \eqref{eq:DLR-eq-U}, the matrix-valued function $[L^{2}(\Omega)]^{R}\ni\boldsymbol{Y}_{t}\mapsto{C}_{\boldsymbol{Y}_{t}}^{-1}\in\mathbb{R}^{R\times R}$
is not defined everywhere in $[L^{2}(\Omega)]^{R}$. Moreover, the
vector field $\mathbb{E}[\boldsymbol{Y}_{t}\,a(t,X_{t})^{\top}]$
requires taking the expectation, and thus depends on the knowledge
of all the paths of $\boldsymbol{Y}_{t}$. 
Hence, the vector fields that define the DO equation are not defined path-wise a priori. This setting makes the existence
and uniqueness result of DLR solutions non-trivial.

To establish existence
and uniqueness of solutions of \eqref{eq:DLR-eq-U} and \eqref{eq:DLR-eq-Y} given an initial datum {$(\boldsymbol{\varphi}, \boldsymbol{\xi})$}, we follow a fixed-point argument. 
We will define a sequence of Picard iterates, which belongs to a specific set of functions, where the first element of the sequence is made of a pair of constant-in-time functions {$(\boldsymbol{\varphi}, \boldsymbol{\xi})$}. 
Then we will show that this sequence converges in this set.  

Let $\boldsymbol{U}\in C([0,t];\mathbb{R}^{R\times d})$ be such that the rows $U^{1}_s, \dots, U^{R}_s\in\mathbb{R}^d$ of $\boldsymbol{U}_s$ are linearly independent for every $s \in [0,t]$ {(but not necessarily orthonormal)}, and let
$\boldsymbol{Y}\in L^{2}(\Omega;C([0,t];\mathbb{R}^{R}))$ be an $(\mathscr{F}_{t})$-adapted
process such that $\boldsymbol{Y}_{s}=(Y^{1}(s),\dots,Y^{R}(s))\in[L^{2}(\Omega)]^{R}$ has linearly independent components for every $s\in[0,t]$. For $\boldsymbol{\varphi}\in\mathbb{R}^{R\times d}$
with $d\geq R$ and $\boldsymbol{\xi}\in[L^{2}(\Omega)]^{R}$,
define
\begin{equation}
F_{1}(\boldsymbol{U},\boldsymbol{Y})(t):=\boldsymbol{\varphi}+\int_{0}^{t}{C}_{\boldsymbol{Y}_{s}}^{-1}\mathbb{E}[\boldsymbol{Y}_{s}a(s,\boldsymbol{U}_{s}^{\top}\boldsymbol{Y}_{s})^{\top}](I_{d\times d}-P_{\boldsymbol{U}_{s}}^{\mathrm{row}})\,ds\in\mathbb{R}^{R\times d},\label{eq:F1}
\end{equation}
\begin{equation}
F_{2}(\boldsymbol{U},\boldsymbol{Y})(t):=\boldsymbol{\xi}+\int_{0}^{t}\boldsymbol{U}_{s}a(s,\boldsymbol{U}_{s}^{\top}\boldsymbol{Y}_{s})\,ds+\int_{0}^{t}\boldsymbol{U}_{s}b(s,\boldsymbol{U}_{s}^{\top}\boldsymbol{Y}_{s})dW_{s}\in[L^{2}(\Omega)]^{R},\label{eq:F2}
\end{equation}
where we recall that $P_{\boldsymbol{U}_{s}}^{\mathrm{row}}$ is the projection-matrix onto the row space $\mathrm{span}\{U_{s}^{1},\dots,U_{s}^{R}\}\subset\mathbb{R}^{d}$ of $\boldsymbol{U}_{s}$. Note that the stochastic integral $\int_{0}^{t}\boldsymbol{U}_{s}b(s,\boldsymbol{U}_{s}^{\top}\boldsymbol{Y}_{s})dW_{s}$
is well defined, since \sloppy $\boldsymbol{Y}\in L^{2}(\Omega;C([0,t];\mathbb{R}^{R}))$
is (indistinguishable from) a progressively measurable process.

We will construct a unique fixed point of $F_{1}$ and $F_{2}$.
Because of the aforementioned difficulties, defining a suitable sequence
of Picard iterates requires some care. Let us consider $\boldsymbol{\varphi}\in\mathbb{R}^{R\times d}$
with $d\geq R$ having orthogonal row vectors, and $\boldsymbol{\xi}=(\xi_{1},\dots,\xi_{R})\in [L^{2}(\Omega)]^R$,
having linearly independent components, $\mathcal{F}_{0}$-measurable with $\rho^2:=\|\boldsymbol{\xi}\|_{[L^{2}(\Omega)]^{R}}^{2}$ and $\gamma :=\|{C}_{\boldsymbol{\xi}}^{-1}\|_{\mathrm{F}}$. 
We force iterations to belong to balls in $\mathbb{R}^{R\times d}$ and $[L^{2}(\Omega)]^{R}$ around $\boldsymbol{\varphi}$ and $\boldsymbol{\xi}$, respectively, of a suitable radius $\eta$.

For the ball in $\mathbb{R}^{R\times d}$, to invoke Proposition~\ref{prop: A inverse} in the appendix, we equip $\mathbb{R}^{R\times d}=[\mathbb{R}^d]^R$ with the norm $\|\boldsymbol{U}\|_{[\mathbb{R}^d]^R}^2=\sum_{j=1}^{R}|U^j|^2$. For $\boldsymbol{\varphi}=(\varphi^1,\dots,\varphi^R)^\top$ orthogonal, we have $\|\boldsymbol{\varphi}\|_{[\mathbb{R}^d]^R}=\sqrt{R}$, and  \[Z_{\boldsymbol{\varphi}}:=(\varphi^j(\varphi^k)^\top)_{j,k=1,\dots,R}
=\boldsymbol{\varphi}\boldsymbol{\varphi}^\top =
I_{R\times R},
\] so
with $\eta_1:=\eta(\sqrt{R},\sqrt{R})$ as in \eqref{eq:def-eta}, 
$\boldsymbol{v}\in B_{\eta_1}(\boldsymbol{\varphi})$ implies $\|(\boldsymbol{v}\boldsymbol{v}^\top)^{-1}\|_{\mathrm{F}}\leq 2{\sqrt{R}}$. Hence, the projection $(I_{d\times d}-P_{\boldsymbol{U}_{t}}^{\mathrm{row}})$ 
is Lipschitz continuous on $B_{\eta_1}(\boldsymbol{\varphi})$; see Lemma~\ref{lem: lip orth proj}.

For the ball in $[L^{2}(\Omega)]^{R}$, first note that, since 
$[L^{2}(\Omega)]^{R}\ni \boldsymbol{Z}\mapsto\mathbb{E}[\boldsymbol{Z}\boldsymbol{Z}^{\top}]^{-1}\in \mathbb{R}^{R\times R}$ 
is continuous in the open set $\Gamma = \{\boldsymbol{Z}\in [L^2(\Omega)]^R\mid   \mathrm{det}(\mathbb{E}[\boldsymbol{Z}\boldsymbol{Z}^{\top}]) \neq 0\}$ (cf.~\cite[Proof of Lemma~3.5]{kazashi2021existence}), 
and 
$\boldsymbol{\xi} \in \Gamma$, via Proposition~\ref{prop: A inverse} there exists a ball $B_{\eta_2}(\boldsymbol{\xi})$ in $[L^{2}(\Omega)]^{R}$
around $\boldsymbol{\xi}$ with radius $\eta_2:= \eta (\rho, \gamma)>0$ such that any $\boldsymbol{w} \in B_{\eta_2}(\boldsymbol{\xi})$ has linearly independent components and $\|{C}_{\boldsymbol{w}}^{-1}\|_{\mathrm{F}}\leq2\gamma$. 
Abusing the notation slightly, we let
\begin{equation}\label{eq: eta}
\eta=\eta(R,\rho,\gamma):=\min\{\eta(\sqrt{R},\sqrt{R}),\eta(\rho,\gamma)\}.
\end{equation}Proposition \ref{prop: A inverse} tells us that $\eta$ is non-increasing in both variables $\rho$ and $\gamma$.

We want to define sequences $\boldsymbol{Y}_{t}^{(n)}\in B_{\eta}(\boldsymbol{\xi})$
and $\boldsymbol{U}_{t}^{(n)}\in B_{\eta}(\boldsymbol{\varphi})$
for $t\in[0,\delta]$ with a suitable $\delta=\delta(\boldsymbol{\varphi},\boldsymbol{\xi})>0$.
To this extent, let $(\boldsymbol{U}_{t}^{(0)},\boldsymbol{Y}_{t}^{(0)}):=(\boldsymbol{\varphi},\boldsymbol{\xi})$,
$\boldsymbol{U}_{t}^{(n+1)}:=F_{1}(\boldsymbol{U}^{(n)},\boldsymbol{Y}^{(n)})(t)$,
$\boldsymbol{Y}_{t}^{(n+1)}:=F_{2}(\boldsymbol{U}^{(n)},\boldsymbol{Y}^{(n)})(t)$,
and $X_{t}^{(n)}:=\bigl(\boldsymbol{U}_{t}^{(n)}\bigr)^{\top}\boldsymbol{Y}_{t}^{(n)}$,
$n=0,1,\dots$ for $t\in[0,\delta]$ with
{
\begin{equation}
\delta:=\min\{1,\frac{
	\min\{1,
\eta^2 
	\}}
{36RC_{\mathrm{lgb}}(1+3R(3\rho^2 +1))},
\frac{
	\min\{
	\eta^2
,R\}
}{8\gamma^2(3\rho^2+1) C_{\mathrm{lgb}}(1+3R(3\rho^2 +1))}\}.\label{eq:def-delta}
\end{equation}
}
Moreover, let 
\[
\mathbb{D}_{\mathrm{det}}:=\left\{ \boldsymbol{V} \in C([0,\delta];\mathbb{R}^{R\times d})\,\left\vert \,\begin{array}{l}
\sup\limits_{0\leq t\leq\delta}\|\boldsymbol{V}_{t}\|_{\mathrm{F}}^{2}\leq3R\text{, and}\\
\boldsymbol{V}_{t}\in B_{\eta}(\boldsymbol{\varphi})\text{ for }t\in[0,\delta]
\end{array}\right\} \right.
\]
and
\[
\mathbb{D}_{\mathrm{sto}}:=\left\{ \boldsymbol{Z} \in L^{2}(\Omega;C([0,\delta];\mathbb{R}^{R}))\,\left\vert \,\begin{array}{l}
\ensuremath{\boldsymbol{Z}}\text{ is \ensuremath{\mathcal{F}_{t}}-adapted, }\\
\mathbb{E}\bigl[\sup\limits_{0\leq t\leq\delta}|\boldsymbol{Z}_{t}|^{2}\bigr]\leq 3\rho^2+1\text{, and}\\
\boldsymbol{Z}_{t}\in B_{\eta}(\boldsymbol{\xi})\text{ for }t\in[0,\delta]
\end{array}\right\} \right. .
\]
The following lemma shows that our Picard sequence takes value in $\mathbb{D}_{\mathrm{det}}\times\mathbb{D}_{\mathrm{sto}}$.
\begin{lem}
\label{lem:sq-in-D} Under Assumptions 2--3, the sequence $\bigl((\boldsymbol{U}^{(n)},\boldsymbol{Y}^{(n)})\bigr)_{n\geq0}$ defined above satisfies
\[
(\boldsymbol{U}^{(n)},\boldsymbol{Y}^{(n)})\in\mathbb{D}_{\mathrm{det}}\times\mathbb{D}_{\mathrm{sto}} \quad 
\text{ for all } \quad n\in \mathbb{N},
\]
where $\delta=\delta(C_{\mathrm{lgb}},d,\eta,R,\rho)$ is defined
in \eqref{eq:def-delta}.
\end{lem}

\begin{proof}
{We have $\|\boldsymbol{\varphi}\|_{\mathrm{F}}^{2}=R$, $\|\boldsymbol{\xi}\|_{[L^{2}(\Omega)]^{R}}^{2}=\rho^2$, $\|{C}_{\boldsymbol{\xi}}^{-1}\|_{\mathrm{F}}=\gamma$,
so, trivially, $\boldsymbol{\xi}\in B_{\eta}(\boldsymbol{\xi})$ and $\boldsymbol{\varphi}\in B_{\eta}(\boldsymbol{\varphi})$, where $\eta=\eta(R,\rho,\gamma)$ is built as in \eqref{eq: eta}}. Moreover,
$\boldsymbol{\xi}$ is $\mathcal{F}_{t}$-adapted thanks to Assumption 3. Thus, $(\boldsymbol{U}^{(0)},\boldsymbol{Y}^{(0)})\in\mathbb{D}_{\mathrm{det}}\times\mathbb{D}_{\mathrm{sto}}$.
Assume $(\boldsymbol{U}^{(n)},\boldsymbol{Y}^{(n)})\in\mathbb{D}_{\mathrm{det}}\times\mathbb{D}_{\mathrm{sto}}$
for $n\in\mathbb{N}$. Then, from Assumption 2, we see that $\boldsymbol{U}^{(n+1)}$
and $\boldsymbol{Y}^{(n+1)}$ are well defined, and that $\boldsymbol{Y}^{(n+1)}$
is $\mathcal{F}_{t}$-adapted. 
Moreover, $\mathbb{E}\bigl[\sup\limits_{0\leq t\leq\delta}|\boldsymbol{Y}_{t}^{(n)}|^{2}\bigr]<\infty$ and the inequality $\|A B\|_F \leq \|A\|_2\|B\|_F$ for $A \in \mathbb{R}^{n \times m}$, $B \in \mathbb{R}^{m \times p}$ implies
\begin{align*}
	\mathbb{E}\Bigl[ & \sup_{0\leq t\leq\delta}|\boldsymbol{Y}_{t}^{(n+1)}|^{2}\Bigr]\\
	& \leq3\mathbb{E}\bigl[|\boldsymbol{\xi}|^{2}\bigr]+3\mathbb{E}\biggl[\sup_{0\leq t\leq\delta}t\!\int_{0}^{t}\|\boldsymbol{U}_{s}^{(n)}\|_{2}^{2}|a(s,X_{s}^{(n)})|^{2}\,ds+\sup_{0\leq t\leq\delta}\Bigl|\int_{0}^{t}\boldsymbol{U}_{s}^{(n)}b(s,X_{s}^{(n)})dW_{s}\Bigr|^{2}\biggr]\\
	& {\leq3\mathbb{E}\bigl[|\boldsymbol{\xi}|^{2}\bigr]+3\mathbb{E}\biggl[\sup_{0\leq t\leq\delta}t\!\int_{0}^{t}\|\boldsymbol{U}_{s}^{(n)}\|_{2}^{2}|a(s,X_{s}^{(n)})|^{2}\,ds\biggr]+12\mathbb{E}\biggl[\Bigl|\int_{0}^{\delta}\boldsymbol{U}_{s}^{(n)}b(s,X_{s}^{(n)})dW_{s}\Bigr|^{2}\biggr]}\\
	& \leq3\rho^2+3\mathbb{E}\biggl[3\delta R\int_{0}^{\delta}|a(s,X_{s}^{(n)})|^{2}\,ds\biggr]+12\mathbb{E}\Bigl[\int_{0}^{\delta}\|\boldsymbol{U}_{s}^{(n)}b(s,X_{s}^{(n)})\|_{\mathrm{F}}^{2}\,ds\Bigr]\\
	& \leq3\rho^2+36 RC_{\mathrm{lgb}}(1+\mathbb{E}[\sup_{0\leq s\leq\delta}|X_{s}^{(n)}|^{2}])\delta\\
	& \leq3\rho^2+36RC_{\mathrm{lgb}}(1+3R(3\rho^2+1))\delta\le 3\rho^2+1,
\end{align*}
where {in the second and in the third lines we employ the 
	Doob's martingale inequality and Itô's isometry, respectively,} in the penultimate line we used $\delta\leq 1$, and in the last line  $|a(s,x)|^{2}+\|b(s,x)\|_{\mathrm{F}}^{2}\leq C_{\mathrm{lgb}}(1+|x|^{2})$
together with the definition of $\delta$. Similarly, we have
\begin{align*}
	\mathbb{E}\Bigl[\sup_{0\leq t\leq\delta}|\boldsymbol{Y}_{t}^{(n+1)}-\boldsymbol{\xi}|^{2}\Bigr] & \leq2\mathbb{E}\biggl[\sup_{0\leq t\leq\delta}t\int_{0}^{t}\|\boldsymbol{U}_{s}^{(n)}\|_{2}^{2}|a(s,X_{s}^{(n)})|^{2}\,ds+\sup_{0\leq t\leq\delta}\Bigl|\int_{0}^{t}\boldsymbol{U}_{s}^{(n)}b(s,X_{s}^{(n)})dW_{s}\Bigr|^{2}\biggr]\\
	& \leq 24C_{\mathrm{lgb}}R(1+3R(3\rho^2+1))\delta \leq\eta^2.
\end{align*}
We readily have $\boldsymbol{Y}^{(n+1)}\in L^{2}(\Omega;C([0,\delta];\mathbb{R}^{R}))$
and hence $\boldsymbol{Y}^{(n+1)}\in\mathbb{D}_{\mathrm{sto}}$.
Likewise, we have $\boldsymbol{U}^{(n+1)}\in\mathbb{D}_{\mathrm{det}}$,
since 
\begin{align*}
	\sup_{0\leq t\leq\delta}\|\boldsymbol{U}_{t}^{(n+1)}\|_{\mathrm{F}}^{2} & \leq2R+2\sup_{0\leq t\leq\delta}t\int_{0}^{t}\|{C}_{\boldsymbol{Y}_{s}^{(n)}}^{-1}\mathbb{E}[\boldsymbol{Y}_{s}^{(n)}a(s,X_{s}^{(n)})^{\top}](I_{d\times d}-P_{\boldsymbol{U}^{(n)}_{s}}^{\mathrm{row}})\|_{\mathrm{F}}^{2}\,ds\\
	& {\leq2R+8\delta\gamma^2\int_{0}^{\delta}\mathbb{E}[\sup_{0\leq s\leq\delta}|\boldsymbol{Y}_{s}^{(n)}|^{2}]\mathbb{E}[\sup_{0\leq s\leq\delta}|a(s,X_{s}^{(n)})|^{2}]\,ds}\\
	& \leq2R+8\delta\gamma^2(3\rho^2+1) C_{\mathrm{lgb}}(1+3R(3\rho^2+1))\\
	&\leq3R,
\end{align*}
where in the second inequality we have used the fact that $P_{\boldsymbol{U}^{(n)}_{s}}^{\mathrm{row}}$ is an orthogonal projector. Finally,
\begin{equation*}
	\sup_{0\leq t\leq\delta}\|\boldsymbol{U}_{t}^{(n+1)}-\boldsymbol{\varphi}\|^2_{\mathrm{F}}\leq 4\gamma^2(3\rho^2+1) C_{\mathrm{lgb}}(1+3R(3\rho^2+1))\delta\leq\eta^2.
\end{equation*}
Thus, by induction we conclude $(\boldsymbol{U}^{(n)},\boldsymbol{Y}^{(n)})\in\mathbb{D}_{\mathrm{det}}\times\mathbb{D}_{\mathrm{sto}}$
for $n\in\mathbb{N}$.
\end{proof}
We now establish a Lipschitz continuity for $F_1$ and $F_2$ on $\mathbb{D}_{\mathrm{det}}\times\mathbb{D}_{\mathrm{sto}}$.
\begin{lem}\label{lem: Lipsc Continuity}
Take $\delta>0$ as in \eqref{eq:def-delta}. There exists a constant $\tilde{C}:=\tilde{C}_{a,b,R,\rho,\delta,\gamma}>0$ such that for any $(\boldsymbol{V},\boldsymbol{Z}),(\tilde{\boldsymbol{V}},\tilde{\boldsymbol{Z}})\in\mathbb{D}_{\mathrm{det}}\times\mathbb{D}_{\mathrm{sto}}$ it holds
\begin{equation}
	\begin{aligned}
		\sup_{t\in[0,\delta]} \|F_{1}(\boldsymbol{V},\boldsymbol{Z})(t)-F_{1}(\tilde{\boldsymbol{V}}&,\tilde{\boldsymbol{Z}})(t)\|^{2}_{\mathrm{F}} 
		+ 
		\mathbb{E}\biggl[\sup_{t\in[0,\delta]} |F_{2}(\boldsymbol{V},\boldsymbol{Z})(t)-F_{2}(\tilde{\boldsymbol{V}},\tilde{\boldsymbol{Z}})(t)|^{2}\biggr]\\
		& \leq \tilde{C}\int_{0}^{\delta} \left( \sup_{t\in[0,s]}\|\boldsymbol{V}_{t}-\tilde{\boldsymbol{V}}_{t}\|^{2}_{\mathrm{F}}+\mathbb{E}[\sup_{t\in[0,s]}|\boldsymbol{Z}_{t}-\tilde{\boldsymbol{Z}}_{t}|^{2}]\right)ds
	\end{aligned}
\end{equation}  
\begin{proof}
For any $(\boldsymbol{V},\boldsymbol{Z}),(\tilde{\boldsymbol{V}},\tilde{\boldsymbol{Z}})\in\mathbb{D}_{\mathrm{det}}\times\mathbb{D}_{\mathrm{sto}}$, 
	from Doob's martingale inequality and Itô's isometry {applied similarly as in the proof of Lemma 3.1,} we have
\begin{align*}
		\mathbb{E}\biggl[\sup_{t\in[0,\delta]} & |F_{2}(\boldsymbol{V},\boldsymbol{Z})(t)-F_{2}(\tilde{\boldsymbol{V}},\tilde{\boldsymbol{Z}})(t)|^{2}\biggr]\\
		& \leq2\mathbb{E}\biggl[\sup_{t\in[0,\delta]}|\int_{0}^{t}\boldsymbol{V}_{s}a(s,\boldsymbol{V}_{s}^{\top}\boldsymbol{Z}_{s})\,ds-\int_{0}^{t}\tilde{\boldsymbol{V}}_{s}a(s,\tilde{\boldsymbol{V}}_{s}^{\top}\tilde{\boldsymbol{Z}}_{s})\,ds|^{2}\biggr]\\
		& \ \ \ \ +2\mathbb{E}\biggl[\sup_{t\in[0,\delta]}\Bigl|\int_{0}^{t}\bigl(\boldsymbol{V}_{s}b(s,\boldsymbol{V}_{s}^{\top}\boldsymbol{Z}_{s})-\tilde{\boldsymbol{V}}_{s}b(s,\tilde{\boldsymbol{V}}_{s}^{\top}\tilde{\boldsymbol{Z}}_{s})\bigr)dW_{s}\Bigr|^{2}\biggr]\\
		& \leq2\mathbb{E}\biggl[\sup_{t\in[0,\delta]}\int_{0}^{t}|\boldsymbol{V}_{s}a(s,\boldsymbol{V}_{s}^{\top}\boldsymbol{Z}_{s})-\tilde{\boldsymbol{V}}_{s}a(s,\tilde{\boldsymbol{V}}_{s}^{\top}\tilde{\boldsymbol{Z}}_{s})|^{2}\,ds\biggr]\\
		& \ \ \ \ +8\mathbb{E}\biggl[\int_{0}^{\delta}\|\boldsymbol{V}_{s}b(s,\boldsymbol{V}_{s}^{\top}\boldsymbol{Z}_{s})-\tilde{\boldsymbol{V}}_{s}b(s,\tilde{\boldsymbol{V}}_{s}^{\top}\tilde{\boldsymbol{Z}}_{s})\|_{\mathrm{F}}^{2}\,ds\biggr]\\
		& \leq C_{a,b,R,\rho,\delta}\int_{0}^{\delta}\bigl(\|\boldsymbol{V}_{s}-\tilde{\boldsymbol{V}}_{s}\|^{2}_{\mathrm{F}}+\mathbb{E}[|\boldsymbol{Z}_{s}-\tilde{\boldsymbol{Z}}_{s}|^{2}]\bigr)\,ds\\
		& \leq  C_{a,b,R,\rho,\delta}\int_{0}^{\delta}(\sup_{r\in[0,s]}\|\boldsymbol{V}_{r}-\tilde{\boldsymbol{V}}_{r}\|^{2}_{\mathrm{F}}+\mathbb{E}[\sup_{r\in[0,s]}|\boldsymbol{Z}_{r}-\tilde{\boldsymbol{Z}}_{r}|^{2}])ds,
	\end{align*}
	where $C_{a,b,R,\rho,\delta}$ is a positive constant. 
Similarly, we have for a constant $C_{a,b,R,\rho,\delta,\gamma} >0$ that 
	\begin{equation*}
		\sup_{t\in[0,\delta]} \|F_{1}(\boldsymbol{V},\boldsymbol{Z})(t)-F_{1}(\tilde{\boldsymbol{V}},\tilde{\boldsymbol{Z}})(t)\|^{2}_{\mathrm{F}} \leq C_{a,b,R,\rho,\delta,\gamma} \int_{0}^{\delta}(\sup_{r\in[0,s]} \|\boldsymbol{V}_{r}-\tilde{\boldsymbol{V}}_{r}\|^{2}_{\mathrm{F}}+\mathbb{E}[\sup_{r\in[0,s]}|\boldsymbol{Z}_{r}-\tilde{\boldsymbol{Z}}_{r}|^{2}])ds,
	\end{equation*}
	where we used the Lipschitz continuity of ${C}_{\boldsymbol{Y}_{s}}^{-1}$
	and $P^{\mathrm{row}}_{\boldsymbol{U}_{s}}$; see \cite[Lemma~3.5]{kazashi2021existence} and Lemma~\ref{lem: lip orth proj}.
\end{proof}
\end{lem}
Thanks to the previous results, we have that sequences $(\boldsymbol{U}^{(n)})_n$ and $(\boldsymbol{Y}^{(n)})_n$ not only live in $\mathbb{D}_{\mathrm{det}}$ and $\mathbb{D}_{\mathrm{sto}}$, respectively, but also converge therein.
\begin{lem}
\label{lem:sq-conv}The sequence $(\boldsymbol{U}^{(n)})_{n}$
admits a limit $\boldsymbol{U}\in\mathbb{D}_{\mathrm{det}}\subset C([0,\delta];\mathbb{R}^{R\times d})$,
and the sequence $(\boldsymbol{Y}^{(n)})_{n}$ admits a limit $\boldsymbol{Y}\in\mathbb{D}_{\mathrm{sto}}\subset C([0,\delta];\mathbb{R}^{R})$
almost surely. Moreover, $\boldsymbol{Y}$ is also an $L^{2}(\Omega;C([0,\delta];\mathbb{R}^{R}))$-limit.
\end{lem}

\begin{proof}
From Lemma~\ref{lem:sq-in-D} we have $(\boldsymbol{U}^{(n)},\boldsymbol{Y}^{(n)})\in\mathbb{D}_{\mathrm{det}}\times\mathbb{D}_{\mathrm{sto}}$ 
for all $n\in\mathbb{N}$. 
Let 
$\Delta_{U}^{(n)}(s):=\sup\limits_{0\leq r\leq s}\|\boldsymbol{U}_{r}^{(n)}-\boldsymbol{U}_{r}^{(n-1)}\|_{\mathrm{F}}^{2}$. 
Then, since $\|\boldsymbol{U}_{}^{(n)}-\boldsymbol{U}_{}^{(n-1)}\|_{\mathrm{F}}^{2}$ is continuous on $[0,\delta]$, so is $\Delta_{U}^{(n)}$, and thus $\Delta_{U}^{(n)}$ is measurable. 
Similarly, $\Delta_{Y}^{(n)}(s):=\sup\limits_{0\leq r\leq s}|\boldsymbol{Y}_{r}^{(n)}-\boldsymbol{Y}_{r}^{(n-1)}|^{2}$ is a.s.~continuous on $[0,\delta]$.
Noting that  $\boldsymbol{Y}^{(n)}\in\mathbb{D}_{\mathrm{sto}}$
implies $\|{C}_{\boldsymbol{Y}_{t}^{(n)}}^{-1}\|_{\mathrm{F}}^{2}\leq2\gamma$, from Lemma~\ref{lem: Lipsc Continuity} we have
\begin{align}
\Delta_{U}^{(n)}(\delta)+\mathbb{E}[\Delta_{Y}^{(n)}(\delta)] & \leq\tilde{C}\int_{0}^{\delta}(\Delta_{U}^{(n-1)}(s)+\mathbb{E}[\Delta_{Y}^{(n-1)}(s)])\,ds\nonumber \\
& \leq\tilde{C}^{n-1}\int_{0}^{\delta}\int_{0}^{s_{n-1}}\dotsb\int_{0}^{s_{2}}(\Delta_{U}^{(1)}(s_1)+\mathbb{E}[\Delta_{Y}^{(1)}(s_1)])\,ds_{1}\dotsb ds_{n-1}\nonumber \\
& =\frac{(\tilde{C}\delta)^{n-1}}{(n-1)!}(\Delta_{U}^{(1)}(\delta)+\mathbb{E}[\Delta_{Y}^{(1)}(\delta)]).\label{eq:Y-increment}
\end{align}
Chebyshev's inequality then implies
\[
\sum_{n=1}^{\infty}\mathbb{P}\Bigl(\Delta_{U}^{(n)}(\delta)+\Delta_{Y}^{(n)}(\delta)\geq\frac{1}{2^{n}}\Bigr)\leq\left(\Delta_{U}^{(1)}(\delta)+\mathbb{E}[\Delta_{Y}^{(1)}(\delta)]\right)2\sum_{n=1}^{\infty}\frac{(2\tilde{C}\delta)^{n-1}}{(n-1)!}<\infty,
\]
and thus from the Borel-Cantelli lemma we have
\[
\mathbb{P}\Bigl(\Bigl\{\exists k=k(\omega)\text{ s.t. }n\geq k\implies\Delta_{U}^{(n)}(\delta)+\Delta_{Y}^{(n)}(\delta)<\frac{1}{2^{n}}\Bigr\}\Bigr)=1.
\]
Hence, $(\boldsymbol{Y}^{(n)}(\omega))_{n}$ has a limit $\boldsymbol{Y}(\omega)\in C([0,\delta];\mathbb{R}^{R})$, where the convergence is uniformly in $t$ on $[0,\delta]$, a.s.
Moreover, 
from the completeness of the underlying probability space, 
$\boldsymbol{Y}$ is $(\mathcal{F}_{t})$-adapted. 
Also, from
$\boldsymbol{Y}^{(n)}\in\mathbb{D}_{\mathrm{sto}}$, $n\in\mathbb{N}$,
Fatou's lemma implies 
\begin{align}
\mathbb{E}\Bigl[\sup_{0\leq t\leq\delta}|\boldsymbol{Y}_{t}|^{2}\Bigr] & \le\liminf_{n\to\infty}\mathbb{E}\Bigl[\sup_{0\leq t\leq\delta}|\boldsymbol{Y}_{t}^{(n)}|^{2}\Bigr]\leq4\rho,\label{eq:Y-bound}\\
\mathbb{E}\Bigl[\sup_{0\leq t\leq\delta}|\boldsymbol{Y}_{t}-\boldsymbol{\xi}|^{2}\Bigr] & \leq\eta^2,
\end{align}
and thus $\boldsymbol{Y}\in\mathbb{D}_{\mathrm{sto}}$. An
analogous argument applies to see that $(\boldsymbol{U}^{(n)})_{n}$
has a limit $\boldsymbol{U}$ in $C([0,\delta];\mathbb{R}^{R \times d})$
with $\boldsymbol{U}\in\mathbb{D}_{\mathrm{det}}$.

The sequences $\boldsymbol{Y}^{(n)}$ converge in $L^{2}(\Omega;C([0,\delta];\mathbb{R}^{R}))$
as well. Indeed, from \eqref{eq:Y-increment}, for $j>n$ we have
\begin{align*}
\sqrt{\mathbb{E}\Bigl[\sup_{0\leq t\leq\delta}|\boldsymbol{Y}_{t}^{(j)}-\boldsymbol{Y}_{t}^{(n)}|^{2}\Bigr]} & \leq\sum_{k=n}^{j-1}\sqrt{\mathbb{E}\Bigl[\sup_{0\leq t\leq\delta}|\boldsymbol{Y}_{t}^{(k+1)}-\boldsymbol{Y}_{t}^{(k)}|^{2}\Bigr]}\\
&\leq \sqrt{\Delta_{U}^{(1)}(\delta)+\mathbb{E}[\Delta_{Y}^{(1)}(\delta)]}\,\sum_{k=n}^{j-1}\sqrt{\frac{(\tilde{C}\delta)^{k-1}}{(k-1)!}}
\end{align*}
and thus Fatou's lemma implies
\[
\mathbb{E}\Bigl[\sup_{0\leq t\leq\delta}|\boldsymbol{Y}_{t}-\boldsymbol{Y}_{t}^{(n)}|^{2}\Bigr]\leq(\Delta_{U}^{(1)}(\delta)+\mathbb{E}[\Delta_{Y}^{(1)}(\delta)])\biggl(\sum_{k=n}^{\infty}\sqrt{\frac{(\tilde{C}\delta)^{k-1}}{(k-1)!}}\biggr)^{2}<\infty,
\]
hence $\lim\limits_{n\to\infty}\mathbb{E}\Bigl[\sup\limits_{0\leq t\leq\delta}|\boldsymbol{Y}_{t}-\boldsymbol{Y}_{t}^{(n)}|^{2}\Bigr]=0$.
\end{proof}
{
	The argument used in Lemmata \ref{lem:sq-in-D}-\ref{lem:sq-conv} to show convergence of the Picard iterates resembles the standard one used in establishing well-posedness of SDEs, however containing important differences. First, the map \eqref{eq:F1} to update $\boldsymbol{U}$ involves a ``stochastic projection" introducing a dependence on the law of the $\boldsymbol{Y}$ process. Second, the same projection involves the term $C^{-1}_{\boldsymbol{Y}}$ (inverse of the Gramian of the stochastic basis), which might not be defined for all iterates or it may fail to be Lipschitz continuous with respect to $\boldsymbol{Y}$. We thus have to restrict to a well-designed set, containing the initial state, in which the inverse of the Gramian is well-defined and Lipschitz continuous. 
	Lemma \ref{lem:sq-in-D} guarantees that the Picard iterates $(\boldsymbol{U}^{(n)},\boldsymbol{Y}^{(n)})$ are contained in such set for all $n$.  Lemma \ref{lem: Lipsc Continuity} shows the Lipschitz continuity of the mapping that defines the iterates on this set, and Lemma \ref{lem:sq-conv} assures that  $(\boldsymbol{U}^{(n)},\boldsymbol{Y}^{(n)})$ admit a limit in the same set. 
	We are now prepared to establish our existence result for the DO system, employing an argument that closely resembles the conventional approach used for SDEs.}
\begin{thm}[Existence of a DO solution]
\label{thm:local-existence}
For any {$(\boldsymbol{\varphi}, \boldsymbol{\xi})$, where $\boldsymbol{\varphi} \in  \mathbb{R}^{d \times R}$ has orthonormal rows and $ \boldsymbol{\xi} \in L^2(\Omega,\mathbb{R}^R)$ satisfies $\|\boldsymbol{\xi}\|_{[L^2(\Omega)]^R} = \rho > 0$ and $\|{C}_{\boldsymbol{\xi}}^{-1}\|_{\mathrm{F}}=\gamma > 0$, the DO equations \eqref{eq:DLR-eq-U} and \eqref{eq:DLR-eq-Y} with the initial condition $(\boldsymbol{U}_0, \boldsymbol{Y}_0)=(\boldsymbol{\varphi}, \boldsymbol{\xi})$} have a local (in time) strong DO solution $(\boldsymbol{U},\boldsymbol{Y}) \in\mathbb{D}_{\mathrm{det}}\times\mathbb{D}_{\mathrm{sto}}$ with $\delta$ given by \eqref{eq:def-delta}. 
\end{thm}

\begin{proof}
We show that the limit $(\boldsymbol{U},\boldsymbol{Y})$ in Lemma \ref{lem:sq-conv} satisfies $\boldsymbol{U}=F_{1}(\boldsymbol{U},\boldsymbol{Y})$ and $\boldsymbol{Y}=F_{2}(\boldsymbol{U},\boldsymbol{Y})$, hence it is a DO solution. 
From Lemma~\ref{lem: Lipsc Continuity}, $F_1$ and $F_2$ are (Lipschitz) continuous on $\mathbb{D}_{\mathrm{det}}\times\mathbb{D}_{\mathrm{sto}}$:
\begin{equation*}
\begin{aligned}
	\max\biggl\{\sup_{t\in[0,\delta]} \|F_{1}(\boldsymbol{V},\boldsymbol{Z})(t)-F_{1}(&\tilde{\boldsymbol{V}},\tilde{\boldsymbol{Z}})(t)\|^{2}_{\mathrm{F}} \,, \, \mathbb{E}\biggl[\sup_{t\in[0,\delta]} |F_{2}(\boldsymbol{V},\boldsymbol{Z})(t)-F_{2}(\tilde{\boldsymbol{V}},\tilde{\boldsymbol{Z}})(t)|^{2}\biggr]\biggr\}\\
	& \leq \tilde{C}\int_{0}^{\delta} \left( \sup_{t\in[0,s]}\|\boldsymbol{V}_{t}-\tilde{\boldsymbol{V}}_{t}\|^{2}_{\mathrm{F}}+\mathbb{E}[\sup_{t\in[0,s]}|\boldsymbol{Z}_{t}-\tilde{\boldsymbol{Z}}_{t}|^{2}]\right)ds \\
	& \leq \tilde{C}\delta \left( \sup_{t\in[0,\delta]}\|\boldsymbol{V}_{t}-\tilde{\boldsymbol{V}}_{t}\|^{2}_{\mathrm{F}}+\mathbb{E}[\sup_{t\in[0,\delta]}|\boldsymbol{Z}_{t}-\tilde{\boldsymbol{Z}}_{t}|^{2}]\right).
\end{aligned}
\end{equation*} 
Thus, we have
\begin{align*}
\mathbb{E}\biggl[\sup_{t\in[0,\delta]}|\boldsymbol{Y}_{t}-F_{2}(\boldsymbol{U},\boldsymbol{Y})(t)|^{2}\biggr] & =\lim_{n\to\infty}\mathbb{E}\biggl[\sup_{t\in[0,\delta]}|\boldsymbol{Y}_{t}^{(n)}-F_{2}(\boldsymbol{U}^{(n)},\boldsymbol{Y}^{(n)})(t)|^{2}\biggr]\\
& =\lim_{n\to\infty}\mathbb{E}\biggl[\sup_{t\in[0,\delta]}|\boldsymbol{Y}_{t}^{(n)}-\boldsymbol{Y}_{t}^{(n+1)}|^{2}\biggr]=0,
\end{align*}
and hence $\boldsymbol{Y}_{t}=F_{2}(\boldsymbol{U},\boldsymbol{Y})(t)$
for $t\in[0,\delta]$, a.s. We see $\boldsymbol{U}_{t}=F_{1}(\boldsymbol{U},\boldsymbol{Y})(t)$
analogously.
\end{proof}
To establish uniqueness, first we will show the following norm bound analogous to Lemma~\ref{lem: gronwall for X}.
\begin{lem}
\label{lem:stab}For $T>0$, suppose that $\boldsymbol{U}\in C([0,T];\mathbb{R}^{R\times d})$, with $\boldsymbol{U}_0$ having orthonormal rows,
and $\boldsymbol{Y}\in L^{2}(\Omega;C([0,T];\mathbb{R}^{R}))$, with $\boldsymbol{Y}_0$ having linearly independent components, satisfy
$\boldsymbol{U}_{t}=F_{1}(\boldsymbol{U},\boldsymbol{Y})(t)$ and
$\boldsymbol{Y}_{t}=F_{2}(\boldsymbol{U},\boldsymbol{Y})(t)$ for all
$t\in[0,T]$. Then, for all $t\in[0,T]$ we have
\begin{align}
\|\boldsymbol{U}_{t}\|_{\mathrm{F}} & =\sqrt{R};\label{eq:U-stab}\\
\mathbb{E}\bigl[|\boldsymbol{Y}_{t}|^{2}\bigr] & \leq3(\mathbb{E}[|\boldsymbol{Y}_{0}|^{2}]+(1+T)TC_{\mathrm{lgb}})\mathrm{exp}\bigl(3(1+T)TC_{\mathrm{lgb}}\bigr)=:M(T).\label{eq:Y-stab}
\end{align}
\end{lem}

\begin{proof}
First, from $\boldsymbol{U}_{t}=F_{1}(\boldsymbol{U},\boldsymbol{Y})(t)$,
the function $\boldsymbol{U}$ is absolutely continuous on $[0,T]$, and
thus differentiable almost everywhere. The derivative $\dot{\boldsymbol{U}}_{t}$
satisfies 
\[
\dot{\boldsymbol{U}}_{t}\boldsymbol{U}_{t}^{\top}={C}_{\boldsymbol{Y}_{t}}^{-1}\mathbb{E}[\boldsymbol{Y}_t a(t,\boldsymbol{U}_{t}^{\top}\boldsymbol{Y})^{\top}](I_{d\times d}-P^{\mathrm{row}}_{\boldsymbol{U}_{t}})\boldsymbol{U}_{t}^{\top}=0,
\]
and thus $\frac{\mathrm{d}}{\mathrm{d}t}(\boldsymbol{U}_{t}\boldsymbol{U}_{t}^{\top})=0$
a.e.\ on $[0,T]$. 
Therefore, from the orthonormality of the
initial condition $\boldsymbol{\varphi}\boldsymbol{\varphi}^{\top}=I_{R\times R}$,
for all $t\in[0,T]$ we have
\[
\bigl(\boldsymbol{U}_{t}\boldsymbol{U}_{t}^{\top}\bigr)_{jk}=\epsilon_{jk}+\int_{0}^{t}0\,ds=\epsilon_{jk},
\]
where $\epsilon_{jk}=1$ only if $j=k$, and $0$ otherwise. This
shows the identity \eqref{eq:U-stab}. 

For $\boldsymbol{Y}$, Itô's isometry implies
\begin{align*}
\mathbb{E}\bigl[|\boldsymbol{Y}_{t}|^{2}\bigr] & \leq3\mathbb{E}\biggl[|\boldsymbol{Y}_{0}|^{2}+t\int_{0}^{t}\|\boldsymbol{U}_{s}\|_{2}^{2}|a(s,\boldsymbol{U}_{s}^{\top}\boldsymbol{Y}_{s})|^{2}\,ds\biggr]+3\mathbb{E}\Bigl[\int_{0}^{t}\|\boldsymbol{U}_{s}b(s,\boldsymbol{U}_{s}^{\top}\boldsymbol{Y}_{s})\|_{\mathrm{F}}^{2}\,ds\Bigr]\\
& \leq3\mathbb{E}[|\boldsymbol{Y}_{0}|^{2}]+3(1+T)C_{\mathrm{lgb}}\int_{0}^{t}(1+\mathbb{E}[|\boldsymbol{Y}_{s}|^{2}])\,ds,
\end{align*}
where we used Assumption~\ref{assump:lin-growth} $|a(s,x)|^{2}+\|b(s,x)\|_{\mathrm{F}}^{2}\leq C_{\mathrm{lgb}}(1+|x|^{2})$. Thus, Gronwall's lemma implies \eqref{eq:Y-stab}.
\end{proof}

Now the following uniqueness result follows.
\begin{thm}[Uniqueness of DO solutions]\label{thm: uniqueness of DO}
Let the assumptions of Lemma~\ref{lem:stab} hold.
For $T>0$, suppose that $\boldsymbol{U},\tilde{\boldsymbol{U}}\in C([0,T];\mathbb{R}^{R\times d})$
and $\boldsymbol{Y},\tilde{\boldsymbol{Y}}\in L^2(\Omega;C([0,T];\mathbb{R}^R)$
satisfy $\boldsymbol{U}_{t}=F_{1}(\boldsymbol{U},\boldsymbol{Y})(t)$
and $\tilde{\boldsymbol{U}}_{t}=F_{1}(\tilde{\boldsymbol{U}},\tilde{\boldsymbol{Y}})(t)$;
$\boldsymbol{Y}_{t}=F_{2}(\boldsymbol{U},\boldsymbol{Y})(t)$ and
$\tilde{\boldsymbol{Y}}_{t}=F_{2}(\tilde{\boldsymbol{U}},\tilde{\boldsymbol{Y}})(t)$
for $t\in[0,T]$. Then, we have
\[
\mathbb{P}\Bigl(\sup_{0\leq t\leq T}\|\boldsymbol{U}_{t}-\tilde{\boldsymbol{U}}_{t}\|_{\mathrm{F}}^{2}+\sup_{0\leq t\leq T}|\boldsymbol{Y}_{t}-\tilde{\boldsymbol{Y}}_{t}|>0\Bigr)=0.
\]
\end{thm}

\begin{proof}
By hypothesis, the solutions $\boldsymbol{U},\tilde{\boldsymbol{U}}$
and $\boldsymbol{Y},\tilde{\boldsymbol{Y}}$ satisfy
the stability estimates shown in Lemma~\ref{lem:stab}. 
Moreover, from the continuity of $t\to\mathbb{E}[\boldsymbol{Y}_{t}\boldsymbol{Y}_{t}^{\top}]^{-1}$, we have
\[
\max\Big\{
\max\limits_{s\in[0,T]}\|\mathbb{E}[\boldsymbol{Y}_{s}\boldsymbol{Y}_{s}^{\top}]^{-1}\|_{\mathrm{F}}
,
\max\limits_{s\in[0,T]}\|\mathbb{E}[
\tilde{\boldsymbol{Y}}_{s}
\tilde{\boldsymbol{Y}}_{s}^{\top}]^{-1}\|_{\mathrm{F}}
\Big\}=\tilde{\gamma}<\infty
\]for some $\tilde{\gamma}>0$. Then, noting the norm bounds \eqref{eq:U-stab}
and \eqref{eq:Y-stab}, by an argument similar to the proof of Lemma~\ref{lem: Lipsc Continuity} (see also \cite[Lemma 3.5]{kazashi2021existence} and Lemma \ref{lem: lip orth proj}), 
with a constant $\tilde{C}=\tilde{C}(\tilde{\gamma})>0$ we have

\begin{align}
\sup_{0\leq s\leq t}\|\boldsymbol{U}_{s}-\boldsymbol{U}'_{s}\|_{\mathrm{F}}^{2} + \mathbb{E}\bigl[\sup_{0\leq s\leq t}|\boldsymbol{Y}_{s}-\boldsymbol{Y}'_{s}|^{2}\bigr] & \leq \tilde{C} \int_{0}^{t} \left(\sup_{0\leq s\leq r}\|\boldsymbol{U}_{s}-\boldsymbol{U}'_{s}\|_{\mathrm{F}}^{2} + \mathbb{E}\bigl[\sup_{0\leq s\leq r}|\boldsymbol{Y}_{s}-\boldsymbol{Y}'_{s}|^{2}\bigr] \right)dr
\end{align}
for $t \in [0,T]$.
Thus, applying the Gronwall's lemma yields
\[
\sup_{0\leq s\leq t}\|\boldsymbol{U}_{s}-\boldsymbol{U}'_{s}\|_{\mathrm{F}}^{2}+\mathbb{E}\bigl[\sup_{0\leq s\leq t}|\boldsymbol{Y}_{s}-\boldsymbol{Y}'_{s}|^{2}\bigr]=0.
\]
Now the proof is complete.
\end{proof}

\begin{remark}
The uniqueness of the DO solution can be also deduced from 
the proof of Proposition~\ref{prop: equiv of param}. 
To see this, let $(\boldsymbol{U},\boldsymbol{Y})$
and $(\tilde{\boldsymbol{U}},\tilde{\boldsymbol{Y}})$ be two strong DO solutions with the same initial datum $(\boldsymbol{U}_0,\boldsymbol{Y}_0)$.
Then, following the proof of Proposition~\ref{prop: equiv of param}, we have  $(\boldsymbol{U},\boldsymbol{Y})=(\Theta \boldsymbol{U}_t, \Theta \boldsymbol{Y}_t)$ with $\Theta=I_{R\times R}$.
\end{remark}
We conclude this section by showing the continuity of the solution with respect to the initial datum.
{\begin{lem}
Let $(\boldsymbol{U}_t,\boldsymbol{Y}_t)$ and  $(\tilde{\boldsymbol{U}}_t,\tilde{\boldsymbol{Y}}_t)$ 
be DO solutions
on $[0,T]$ with initial data $(\boldsymbol{U}_0,\boldsymbol{Y}_0)$ and  $(\tilde{\boldsymbol{U}}_0,\tilde{\boldsymbol{Y}}_0)$, respectively. Define $\rho = \max\{\|\boldsymbol{Y}_0\|_{[L^{2}(\Omega)]^{R}},\|\tilde{\boldsymbol{Y}}_0\|_{[L^{2}(\Omega)]^{R}}\}$ and $\gamma = \max\{\|C^{-1}_{\boldsymbol{Y}_0}\|_{\mathrm{F}},\|C^{-1}_{\tilde{\boldsymbol{Y}}_0}\|_{\mathrm{F}}\}$. Then, under Assumptions $1$ and $2$,
\begin{equation*}
\mathbb{E}[\sup_{t \in [0,T]} |\boldsymbol{Y}_t-\tilde{\boldsymbol{Y}}_t|^2] +\sup_{t \in [0,T]} \|\boldsymbol{U}_t-\tilde{\boldsymbol{U}}_t\|^2_{\mathrm{F}} \leq C_{a,b,R,\rho,T,\gamma} \left( \mathbb{E}[ |\boldsymbol{Y}_0-\tilde{\boldsymbol{Y}}_0|^2] + \|\boldsymbol{U}_0-\tilde{\boldsymbol{U}}_0\|^2_{\mathrm{F}} \right)
\end{equation*}
holds. 
Moreover, let 
$X_t$
and  $\tilde{X_t}$ be DO approximations on $[0,T]$ with initial data $X_0$ and $\tilde{X}_0$, respectively, and
let 
\[
\tilde{\gamma} := 
\inf_{t \in [0,T]} \max \{\sigma_R \!\!\;\bigl(\mathbb{E}[{X}_t{X}_t^{\top}]\bigr), \sigma_R \!\!\;\bigl(\mathbb{E}[\tilde{X}_t\tilde{X}_t^{\top}]\bigr) \}.
\] 
Then, there exists a positive constant ${C} >0$ such that
\begin{equation}\label{eq:labeling-just-for-discussions}
\sup_{t \in [0,{T}]} \mathbb{E}[ |{X}_t-\tilde{X}_t|^2] \leq {C}\mathbb{E}[ |X_0-\tilde{X}_0|^2].
\end{equation}
\begin{proof}
Using an argument similar to
the proof of Lemma~\ref{lem: Lipsc Continuity} and the Lipschitz continuity  of the projections as in Lemma \ref{lem: lip orth proj}, we obtain the following relation for the stochastic basis:
\begin{equation*}
\begin{aligned}
	\mathbb{E}[\sup_{t \in [0,T]} |\boldsymbol{Y}_t-\tilde{\boldsymbol{Y}}_t|^2]  \leq &2 \mathbb{E}[ |\boldsymbol{Y}_0-\tilde{\boldsymbol{Y}}_0|^2] \\
 &+ C_{a,b,R,\rho,T}\int_{0}^{T}(\sup_{r\in[0,t]}\|\boldsymbol{U}_{r}-\tilde{\boldsymbol{U}}_{r}\|^{2}_{\mathrm{F}}+\mathbb{E}[\sup_{r\in[0,t]}|\boldsymbol{Y}_{r}-\tilde{\boldsymbol{Y}}_{r}|^{2}])dt.
\end{aligned}
\end{equation*} 
Similarly, noting the Lipschitz continuity of the inverse Gram matrix \cite[Lemma~3.5]{kazashi2021existence}, for the deterministic basis we have
\begin{equation*}
\begin{aligned}
	\|\boldsymbol{U}_t-\tilde{\boldsymbol{U}}_t\|^2_{\mathrm{F}}  & \leq 2\|\boldsymbol{U}_0-\tilde{\boldsymbol{U}}_0\|^2_{\mathrm{F}} + C_{a,b,R,\rho,T,{\gamma}} \int_{0}^{T}(\sup_{r\in[0,t]} \|\boldsymbol{U}_{r}-\tilde{\boldsymbol{U}}_{r}\|^{2}_{\mathrm{F}}+\mathbb{E}[\sup_{r\in[0,t]}|\boldsymbol{Y}_{r}-\tilde{\boldsymbol{Y}}_{r}|^{2}])dt.
\end{aligned}
\end{equation*}
Then, the Gronwall's lemma yields the first part of the statement. 
Finally, to prove {\eqref{eq:labeling-just-for-discussions}}, we proceed as above and as done in Theorem~\ref{thm:uniqueness of DLR}.
Therefore, for all $t \in [0,{T}]$ there exists a positive constant $C$ such that it holds
\begin{equation*}	
\begin{aligned}    \mathbb{E}[|{X}_{t}-\tilde{X}_{t}|^2] & \leq 2  \mathbb{E}[|{X}_{0}-\tilde{X}_{0}|^2] + 2\overline{C}_{a,b,R,T,\tilde{\gamma}} \int_{0}^{T} \mathbb{E}[|{X}_{s}-\tilde{X}_{s}|^2] ds, \\
	& \leq C \mathbb{E}[|{X}_{0}-\tilde{X}_{0}|^2],
\end{aligned}
\end{equation*}
where in the last line we use the Gronwall's lemma.  
\end{proof}
\end{lem}}
\section{Maximality}\label{sec: maximality}
In the previous section, we established the existence and uniqueness of strong DO solutions locally in time on an interval $[0,T]$.
In this section, we investigate how much such an interval can be extended. 

We give a characterisation of the \textrm{maximal} interval of existence of the strong DO solution in terms of $\|{C}_{\boldsymbol{Y}_{t}}^{-1}\|_{\mathrm{F}}$.   
It turns out that the DO solution can be extended until  $\|{C}_{\boldsymbol{Y}_{t}}^{-1}\|_{\mathrm{F}}$ explodes. 
If $\|{C}_{\boldsymbol{Y}_{t}}^{-1}\|_{\mathrm{F}}$ stays bounded for all $t>0$, then the DO solution exists globally.    
If $\|{C}_{\boldsymbol{Y}_{t}}^{-1}\|_{\mathrm{F}}$ explodes at a finite \textit{explosion time} $T_e$,  the couple $(\boldsymbol{U},\boldsymbol{Y})$ inevitably ceases to exist at $T_e$; we will show nevertheless that the DO approximation $X = \boldsymbol{U}^{\top}\boldsymbol{Y}$ can be extended beyond $T_e$.

\subsection{Explosion time of the DO solution} \label{sec:blow-up}

Theorem \ref{thm:local-existence} guarantees the (unique) existence of a strong DO solution, albeit up to possibly a short time $T$. In this section, we show that the solution can be extended until ${C}_{\boldsymbol{Y}_{t}}^{-1}$
becomes singular. 

Analogous maximality results have been considered in the DLRA literature for deterministic and random PDEs  \cite{koch2007regularity,kazashi2021existence}. 
For the SDE case of this paper, we need to proceed with caution. 
Indeed, adhering to the definition of SDEs, we need the DO solution to be path-wise continuous a.s. 
As such, we consider extension with countable number of operations, which we describe in the following.

Let $[0,T]$ be the interval on which the DO solution $(\boldsymbol{U},\boldsymbol{Y})$ exists; such existence is guaranteed by Theorem \ref{thm:local-existence}. 
Let us choose $n \in \mathbb{N}$ such that the following two bounds are satisfied
\begin{equation}\label{eq:rhogamma-ub}
\begin{aligned}
\mathbb{E}[\|\boldsymbol{Y}_T\|^2] & \leq \mathbb{E}[\|\boldsymbol{Y}_0\|^2] + n := \rho^2_n ;\\
\|{C}_{\boldsymbol{Y}_{T}}^{-1}\|^2_{\mathrm{F}} & \leq \|{C}_{\boldsymbol{Y}_{0}}^{-1}\|^2_{\mathrm{F}} + n := \gamma^2_n .
\end{aligned}
\end{equation}
We will show that the solution can be extended at least until $\mathbb{E}[\|\boldsymbol{Y}_t\|^2]$ or  $\|{C}_{\boldsymbol{Y}_{t}}^{-1}\|^2_{\mathrm{F}}$ hits the bound  $\rho^2_n$ or $\gamma^2_n$. 
Define $\delta(n)$ by
{
\begin{equation}\label{delta_n}
\delta(n):=\min\biggl\{1 \,,\,
\frac{\min\{1,
\eta^2_n
\}}{36RC_{\mathrm{lgb}}(1+3R(3\rho_n^2 +1))} \,,\,
\frac{\min\{
\eta^2_n
,R\}}{8\gamma_n^2(3\rho_n^2+1) C_{\mathrm{lgb}}(1+3R(3\rho_n^2 +1))}\biggr\},
\end{equation}}where  $\eta_n:=\min\{\eta(\rho_n,\gamma_n),\eta(\sqrt{R},\sqrt{R})\}$, with $\eta(\cdot,\cdot)$ defined in \eqref{eq:def-eta}.
Then from the proofs of Lemmata~\ref{lem:sq-in-D} and~\ref{lem:sq-conv}, and Theorem~\ref{thm:local-existence},
we can construct a convergent Picard-iteration for the interval $[T,T+\delta(n)]$, which yields the extension of the DO solution up to $[0,T+\delta(n)]$. 
Set $T:=T+\delta(n)$. 
By construction $\delta(n)$ is independent of $T$, and with the same $\delta(n)$ we can repeat the same argument as long as the bound \eqref{eq:rhogamma-ub} is satisfied. Hence, the solution can be extended until either  $\mathbb{E}[\|\boldsymbol{Y}_{T+\delta(n)}\|^2] \leq \rho^2_n$ or  $\|{C}^{-1}_{\boldsymbol{Y}_{T+\delta(n)}}\|^2_{\mathrm{F}} \leq \gamma^2_n$ gets violated.

From the argument above, the following quantities are well defined for any $n \in \mathbb{N}$:
\begin{align*}
\tau_{n}^{1}&:=\inf\{t>0\mid\|{C}_{\boldsymbol{Y}_{t}}^{-1}\|_{\mathrm{F}}=\|{C}_{\boldsymbol{Y}_{0}}^{-1}\|_{\mathrm{F}}+n\},\\
\tau_{n}^{2}&:=\inf\{t>0\mid\|\boldsymbol{Y}_{t}\|_{[L^{2}(\Omega)]^{R}}=\|\boldsymbol{Y}_{0}\|_{[L^{2}(\Omega)]^{R}}+n\}
\end{align*}
with a convention $\inf\varnothing=\infty$. With these, we define  the sequence
\begin{align}\label{eq: tau_n}
\tau_{n}:=\tau_{n}(\boldsymbol{U}_{0},\boldsymbol{Y}_{0}):=\min\{\tau_{n}^{1},\tau_{n}^{2}\},\quad n\in\mathbb{N},
\end{align}
which is a sequence of stopping times. 
By continuity of the paths,  $(\tau_{n})_{n\in\mathbb{N}}$ 
is a non-decreasing sequence, which allows us to define  $T_e:=\lim\limits_{n\to\infty}\tau_{n}$. 

Now we will show that if $T_e<\infty$, then the norm of the inverse of the Gram matrix must blow up.

\begin{prop}\label{prop: max T_e}
We have either $T_e=\lim\limits_{n\to\infty}\tau_{n}=\infty$ or 
$T_e<\infty$, where $\tau_n$ is defined in~\eqref{eq: tau_n}. If $T_e<\infty$, we necessarily have $\lim\limits_{t\uparrow T_e}\|{C}_{\boldsymbol{Y}_{t}}^{-1}\|_{\mathrm{F}}=\infty$.
\end{prop}
\begin{proof}
From the norm bound \eqref{eq:Y-stab},
for sufficiently large $n$ we must have

\noindent $\|\boldsymbol{Y}_{t}\|_{[L^{2}(\Omega)]^{R}}<\|\boldsymbol{Y}_{0}\|_{[L^{2}(\Omega)]^{R}}+n$
for any $t\in(0,T_e)$, and thus without loss of generality we  assume
\[
\tau_{n}=\tau_{n}^{1}=\inf\{t>0\mid\|{C}_{\boldsymbol{Y}_{t}}^{-1}\|_{\mathrm{F}}=\|{C}_{\boldsymbol{Y}_{0}}^{-1}\|_{\mathrm{F}}+n\},\quad n\in\mathbb{N}.
\]

We will first show that $T_e<\infty$ implies $\limsup\limits_{t\uparrow T_e}\|{C}_{\boldsymbol{Y}_{t}}^{-1}\|_{\mathrm{F}}=\infty$.
We argue by contradiction, and assume $T_e<\infty$ and $\limsup\limits_{t\uparrow T_e}\|{C}_{\boldsymbol{Y}_{t}}^{-1}\|_{\mathrm{F}}<\infty$.
Then, we have $\sup\limits_{t\in[T_e-\delta,T_e)}\|{C}_{\boldsymbol{Y}_{t}}^{-1}\|_{\mathrm{F}}<K<\infty$
for some $K>0$ and $\delta>0$. But then since $T_e<\infty$, the continuity
of $t\mapsto\boldsymbol{Y}_{t}$ implies $\|{C}_{\boldsymbol{Y}_{\tau_{n}}}^{-1}\|_{\mathrm{F}}=\|{C}_{\boldsymbol{Y}_{0}}^{-1}\|_{\mathrm{F}}+n$
and thus for any $n$ sufficiently large we have 
\[
\|{C}_{\boldsymbol{Y}_{\tau_{n}}}^{-1}\|_{\mathrm{F}}=\|{C}_{\boldsymbol{Y}_{0}}^{-1}\|_{\mathrm{F}}+n<K,
\]
which is absurd. Hence, $T_e<\infty$ implies $\limsup\limits_{t\uparrow T_e}\|{C}_{\boldsymbol{Y}_{t}}^{-1}\|_{\mathrm{F}}=\infty$.

To conclude the proof we will show 
\[
\lim_{t\uparrow T_e}\|{C}_{\boldsymbol{Y}_{t}}^{-1}\|_{\mathrm{F}}=\infty.
\]
If this is false, then there exist a sequence $t_{m}\uparrow T_e$ and
$\gamma>0$ such that $\|{C}_{\boldsymbol{Y}_{t_{m}}}^{-1}\|_{\mathrm{F}}\leq\gamma$
for all $m\ge0$. But since $\limsup\limits_{t\uparrow T_e}\|{C}_{\boldsymbol{Y}_{t}}^{-1}\|_{\mathrm{F}}=\infty$
there is a sequence $s_{k}\uparrow e$ such that $\|{C}_{\boldsymbol{Y}_{s_{k}}}^{-1}\|_{\mathrm{F}}\geq\gamma+1$
for all $k\ge0$. We take a sub-sequence $(s_{k_{m}})_{m}$ so that
$t_{m}<s_{k_{m}}$ for all $m$. From the continuity of $t\mapsto\|{C}_{\boldsymbol{Y}_{t}}^{-1}\|_{\mathrm{F}}$
on $[t_{m},s_{k_{m}}]$, there exists $h_{m}\in[0,s_{k_{m}}-t_{m}]$
such that $\|{C}_{\boldsymbol{Y}_{t_{m}+h_{m}}}^{-1}\|_{\mathrm{F}}=\gamma+1$.
Now, from \eqref{eq:Y-stab} and \cite[Lemma 3.5]{kazashi2021existence}
we have for any $m\geq0$
\[
1\leq\|{C}_{\boldsymbol{Y}_{t_{m}+h_{m}}}^{-1}\|_{\mathrm{F}}-\|{C}_{\boldsymbol{Y}_{t_{m}}}^{-1}\|_{\mathrm{F}}\leq C_{T_e,R,\gamma}\|\boldsymbol{Y}_{t_{m}+h_{m}}-\boldsymbol{Y}_{t_{m}}\|_{[L^{2}(\Omega)]^{R}},
\]
which is {impossible} since $h_{m}\to0$ as $m\to\infty$ and $\boldsymbol{Y}$
is continuous on $[0,T_e)$. Hence, the proof is complete.
\end{proof} 
\subsection{Extension up to the explosion time}
Even when the explosion time $T_e$ for the DO solution is finite, and thus $\boldsymbol{U}$ and $\boldsymbol{Y}$ cease to exist at $T_e$, we will show that the product $X=\boldsymbol{U}^\top\boldsymbol{Y}$ nevertheless admits a continuous extension up to $[0,T_e]$, and beyond $T_e$, under suitable assumptions.

For any $t'<t<T_e$, from \eqref{eq:eq-manifold DLR} we have 
\begin{equation}\label{eq:DLR-diff}
\begin{aligned}
X_{t}-X_{t'}
=\int_{t'}^{t}\bigl(\mathscr{P}_{\mathcal{U}({X}_s)}+\mathscr{P}_{\mathcal{Y}({X}_s)}-\mathscr{P}_{\mathcal{U}({X}_s)}\mathscr{P}_{\mathcal{Y}({X}_s)}\bigr)a(s,X_{s})\,ds
+\int_{t'}^{t}\mathscr{P}_{\mathcal{U}({X}_s)}b(s,X_{s})\,dW_{s},
\end{aligned}
\end{equation}
and, hence, the Itô's isometry implies
\[
\mathbb{E}[|X_{t}-X_{t'}|^{2}]\leq2T_e\int_{t'}^{t}|a(s,X_{s})|^{2}\,ds+2\int_{t'}^{t}\|b(s,X_{s})\|_{\mathrm{F}}^{2}\,ds.
\]
From Assumption~\ref{assump:lin-growth}, the orthogonality {of the rows}
of $\boldsymbol{U}_{t}$ and the norm bound \eqref{eq:Y-stab} of $\boldsymbol{Y}_{t}$
it follows that
\[
\mathbb{E}[|X_{t}-X_{t'}|^{2}]\leq C(t-t')
\]
with $C:=4\max\{1,T_e\}C_{\mathrm{lgb}}\bigl(1+M(T_e)\bigr).$
Therefore, $(X_{t})_{0\leq t< T_e}$ admits a
unique extension $X_{T_e}:=\lim_{t\uparrow T_e}X_{t}\in L^{2}(\Omega;\mathbb{R}^{d})$.

Thus obtained $(X_{t})_{0\leq t\leq  T_e}$ is continuous from $[0,T_e]$ to $L^{2}(\Omega;\mathbb{R}^{d})$, but not necessarily path-wise a.s.~continuous on $[0,T_e]$.
It turns out that for initial data with suitable integrability, the DO approximation $X_{t}$  actually admits a.s.\ Hölder continuous paths. Namely, we assume the following $\mathbb{P}$-integrability condition.
\begin{assumption}\label{assump:2k moment X0} The initial condition $\boldsymbol{X}_0$ satisfies
\begin{equation} \label{2k moment in 0}
\mathbb{E}[|\boldsymbol{X}_{0}|^{2k}] < +\infty,\qquad \text{for some}\ \  k \in \mathbb{N}.
\end{equation}
\end{assumption}
Notice that this condition is equivalent to
\begin{equation} \label{2k moment Y0}
\mathbb{E}[|\boldsymbol{Y}_{0}|^{2k}]<\infty,\qquad \text{for some}\ \  k \in \mathbb{N},
\end{equation}
where $\boldsymbol{Y}_{0}\in [L^2(\Omega)]^R$ is arbitrary such that $X_0=\boldsymbol{U}_0^\top\boldsymbol{Y}_{0}$ for $\boldsymbol{U}_0\in\mathbb{R}^{R\times d}$ orthonormal. 
Indeed, for any DO initial condition 
$\boldsymbol{U}_0^\top \boldsymbol{Y}_0=X_0$, {the orthogonality of the rows} of $\boldsymbol{U}_0$ implies
\[\mathbb{E}[|\boldsymbol{Y}_{0}|^{2k}] \leq \mathbb{E}[\left(\|\boldsymbol{U}^{\top}_{0}\|^2_2 |\boldsymbol{U}^{\top}_{0}\boldsymbol{Y}_{0}|^{2}\right)^{k}] \leq \mathbb{E}[|X_{0}|^{2k}] \leq \mathbb{E}[\left(\|\boldsymbol{U}^{\top}_{0}\|^2_2 |\boldsymbol{Y}_{0}|^{2}\right)^{k}] = \mathbb{E}[|\boldsymbol{Y}_{0}|^{2k}] < +\infty,\]
provided that the $2k$-th moment of either $\boldsymbol{Y}_0$ or $X_0$ (hence both) exists.
Thus,
$
\mathbb{E}[|X_{0}|^{2k}] = \mathbb{E}[|\boldsymbol{Y}_{0}|^{2k}] 
$. 
Similarly, for any other DO initial condition 
$\tilde{\boldsymbol{U}}_0^\top \tilde{\boldsymbol{Y}}_0=X_0$ we have \[\mathbb{E}[|\tilde{\boldsymbol{Y}}_{0}|^{2k}]=\mathbb{E}[|X_{0}|^{2k}] = \mathbb{E}[|\boldsymbol{Y}_{0}|^{2k}].\]
Hence, Assumption~\ref{assump:2k moment X0} is equivalent to \eqref{2k moment Y0}, with an arbitrarily fixed stochastic basis.

Analogously, for all $t>0$ we have
\begin{equation}\label{eq:2k-mom}
\mathbb{E}[|X_{t}|^{2k}] = \mathbb{E}[|\boldsymbol{Y}_{t}|^{2k}] < +\infty.
\end{equation}

The DO solution preserves the $\mathbb{P}$-integrability of the initial datum.
\begin{lem}[Even order moments of the solution]\label{lem: 2k moments}
Let $a$ and $b$ satisfy Assumptions~\ref{assump:Lip} and~\ref{assump:lin-growth}.
Suppose that the DO solution for \eqref{eq:true-eq}  exists on $[0,T]$. 
Suppose that Assumption~\ref{assump:2k moment X0} is fulfilled for some $k\in\mathbb{N}$. 
Then, 
for any $t\in[0,T]$ we have
\begin{equation*}
\mathbb{E}[|\boldsymbol{Y}_{t}|^{2k}]  \leq \bigl(\mathbb{E}[|\boldsymbol{Y}_{0}|^{2k}] + K_1(T)\bigr)K_2(T),
\end{equation*}
with $K_1(T):= \Bigl(3k^2 \frac{C_{\mathrm{lgb}}\,T}{(1+{1}/{C_{\mathrm{lgb}}})^{k-1}} \Bigr)$ and $K_2(T) := \exp\{6k^2C_{\mathrm{lgb}}(1+{1}/{C_{\mathrm{lgb}}})T\}$.
\end{lem}
\begin{proof}
From the It\^o formula,  $|\boldsymbol{Y}_{t}|^{2k}$ satisfies the following SDE (see also \cite[Theorem~4.5.4]{kloeden1992stochastic}):
\begin{equation*}
\begin{aligned}
	|\boldsymbol{Y}_{t}|^{2k} & = |\boldsymbol{Y}_{0}|^{2k} +\int_{0}^{t} 2k |\boldsymbol{Y}_{s}|^{2k-2} \left(\boldsymbol{U}_{s} a(s,\boldsymbol{U}_{s}^{\top}\boldsymbol{Y}_{s})\right)^{\top}\boldsymbol{Y}_{s} + \\
	& + k|\boldsymbol{Y}_{s}|^{2k-2} \mathrm{Tr}\left(\left(\boldsymbol{U}_{s} b(s,\boldsymbol{U}_{s}^{\top}\boldsymbol{Y}_{s})\right)\left(\boldsymbol{U}_{s} b(s,\boldsymbol{U}_{s}^{\top}\boldsymbol{Y}_{s})\right)^{\top}\right) \\ 
	& + k(2k-1)|\boldsymbol{Y}_{s}|^{2k-4} |\boldsymbol{Y}^{\top}_{s}\boldsymbol{U}_{s} b(s,\boldsymbol{U}_{s}^{\top}\boldsymbol{Y}_{s})|^2ds \\ 
	& + \int_{0}^{t} 2k |\boldsymbol{Y}_{s}|^{2k-2} \left( \boldsymbol{Y}_{s}
	\right)^{\top} \boldsymbol{U}_{s} b(s,\boldsymbol{U}_{s}^{\top}\boldsymbol{Y}_{s})  dW_s.
\end{aligned}
\end{equation*}
We take the expectation of both sides,  and, noting the progressive measurability of $\boldsymbol{Y}_t$, $a(t,\boldsymbol{U}_{t}^{\top}\boldsymbol{Y}_{t})$, and $b(t,\boldsymbol{U}_{t}^{\top}\boldsymbol{Y}_{t})$, use Fubini's theorem to obtain 
\begin{equation*}
\begin{aligned}
	\mathbb{E}[|\boldsymbol{Y}_{t}|^{2k}] & \leq \mathbb{E}[|\boldsymbol{Y}_{0}|^{2k}] +\int_{0}^{t}\mathbb{E}[ 2k |\boldsymbol{Y}_{s}|^{2k-2} \left(\boldsymbol{U}_{s} a(s,\boldsymbol{U}_{s}^{\top}\boldsymbol{Y}_{s})\right)^{\top}\boldsymbol{Y}_{s} ]ds + \\
	& + \int_{0}^{t}\mathbb{E}[2k^2|\boldsymbol{Y}_{s}|^{2k-2} \mathrm{Tr}\left(\left(\boldsymbol{U}_{s} b(s,\boldsymbol{U}_{s}^{\top}\boldsymbol{Y}_{s})\right)\left(\boldsymbol{U}_{s} b(s,\boldsymbol{U}_{s}^{\top}\boldsymbol{Y}_{s})\right)^{\top}\right)]ds  \\
	& \leq \mathbb{E}[|\boldsymbol{Y}_{0}|^{2k}] +\int_{0}^{t}\mathbb{E}[ k |\boldsymbol{Y}_{s}|^{2k-2} \left(| a(s,\boldsymbol{U}_{s}^{\top}\boldsymbol{Y}_{s})|^2+ |\boldsymbol{Y}_{s}|^2 \right)] ds  \\
	& + \int_{0}^{t}\mathbb{E}[2k^2|\boldsymbol{Y}_{s}|^{2k-2} \|b(s,\boldsymbol{U}_{s}^{\top}\boldsymbol{Y}_{s})\|_{\mathrm{F}}^2ds \\
	& \leq \mathbb{E}[|\boldsymbol{Y}_{0}|^{2k}] +\int_{0}^{t}\mathbb{E}[ k |\boldsymbol{Y}_{s}|^{2k-2} C_{\mathrm{lgb}}\left(1+(1+1/C_{\mathrm{lgb}})|\boldsymbol{Y}_{s}|^2\right)] ds \\
	& + \int_{0}^{t}\mathbb{E}[2k^2|\boldsymbol{Y}_{s}|^{2k-2} C_{\mathrm{lgb}}(1+|\boldsymbol{Y}_{s}|^2)]ds \\
	& \leq \mathbb{E}[|\boldsymbol{Y}_{0}|^{2k}] +\int_{0}^{t}\mathbb{E}[ k |\boldsymbol{Y}_{s}|^{2k-2} C_{\mathrm{lgb}}\left(1+(1+1/C_{\mathrm{lgb}})|\boldsymbol{Y}_{s}|^2\right)] ds \\
	& + \int_{0}^{t}\mathbb{E}[2k^2|\boldsymbol{Y}_{s}|^{2k-2} C_{\mathrm{lgb}}\left(1+(1+1/C_{\mathrm{lgb}})|\boldsymbol{Y}_{s}|^2\right)]ds \\
	& \leq \mathbb{E}[|\boldsymbol{Y}_{0}|^{2k}] +\int_{0}^{t}\mathbb{E}[ 3k^2 |\boldsymbol{Y}_{s}|^{2k-2} C_{\mathrm{lgb}}\left(1+(1+1/C_{\mathrm{lgb}})|\boldsymbol{Y}_{s}|^2\right)]ds \\
	& \leq \mathbb{E}[|\boldsymbol{Y}_{0}|^{2k}] +   3k^2\frac{C_{\mathrm{lgb}}}{\left(\sqrt{(1+{1}/{C_{\mathrm{lgb}}})}\right)^{2k-2}}  \\
	& \cdot \int_{0}^{t}\mathbb{E}[ \left(\sqrt{(1+{1}/{C_{\mathrm{lgb}}})}\right)^{2k-2}|\boldsymbol{Y}_{s}|^{2k-2} \left(1+\left(\sqrt{(1+1/C_{\mathrm{lgb}})}\right)^2|\boldsymbol{Y}_{s}|^2\right)]ds 
\end{aligned}
\end{equation*}
Then, by using the relation $(1+r^2)r^{2k-2} \leq 1+2r^{2k}$ for $r\in \mathbb{R}_{+}$, the statement follows by Gronwall's lemma:
\begin{equation*}
\begin{aligned}
	\mathbb{E}[|\boldsymbol{Y}_{t}|^{2k}] & \leq \mathbb{E}[|\boldsymbol{Y}_{0}|^{2k}] + 3k^2 \frac{C_{\mathrm{lgb}}}{\left(\sqrt{(1+{1}/{C_{\mathrm{lgb}}})}\right)^{2k-2}}\int_{0}^{t}\mathbb{E}[ (1+2(\sqrt{(1+1/C_{\mathrm{lgb}})})^{2k}|\boldsymbol{Y}_{s}|^{2k}]ds \\
	& \leq \mathbb{E}[|\boldsymbol{Y}_{0}|^{2k}] + 3k^2 \frac{C_{\mathrm{lgb}}}{\left(\sqrt{(1+{1}/{C_{\mathrm{lgb}}})}\right)^{2k-2}}T + 6k^2C_{\mathrm{lgb}}(1+1/C_{\mathrm{lgb}})\int_{0}^{t}\mathbb{E}[ |\boldsymbol{Y}_{s}|^{2k}]ds
\end{aligned}
\end{equation*}
\end{proof}
With Lemma \ref{lem: 2k moments}, we can now establish a H\"{o}lder-type bound on $\mathbb{E}[|X_t - X_{t'}|^{2k}]$ for  

\noindent $0\leq t' < t < T_e$, which by Kolmogorov-Chenstov theorem, implies the existence of an a.s. continuous version of $X_t$.\begin{prop}\label{prop: 2k differences}
Suppose that the DO approximation $X_t$ of \eqref{eq:true-eq} exists on $[0,T_e)$. 
Suppose that Assumption~\ref{assump:2k moment X0} holds for some $k \in \mathbb{N}$.
If $a$ and $b$ satisfy Assumption~\ref{assump:lin-growth}, then there exists a constant $C:=C(X_0,R,k,T_e,C_{\mathrm{lgb}}) >0$ such that
\begin{equation}\label{eq:hold-2k}
\mathbb{E}[|X_t - X_{t'}|^{2k}] \leq C(t - t')^{k}, \quad 0 \leq t' < t < T_e.
\end{equation}
\end{prop}
\begin{proof}
The proof follows closely the one in \cite[Theorem $4.5.4$]{kloeden1992stochastic}. 
First, notice that, in the view of \eqref{eq:2k-mom}, Lemma \ref{lem: 2k moments} gives an estimate of the $2k$-moments of the DO approximation $X$ for some $k \in \mathbb{N}$.  
Now we use the inequality $|r_1 + \dots + r_m|^{p} \leq m^{p-1}\sum\limits_{i=1}^{m} |r_i|^p $ for $r \in \mathbb{R}$ twice and Jensen's inequality to \eqref{eq:DLR-diff} to obtain
\begin{equation*}
\begin{aligned}
	\mathbb{E}[|X_t - X_{t'}|^{2k}] & =  
	\mathbb{E}[|  \int_{t'}^{t}\bigl(\mathscr{P}_{\mathcal{U}({X}_s)}+\mathscr{P}_{\mathcal{Y}({X}_s)}-\mathscr{P}_{\mathcal{U}({X}_s)}\mathscr{P}_{\mathcal{Y}({X}_s)}\bigr)a(s,X_{s})\,ds \\
	&\qquad\qquad+ \int_{t'}^{t} \mathscr{P}_{\mathcal{U}(X_s)} b(s, X_s) dW_s|^{2k}]\\     
	& \leq 2^{2k-1}   \mathbb{E}[|
	\int_{t'}^{t}\bigl(\mathscr{P}_{\mathcal{U}({X}_s)}+\mathscr{P}_{\mathcal{Y}({X}_s)}-\mathscr{P}_{\mathcal{U}({X}_s)}\mathscr{P}_{\mathcal{Y}({X}_s)}\bigr)a(s,X_{s})\,ds
|^{2k}] \\
	&\qquad\qquad + 2^{2k-1}  \mathbb{E}[|\int_{t'}^{t}\mathscr{P}_{\mathcal{U}(X_s)}   b(s, X_s) dW_s |^{2k}] \\
	& \leq 2^{2k-1} (t-t')^{2k-1}  \mathbb{E}[\int_{t'}^{t}  \left(\| \mathscr{P}_{\mathcal{U}(X_s)}+\mathscr{P}_{\mathcal{Y}(X_s)} -\mathscr{P}_{\mathcal{U}(X_s)}\mathscr{P}_{\mathcal{Y}(X_s)} \|^{2}_2 |a(s, X_s) |^{2}\right)^k ds] \\
	&\qquad\qquad+ 2^{2k-1}  \mathbb{E}[|\int_{t'}^{t} \mathscr{P}_{\mathcal{U}(X_s)} b(s, X_s) dW_s |^{2k} ]\\
	& \leq 2^{2k-1} (t-t')^{2k-1}  \mathbb{E}[\int_{t'}^{t} |a(s, X_s) |^{2k} ds] + 2^{2k-1}  \mathbb{E}[|\int_{t'}^{t} \mathscr{P}_{\mathcal{U}(X_s)} b(s, X_s) dW_s |^{2k} ]\\
\end{aligned}
\end{equation*}

For the first term in the last inequality above, using
the inequality $(1+r^2)^k \leq 2^{k-1}(1+r^{2k})$ for $r\in \mathbb{R}_{+}$, Assumption~\ref{assump:lin-growth}
and Lemma \ref{lem: 2k moments}, we get
\begin{equation}\label{bound moments a}
\begin{aligned}
	\mathbb{E}[\int_{t'}^{t} |a(s, X_s) |^{2k} ds ]& \leq  \mathbb{E}[\int_{t'}^{t} C_{\mathrm{lgb}}^{k} (1+ |X_s |^{2})^k ds ]\\
	&  \leq  C_{\mathrm{lgb}}^{k} 2^{k-1} \int_{t'}^{t}  (1+ \mathbb{E}[|\boldsymbol{U}_{s}^{\top}\boldsymbol{Y}_{s} |^{2k}]) ds \\
	& \leq C_{\mathrm{lgb}}^{k} 2^{k-1} \int_{t'}^{t}  (1+ \mathbb{E}[|\boldsymbol{Y}_{s} |^{2k}]) ds \\
	& \leq C_{\mathrm{lgb}}^{k} 2^{k-1} \int_{t'}^{t}  \left(1+  \left(\mathbb{E}[|\boldsymbol{Y}_{0}|^{2k}] + K_1(T_e) \right)K_2(T_e) \right)ds \\
	& = C_{\mathrm{lgb}}^{k} 2^{k-1} \left(1+  \left(\mathbb{E}[|X_{0}|^{2k}] + K_1(T_e) \right)K_2(T_e) \right)(t-t').
\end{aligned}
\end{equation}
For the second term, we will show
\begin{equation}\label{ito inequality 2k}
\mathbb{E}[|\int_{t'}^{t} \mathscr{P}_{\mathcal{U}(X_s)} b(s,X_s)dW_s|^{2k}] \leq (t - t')^{k-1}[k(2k-1)]^k \int_{t'}^{t}\mathbb{E}[\| b(s,X_s)\|_{\mathrm{F}}^{2k}]ds.
\end{equation}
To see this, let 
$I_t := \int_{t'}^{t} \mathscr{P}_{\mathcal{U}(X_s)} b(s,X_s)dW_s$. Then the Itô formula implies
\begin{equation} \label{eq: 2k-der of I_t}
d |I_t |^{2k} = 2k |I_t |^{2k-1} \mathrm{sgn}(I_t) \mathscr{P}_{\mathcal{U}(X_t)}b(t,X_t)dW_t + \frac{1}{2} 2k(2k-1) |I_t |^{2k-2}\|\mathscr{P}_{\mathcal{U}(X_t)}b(t,X_t)\|_{\mathrm{F}}^2dt.
\end{equation}
Taking the expectation of both sides, by Hölder's inequality we get
\begin{equation*}
\begin{aligned}
	\mathbb{E}[|\int_{t'}^{t}\mathscr{P}_{\mathcal{U}(X_s)} b(s,X_s)dW_s|^{2k}]=\mathbb{E}[|I_t|^{2k}]
	& = 0 + k(2k-1) \int_{t'}^{t}\mathbb{E}[|I_s |^{2k-2}\|\mathscr{P}_{\mathcal{U}(X_s)}b(s,X_s)\|_{\mathrm{F}}^2]ds \\
	& \leq k(2k-1)\int_{t'}^{t}\mathbb{E}[|I_s |^{2k}]^{1-\frac{1}{k}} \mathbb{E}[\|b(s,X_s)\|_{\mathrm{F}}^{2k}]^{\frac{1}{k}}ds 
\end{aligned}
\end{equation*}
By Lemma \ref{lem: 2k moments}, we have $\int_{t'}^t \mathbb{E}|X_s|^{2k}ds=\int_{t'}^t \mathbb{E}|\boldsymbol{Y}_s|^{2k}ds<\infty$ and, hence, by $\|A B\|_F \leq \|A\|_2\|B\|_F$ and \eqref{eq:lin-growth}, it holds $\mathbb{E}[|\int_{t'}^{t} \mathscr{P}_{\mathcal{U}(X_s)} b(s,X_s)dW_s|^{2k}]<\infty$. 

This implies that $|I_\cdot|^{2k}=|\int_{t'}^{\cdot} \mathscr{P}_{\mathcal{U}(X_s)} b(s,X_s)dW_s|^{2k}$ is a submartingale. Therefore, it follows
\begin{equation*}
\begin{aligned}
	\mathbb{E}[|\int_{t'}^{t} \mathscr{P}_{\mathcal{U}(X_s)} b(s,X_s)dW_s|^{2k}] & \leq k(2k-1) \mathbb{E}[|I_t |^{2k}]^{1-\frac{1}{k}} \int_{t'}^{t}\mathbb{E}[ \| b(s,X_s)\|_{\mathrm{F}}^{2k}]^{\frac{1}{k}}ds.
\end{aligned}
\end{equation*}
Raising both sides to the $k$-th power and dividing them by $\mathbb{E}[|I_t |^{2k}]^{k-1}$, we get
\begin{equation*}
\mathbb{E}[|\int_{t'}^{t} \mathscr{P}_{\mathcal{U}(X_s)} b(s,X_s)dW_s|^{2k}] \leq [k(2k-1)]^k \left(\int_{t'}^{t} \mathbb{E}[\| b(s,X_s)\|_{\mathrm{F}}^{2k}]^{\frac{1}{k}}ds\right)^k,
\end{equation*}
and thus \eqref{ito inequality 2k} follows by Jensen's inequality. 
Therefore, bounding $\int_{t'}^{t} \mathbb{E}[\| b(s,X_s)\|_{\mathrm{F}}^{2k}]ds$ similarly to \eqref{bound moments a} we obtain
\begin{equation*}
\mathbb{E}[|\int_{t'}^{t} \mathscr{P}_{\mathcal{U}(X_s)} b(s,X_s)dW_s|^{2k}] \leq (t - t')^{k}[k(2k-1)]^k C_{\mathrm{lgb}}^{k} 2^{2k-1} \left(1+  \left(\mathbb{E}[|X_{0}|^{2k}] + K_1(T_e) \right)K_2(T_e) \right).
\end{equation*}
Putting all together, \eqref{eq:hold-2k} follows:
\begin{equation}
\begin{aligned}
	\mathbb{E}[|X_t - X_{t'}|^{2k}] & \leq 2^{2k-1} (t-t')^{2k-1}  \mathbb{E}[\int_{t'}^{t} |a(s, X_s) |^{2k} ds] + 2^{2k-1}  \mathbb{E}[|\int_{t'}^{t} \mathscr{P}_{\mathcal{U}(\mathcal{X}_s)} b(s, X_s) dW_s |^{2k} ] \\
	& \leq 2^{4k-2} (t-t')^{2k}  C_{\mathrm{lgb}}^{k} \left(1+  \left(\mathbb{E}[|X_{0}|^{2k}] + K_1(T_e) \right)K_2(T_e) \right) \\
	& + 2^{4k-2} (t - t')^{k}[k(2k-1)]^k C_{\mathrm{lgb}}^{k} \left(1+  \left(\mathbb{E}[|X_{0}|^{2k}] + K_1(T_e) \right)K_2(T_e) \right) \\
	&  \leq (t - t')^{k} \left(T_e^{k} + [k(2k-1)]^k\right) \left[ 2^{4k-2}C_{\mathrm{lgb}}^{k} \left(1 + \left(\mathbb{E}[|X_{0}|^{2k}] + K_1(T_e) \right) K_2(T_e) \right)\right].
\end{aligned}
\end{equation}
\end{proof}

Having proved these results concerning boundedness of $2k$-moments of the DO solution, we are now ready to extend it up to the explosion time $T_e$ with a.s. continuous paths.
\begin{thm}\label{thm:extension}
Let $a$ and $b$ satisfy Assumptions~\ref{assump:Lip} and~\ref{assump:lin-growth}.
Let an DO approximation $X$ on $[0,T_e)$ be given.
If $X_0$ satisfies Assumption~\ref{assump:2k moment X0} for some integer $k > 1$, then 
$X$ admits a unique continuous extension to $[0,T_e]$, almost surely. This extension is Hölder continuous on $[0,T_e]$, and satisfies $\mathbb{E}[|X_t|^{2k}]<\infty$ for all $t\in[0,T_e]$.
\end{thm}
\begin{proof}
By Assumption~\ref{assump:2k moment X0}, 
Proposition~\ref{prop: 2k differences} holds and $(X_{t})_{0\leq t< T_e}$ admits an extension $X_{T_e}:=\lim_{t\uparrow T_e}X_{t}\in L^{2k}(\Omega;\mathbb{R}^{d})$ that is unique in $L^{2k}(\Omega;\mathbb{R}^{d})$.
From $\|X_{t}-X_{t'}\|_{L^2(\Omega;\mathbb{R}^d)}\leq \|X_{t}-X_{t'}\|_{L^{2k}(\Omega;\mathbb{R}^d)}$ for $0\leq t' < t < T_e$, the extension is unique in $L^{p}(\Omega;\mathbb{R}^{d})$ for $p\in[2,2k]$. 
By construction of $X_{T_e}$, we have
\begin{equation}
\mathbb{E}[|X_t - X_{t'}|^{2k}] \leq C(t - t')^{k}, \quad 0 \leq t' < t \leq T_e.
\end{equation} 
Therefore, by Kolmogorov-Chenstov Theorem (see for example \cite{karatzas2012brownian}) there exists a version $(\tilde{X}_{t})_{t\in[0,T_e]}$ of the DO approximation that is $q$-Hölder continuous for all $0 < q < \frac{k-1}{2k}$. 
By construction, $(\tilde{X}_{t})_{t\in[0,T_e]}$ is a process with bounded $2k$-moments. 

Finally, let us see that $(X_t)_{t\in[0,T_e]}$ and $(\tilde{X}_t)_{t\in[0,T_e]}$ are indistinguishable. 
Indeed, by construction of $\tilde{X}$ we have $X_{T_e}=\tilde{X}_{T_e}$ a.s. 
Moreover, 
$(X_t)_{t\in[0,T_e)}$ and $(\tilde{X}_t)_{t\in[0,T_e)}$ have a.s. continuous paths and are versions of each other, and thus indistinguishable. Hence, $(X_t)_{t\in[0,T_e]}$ and $(\tilde{X}_t)_{t\in[0,T_e]}$ are indistinguishable. This completes the proof.
\end{proof}
Together with Proposition~\ref{prop: max T_e}, Theorem~\ref{thm:extension} gives us an insight into how to continue a DO approximation \emph{beyond} the explosion time $T_e$. Let a DO approximation $X$ of \eqref{eq:true-eq} with explosion time $T_e$ that satisfies \eqref{2k moment in 0} be given. Suppose $\mathrm{dim}\left(\mathrm{Im}(\mathbb{E}[\mathcal{X}_t\,\cdot\,])\right)=R$ for all $t \in [0,T_e)$, with a positive integer $R$. By  Theorem~\ref{thm:extension}, $X$ can be continuously extended up to $T_e$, while from Proposition~\ref{prop: max T_e} we know $\lim\limits_{t\uparrow T_e}\|{C}_{\boldsymbol{Y}_{t}}^{-1}\|_{\mathrm{F}}=\infty$. This implies that $\mathrm{dim}\left(\mathrm{Im}(\mathbb{E}[\mathcal{X}_{T_e}\,\cdot\,])\right)=R'<R$. Therefore, one can extend our DO approximation $X$ continuously in $t$ beyond $T_e$ by considering the DO system \eqref{eq:DLR-eq-U}--\eqref{eq:DLR-eq-Y} with initial datum $X_{T_e}$. The corresponding DO solution satisfies 
$\boldsymbol{U}_t \in \mathbb{R}^{R'\times d}$ and  $\boldsymbol{Y}_t \in [L^2(\Omega)]^{R'}$  for 
$t\in[T_e,T_e+T'_e)$, where $T'_e$ is the new explosion time for $(\boldsymbol{U}_t,\boldsymbol{Y}_t)$.

\subsection{The case of uniformly positive-definite diffusion} \label{sec: cov}
In the previous section, we have shown that by assuming the boundedness of the $2k$ moments of the initial condition of the DO approximation for $k > 1$, the DO approximation can be extended up to the explosion time $T_e$. 
It turns out that, under the condition that we will introduce in the following, the explosion time $T_e$ is never finite; as a result,  rank-$R$ DO solution exists globally.

A sufficient condition is that the diffusion matrix $b(t, x)b(t, x)^{\top}$ is positive definite with a lower bound on the smallest eigenvalue, where the bound is uniform in $t$ and $x$. This condition turns out to ensure that the smallest eigenvalue of the Gram matrix $\mathbb{E}[\boldsymbol{Y}_{t}\boldsymbol{Y}^{\top}_{t}]$ remains bounded below uniformly in $t$ and $x$, which in turn guarantees the global existence of the DO solution.
In the following, we use the notation $A \succ B$ (respectively  $A \succeq B$) with $A,B$ square matrices to indicate that $A-B$ is positive definite (respectively positive semidefinite). 

\begin{prop} \label{prop:covariance full-rank}
Suppose the rank $R$ DO solution $(\boldsymbol{U},\boldsymbol{Y})$ exists on $[0,T]$. Assume moreover that there exist $\sigma_{\boldsymbol{Y}_{0}},\sigma_B >0$ such that
\begin{align}
C_{\boldsymbol{Y}_{0}}&:= \mathbb{E}[\boldsymbol{Y}_{0}\boldsymbol{Y}_{0}^{\top}] \succeq \sigma_{\boldsymbol{Y}_{0}} \, I_{R \times R}\label{eq: sigma_0};\\
b(t, x)b(t, x)^{\top}&\succeq\sigma_{B}\, I_{d \times d}, \ \text{ for any }\ t \in [0,T]\text{ and for any }\ x \in \mathbb{R}^d. \label{eq: sigma_B}
\end{align}
Then for all  $t \in [0,T]$ we have
\begin{equation}
C_{\boldsymbol{Y}_{t}}\succeq \min\{ \sigma_{\boldsymbol{Y}_{0}}; \frac{\sigma^2_{B}}{4C_{\mathrm{lgb}}(1+M)} \} \, I_{R \times R} ,
\end{equation}
where $M=M(T)$ is defined in Lemma \ref{lem:stab}.
\end{prop}
\begin{proof}
By introducing the shorthand notation $a_t = a(t, X_t) \in \mathbb{R}^d$ and $b_t = b(t, X_t) \in \mathbb{R}^{d \times m}$, the $k$-th component ${Y}_{t}^{k}$ of $\boldsymbol{Y}_t\in [L^2(\Omega)]^R$ satisfies
\begin{equation*}
d{Y}_{t}^{k} = \sum\limits_{i=1}^{d} U^{ki}_t a^{i}_t dt + \sum\limits_{l = 1}^{m}\sum\limits_{r=1}^{d} {U}_{t}^{kr}b_t^{rl} dW_t^{l} 
\end{equation*}
Hence, 
using Itô's formula for $1\leq j,k \leq R$, we have 
\begin{equation*}
\begin{aligned}
	d({Y}_{t}^{j}{Y}_{t}^{k})  = & d({Y}_{t}^{j}){Y}_{t}^{k} +  {Y}_{t}^{j}d({Y}_{t}^{k}) + \sum_{l=1}^{m} \sum_{i=1}^{d} {U}_{t}^{ji}b_t^{il} \sum_{r=1}^{d} {U}_{t}^{kr}b_t^{rl} dt \\
	 = & (\sum\limits_{i=1}^{d} U^{ji}_t a^{i}_t dt + \sum\limits_{l = 1}^{m}\sum\limits_{i=1}^{d} {U}_{t}^{ji}b_t^{il} dW_t^{l}){Y}_{t}^{k} \\
	& +  {Y}_{t}^{j}(\sum\limits_{i=1}^{d} U^{ki}_t a^{i}_t dt + \sum\limits_{l = 1}^{m}\sum\limits_{i=1}^{d} {U}_{t}^{kr}b_t^{rl} dW_t^{l}) \\
	& + \sum_{l=1}^{m} \sum_{i=1}^{d} {U}_{t}^{ji}b_t^{il} \sum_{r=1}^{d} {U}_{t}^{kr}b_t^{rl} dt
\end{aligned}
\end{equation*}
Hence, 
\begin{equation*}
d(\boldsymbol{Y}_{t}\boldsymbol{Y}_{t}^{\top}) = (d \boldsymbol{Y}_{t})\boldsymbol{Y}_{t}^{\top} + \boldsymbol{Y}_{t} (d \boldsymbol{Y}_{t})^{\top} + \boldsymbol{U}_{t} b_t (\boldsymbol{U}_{t} b_t)^{\top} dt
\end{equation*}
and taking the expectation of both sides yields
\begin{equation*}
\frac{d\mathbb{E}[(\boldsymbol{Y}_{t}\boldsymbol{Y}_{t}^{\top})]}{dt} = \mathbb{E}[ \boldsymbol{U}_{t}a_t(\boldsymbol{Y}_{t}^{\top})] + \mathbb{E}[\boldsymbol{Y}_{t} (\boldsymbol{U}_{t}a_t)^{\top}] + \mathbb{E}[\boldsymbol{U}_{t} b_t b_t^{\top} \boldsymbol{U}_{t}^{\top}].
\end{equation*}
We now aim at analyzing the smallest eigenvalue of $C_{\boldsymbol{Y}_{t}}:=\mathbb{E}[(\boldsymbol{Y}_{t}\boldsymbol{Y}_{t}^{\top})]$ through the Rayleigh quotient. 
For any unit vector $v\in \mathbb{R}^{R}$, we have
\begin{equation}\label{rayleigh quot continuous}
v^{\top}\frac{d\mathbb{E}[(\boldsymbol{Y}_{t}\boldsymbol{Y}_{t}^{\top})]}{dt}v = 2 v^{\top}\mathbb{E}[ \boldsymbol{U}_{t}a_t(\boldsymbol{Y}_{t}^{\top})]v  + v^{\top}\mathbb{E}[\boldsymbol{U}_{t} b_t b_t^{\top} \boldsymbol{U}_{t}^{\top}]v.
\end{equation}
Thanks to \eqref{eq: sigma_B},  the last term can be bounded below as
\begin{equation*}
v^{\top}\mathbb{E}[\boldsymbol{U}_{t} b_t b_t^{\top} \boldsymbol{U}_{t}^{\top}]v = v^{\top}\boldsymbol{U}_{t}\mathbb{E}[ b_t b_t^{\top}] \boldsymbol{U}_{t}^{\top}v \geq v^{\top}\boldsymbol{U}_{t}\sigma_B I_{R \times R} \boldsymbol{U}_{t}^{\top}v \geq \sigma_B.
\end{equation*}   
The first term can be bounded above as
\begin{equation*}
\begin{aligned}
	|v^{\top}\mathbb{E}[ \boldsymbol{Y}_{t} (\boldsymbol{U}_{t}a_t)^{\top}]v |  & = |\mathbb{E}[ v^{\top}\boldsymbol{Y}_{t} a_t^{\top}] \boldsymbol{U}_{t}^{\top}v | \leq |\mathbb{E}[ v^{\top}\boldsymbol{Y}_{t} a_t^{\top}]| \\
	& \leq  \frac{1}{2 \varepsilon}\mathbb{E}[|v^{\top}\boldsymbol{Y}_{t}|^2]+ \frac{\varepsilon}{2}\mathbb{E}[|a_t|^2] \\
	& \leq  \frac{1}{2 \varepsilon}\mathbb{E}[v^{\top}\boldsymbol{Y}_{t}\boldsymbol{Y}_{t}^{\top}v]+ \frac{\varepsilon}{2}C_{\mathrm{lgb}}(1+\mathbb{E}[|\boldsymbol{Y}_{t}|^2] ) \\
	& \leq  \frac{1}{2 \varepsilon}v^{\top}\mathbb{E}[\boldsymbol{Y}_{t}\boldsymbol{Y}_{t}^{\top}]v+ \frac{\varepsilon}{2}C_{\mathrm{lgb}}(1+M )\ \text{ for any } \varepsilon >0.
\end{aligned}
\end{equation*}   
Taking $\varepsilon = \frac{\sigma_{B}}{2C_{\mathrm{lgb}}(1+M)}$ leads to the following estimate on the derivative of $A_t := v^{\top}C_{\boldsymbol{Y}_{t}}v = v^{\top}\mathbb{E}[(\boldsymbol{Y}_{t}\boldsymbol{Y}_{t}^{\top})]v$:

\begin{equation}\label{eq:A_t}
\frac{d}{dt}A_t \geq -\frac{1}{ \varepsilon}A_t + \frac{\sigma_B}{2},
\end{equation}
from which we deduce
\begin{equation*}
A_t \geq \frac{\varepsilon\sigma_B}{2} (e^{\frac{t}{ \varepsilon}} -1) e^{-\frac{t}{ \varepsilon}} + A_0 e^{-\frac{t}{\varepsilon}}.
\end{equation*}
Noting that \eqref{eq: sigma_0} implies $A_0 \geq \sigma_{\boldsymbol{Y}_{0}}$, we conclude
\begin{equation*}
\begin{aligned}
	v^{\top}\mathbb{E}[\boldsymbol{Y}_{t}\boldsymbol{Y}_{t}^{\top}]v & \geq \frac{\varepsilon\sigma_B}{2} + (\sigma_{\boldsymbol{Y}_{0}} - \frac{\varepsilon\sigma_B}{2}) e^{-\frac{t}{ \varepsilon}} \\
	& \geq \min\{ \sigma_{\boldsymbol{Y}_{0}}, \frac{\sigma^2_{B}}{4C_{\mathrm{lgb}}(1+M)} \} > 0 \quad \ \text{ for any }\ t \in [0,T].
\end{aligned}
\end{equation*}   
\end{proof}
As a consequence of the proposition above, the following global existence result is obtained.
\begin{thm}[Global Existence of DO solution]\label{thm:global existence of DLR}
Let Assumptions 1--3 hold. Suppose that the assumptions of Proposition~\ref{prop:covariance full-rank} hold. Then, the DO solution exists for all $t \geq 0$.
\end{thm}
\begin{proof}
From Theorems \ref{thm:local-existence} and \ref{thm: uniqueness of DO}, the rank $R$ solution uniquely exists up to a certain time $T>0$. 
Denote by $T_e$ its explosion time and suppose, to obtain a contradiction, $T_e < +\infty$. 
Then, from Proposition~\ref{prop: max T_e} we have $\lim\limits_{t \uparrow T_e} \|C^{-1}_{\boldsymbol{Y}_t}\|_{\mathrm{F}} = +\infty$. 
But by Proposition \ref{prop:covariance full-rank} the Rayleigh quotient $A_t = v^{\top}C_{\boldsymbol{Y}_{t}}v$ with $v\in \mathbb{R}^{R}$ satisfies
\begin{equation*}
A_t \geq \min\{ \sigma_{\boldsymbol{Y}_{0}}; \frac{\sigma^2_{B}}{4C_{\mathrm{lgb}}(1+M(T_e))} \} > 0,
\end{equation*}
and, hence, for some constant $\bar{k} > 0$ we have $\|C^{-1}_{\boldsymbol{Y}_t}\|_{\mathrm{F}} \leq \bar{k}$ for all $ t \in [0, T_e)$. This contradicts $\lim\limits_{t \uparrow T_e} \|C^{-1}_{\boldsymbol{Y}_t}\|_{\mathrm{F}} = +\infty$. Therefore, $T_e = +\infty$.
\end{proof}
\begin{Remark}
It is worth
noticing that \eqref{eq: sigma_B} is satisfied in the case of additive non-degenerate noise.
\end{Remark}

\section{Conclusion}\label{sec:conclusion}
In this work, we achieved to set a rigorous DO setting for SDEs under the conditions that the studied drift and diffusion satisfy a Lipschitz condition and a linear-growth bound. First, we (re-)derived the equations which characterize the evolution of the deterministic and stochastic modes, in a DO formulation, and showed how these can be re-interpreted as a projected dynamics leading to a DLRA formulation. Our derivation makes use of the It\^o's formula and avoids the direct use of time derivatives. We proved local-existence and uniqueness of the DO formulation and analyzed the possibility of extending the solution up to and beyond the explosion time. Finally, we gave a sufficient condition that assures the global existence of the DLR approximation.

One natural development of this work would be to extend this DO framework and well-posedness results to accommodate weaker assumptions on the drift and diffusion (e.g. local lipschitzianity and/or weak monotonicity).

Furthermore, it would be interesting to build a DLR formulation as a fully projecting dynamics as in \eqref{eq:eq-manifold2} by giving a rigorous meaning to the term  $\mathscr{P}_{\mathcal{Y}(\mathcal{X}_t)}[b(t,\mathcal{X}_{t})\,dW_{t}]$, possibly in a distributional sense. In case one manages to achieve this goal, the direct connection to the standard DLRA formulation for deterministic or random equations would allow us to apply numerical projector splitting schemes, which have been shown to perform very well in those contexts \cite{ceruti2022unconventional,lubich2014projector}.

\section*{Acknowledgements}
Yoshihito Kazashi acknowledges the financial support of the University of Strathclyde through a Faculty of Science Starter Grant. This work has also been supported by the Swiss National Science Foundation under the
Project n. 200518 “Dynamical low rank methods for uncertainty quantification and data assimilation”.

\appendix
\section{Appendix}
For a Hilbert space $(H,\langle\cdot,\cdot\rangle)$,
denote by $[H]^{R}$ the product Hilbert space equipped with the
norm $\|\boldsymbol{Y}\|_{[H]^{R}}=\sqrt{\sum_{j=1}^{R}\|Y^{j}\|_{H}^{2}}$
for $\boldsymbol{Y}=(Y^{j})\in[H]^{R}$. For $\boldsymbol{Y}\in[H]^{R}$
let $Z_{\boldsymbol{Y}}$ be the Gram matrix 
\[
Z_{\boldsymbol{Y}}:=\left(\langle{Y}^{j},{Y}^{k}\rangle\right)_{j,k=1,\dots R}\in\mathbb{R}^{R\times R}.
\]
If $Z_{\boldsymbol{Y}_0}$ is invertible, then for $\boldsymbol{Y}$ close to $\boldsymbol{Y}_0$, $Z_{\boldsymbol{Y}}$ is also invertible. The following result makes this intuition precise. 
\begin{prop} \label{prop: A inverse}
Suppose that $\boldsymbol{Y}_{0}\in[H]^{R}$ has linearly independent
components $Y^{j}_0$, $j=1,\dots,R$ in $H$, and that $\|\boldsymbol{Y}_{0}\|_{H}\leq\rho$
and $\|Z_{\boldsymbol{Y}_{0}}^{-1}\|_{\mathrm{F}}\leq\gamma$ hold
for $\rho,\gamma>0$. Then, there exists $\eta:=\eta(\rho,\gamma)>0$ {defined as
\begin{equation}
	\eta(\rho,\gamma):=-\rho+\sqrt{\rho^{2}+\frac{1}{2\gamma}}\label{eq:def-eta}
\end{equation}
}such that we have 
\[
\|Z_{\boldsymbol{Y}}^{-1}\|_{\mathrm{F}}\leq2\gamma,\quad\text{for any }\boldsymbol{Y}\in B_{\eta}(\boldsymbol{Y}_{0}),
\]
and $\eta(\rho,\gamma)$ is decreasing in both $\rho$ and $\gamma$.
Here, $B_{\eta}(\boldsymbol{Y}_{0})$ is the open ball in $[H]^{R}$
of radius $\eta$ around $\boldsymbol{Y}_{0}$. \end{prop} 
\begin{proof}
Any $\boldsymbol{Y}\in B_{\eta}(\boldsymbol{Y}_{0})$ may be written
as $\boldsymbol{Y}=\boldsymbol{Y}_{0}+r\boldsymbol{\delta}$ with
$r<\eta$ and $\|\boldsymbol{\delta}\|_{[H]^{R}}=1$. We will derive
an upper bound $\eta$ on $r$ that guarantees $\|Z_{\boldsymbol{Y}}^{-1}\|_{\mathrm{F}}\leq2\gamma$. 

Define $\boldsymbol{Y}(s):=\boldsymbol{Y}_{0}+s\boldsymbol{\delta}$
and let $f(s):=Z_{\boldsymbol{Y}(s)}^{-1}$. Then, we have 
\[
\frac{\mathrm{d}}{\mathrm{d}s}f(s)=-f(s)\left(\underbrace{(\langle\delta^{j},{Y(s)}^{k}\rangle)_{j,k=1,\dots R}}_{:=Z_{\boldsymbol{\delta},\boldsymbol{Y}(s)}}+\underbrace{(\langle{Y(s)}^{j},\delta^{k}\rangle)_{j,k=1,\dots R}}_{:=Z_{\boldsymbol{Y}(s),\boldsymbol{\delta}}}\right)f(s).
\]
Therefore, it holds 
\[
\begin{aligned}\frac{\mathrm{d}}{\mathrm{d}s}\|f(s)\|_{\mathrm{F}}^{2}  =&\frac{\mathrm{d}}{\mathrm{d}s}\mathrm{Tr}(f(s)f^{\top}(s))\\
 =&-\mathrm{Tr}\bigl(f(s)(Z_{\boldsymbol{\delta},\boldsymbol{Y}(s)}+Z_{\boldsymbol{Y}(s),\boldsymbol{\delta}})f(s)f^{\top}(s)\bigr)-\mathrm{{Tr}}\bigl(f(s)f^{\top}(s)(Z_{\boldsymbol{Y}(s),\boldsymbol{\delta}}+Z_{\boldsymbol{\delta},\boldsymbol{Y}(s)})f^{\top}(s)\bigr)\\
 \leq&2\|Z_{\boldsymbol{\delta},\boldsymbol{Y}(s)}+Z_{\boldsymbol{Y}(s),\boldsymbol{\delta}}\|_{\mathrm{F}}\|f^{3}(s)\|_{\mathrm{F}}\\
 \leq&2(\|Z_{\boldsymbol{\delta},\boldsymbol{Y}(s)}\|_{\mathrm{F}}+\|Z_{\boldsymbol{Y}(s),\boldsymbol{\delta}}\|_{\mathrm{F}})\|f^{3}(s)\|_{\mathrm{F}}\\
 \leq&4(\|\boldsymbol{\delta}\|_{[H]^{R}}\|\boldsymbol{Y}(s)\|_{[H]^{R}})\|f(s)\|_{\mathrm{F}}^{3}\\
 \leq&4(\rho+s)\|f(s)\|_{\mathrm{F}}^{3}.
\end{aligned}
\]
Let $E_{s}=\|f(s)\|_{\mathrm{F}}^{2}$ so that $E_{s}^{-\frac{3}{2}}\frac{\mathrm{d}E_{s}}{\mathrm{d}s}\leq4(\rho+s)$
and
\[
\int_{0}^{r}2\frac{\mathrm{d}(-E_{s}^{-1/2})}{\mathrm{d}s}\mathrm{d}s=\int_{0}^{r}E_{s}^{-\frac{3}{2}}\frac{\mathrm{d}E_{s}}{\mathrm{d}s}\mathrm{d}s\leq4\rho r+2r^{2},
\]
which implies $-2E_{r}^{-\frac{1}{2}}+2E_{0}^{-\frac{1}{2}}\leq4\rho r+2r^{2}$.
Now we use $E_{0}\leq\gamma^{2}$ to obtain the bound $\gamma^{-1}-2\rho r-r^{2}\leq E_{r}^{-\frac{1}{2}}$.
For any $0<r<-\rho+\sqrt{\rho^{2}+\frac{1}{\gamma}}$, we have $\frac{1}{\gamma}-2\rho r-r^{2}>0$
and thus this bound yields
\[
\|Z_{\boldsymbol{Y}_{r}}^{-1}\|_{\mathrm{F}}\leq\frac{1}{\frac{1}{\gamma}-2\rho r-r^{2}}=\frac{1}{\frac{1}{\gamma}+\rho^{2}-(\rho+r)^{2}}.
\]
Then, $\|Z_{\boldsymbol{Y}_{r}}^{-1}\|_{\mathrm{F}}^{2}\leq4\gamma^{2}$
is guaranteed by the condition $0<r\leq-\rho+\sqrt{\rho^{2}+\frac{1}{2\gamma}}$,
which also implies $\frac{1}{\gamma}-2\rho r-r^{2}>0$ above. Finally, {the function $\eta(\rho,\gamma)$, defined as in \eqref{eq:def-eta}}, is decreasing
in $\rho$ and $\gamma$:
\[
\begin{aligned}\frac{\partial\eta}{\partial\rho} & =\frac{\rho}{\sqrt{\rho^{2}+\frac{1}{2\gamma}}}-1<0,\quad\forall\rho,\gamma>0;\\
\frac{\partial\eta}{\partial\gamma} & =-\frac{1}{4\gamma^{2}}\frac{1}{\sqrt{\rho^{2}+\frac{1}{2\gamma}}}<0,\quad\forall\rho,\gamma>0.
\end{aligned}
\]	
\end{proof}

{Recall that a finite rank function $X \in L^2\left(\Omega; H\right)$ admits a representation reminiscent of the singular value decomposition as in \eqref{eq:A1} below; see \cite[Lemma 2.1]{kazashi2021existence}, as well as Lemma \ref{lem:SVD-finite-rank-op} and Remark 
	\ref{rem: rank of span X}.}
{For this decomposition, in Proposition \ref{prop: lip of the proj} we show that projector-valued mappings, where the projection is onto the space spanned by the spatial basis of this decomposition, are Lipschitz continuous. In order to achieve this goal, we need the following Lemma \ref{lem: PuPu^}. The proofs of Lemma \ref{lem: PuPu^} and Proposition \ref{prop: lip of the proj} follow closely \cite[Proof of Lemma A.2]{bachmayr2021existence}; see also \cite[Lemmata 4.1 and 4.2]{wei2016guarantees} for a finite dimensional version. However, unlike \cite[Lemma A.2]{bachmayr2021existence} we do not assume $L^2(\Omega;H)$ to be separable.} \begin{lem}\label{lem: PuPu^}
{	Consider $X, \hat{X} \in L^2(\Omega; H)$ that have the following representations
	\begin{equation}\label{eq:A1}
		X = \sum_{j=1}^{R} \sigma_j U_j V_j, \quad \hat{X} = \sum_{j=1}^{R} \hat{\sigma}_j \hat{U}_j \hat{V}_j,
	\end{equation}
	with $\sigma_j,\hat{\sigma}_j>0$, and $ U_j, \hat{U}_j \in H$, $V_j, \hat{V}_j \in L^2(\Omega; \mathbb{R})$ all orthonormal in their respective Hilbert spaces, for $j=1,\dots,R$.
	Here, $\sigma_j$ and $\hat{\sigma}_j$, $j=1,\dots,R$, are ordered in the descending order. 
	Define the projections $P_{U} = \sum_{j=1}^{R} \langle U_j, \cdot \rangle_{H} U_j : L^{2}(\Omega;H) \to L^{2}(\Omega;H)$, $P_{V} = \sum_{j=1}^{R} \mathbb{E}[V_j\,\cdot\,] V_j : L^2(\Omega; H) \to L^2(\Omega; H)$, $P_{U^{\perp}}:=\mathrm{id}-P_{U}$, and $P_{V^{\perp}}:=\mathrm{id}-P_{V}$; and similarly for $P_{\hat{U}}$, $P_{\hat{U}^{\perp}}$,
	$P_{\hat{V}}$, and $P_{\hat{V}^{\perp}}$.
	Then, we have
	\begin{equation}\label{eq: P_UP_u^}
		\begin{aligned}
		\|P_{U}-P_{\hat{U}}\|_{L^{2}(\Omega;H)\to L^{2}(\Omega;H)}&=\|P_{\hat{U}^{\perp}}P_{U}\|_{L^{2}(\Omega;H)\to L^{2}(\Omega;H)}=\|P_{\hat{U}}P_{U^{\perp}}\|_{L^{2}(\Omega;H)\to L^{2}(\Omega;H)}; \\
	\|P_{V}-P_{\hat{V}}\|_{L^{2}(\Omega;H)\to L^{2}(\Omega;H)}&=\|P_{\hat{V}^{\perp}}P_{V}\|_{L^{2}(\Omega;H)\to L^{2}(\Omega;H)}=\|P_{\hat{V}}P_{V^{\perp}}\|_{L^{2}(\Omega;H)\to L^{2}(\Omega;H)}.
 \end{aligned}
\end{equation}}
\end{lem}
\begin{proof}
{	The mappings $P_{U}$ and $P_{V}$
	are orthogonal projections. Indeed, 
	the idempotency of $P_U$ and $P_V$ can be easily checked and the $L^2(\Omega;H)$-orthogonality of $P_U$ follows from its $H$-orthogonality. 
	To show the orthogonality of $P_V$, first consider $z,w\in L^2(\Omega;H)$. Noting that $\langle\mathbb{E}[V_{k}w],\cdot\rangle_{H}\colon H\to H$ is a continuous linear functional on $H$
	for $k=1,\dots,R$, from \cite[Theorem 8.13]{Leoni.G_2017_book_Sobolev_2nd} we have:
		\begin{equation*}
		\begin{aligned}
			\langle z, P_V w \rangle_{L^2(\Omega; H)}	& = \mathbb{E}[\langle z, \sum_{k=1}^{R} \mathbb{E}[ w V_k ] V_k \rangle_H]  = \sum_{k=1}^{R} \mathbb{E}[\langle z, \mathbb{E}[ w V_k ] V_k \rangle_H] \\
& =\sum_{k=1}^{R} \langle \mathbb{E}[z V_k], \mathbb{E}[ w V_k ]  \rangle_H 
				 =
			\sum_{k=1}^{R}  \mathbb{E}[ \langle \mathbb{E}[z V_k],w V_k  \rangle_H ] = 	\langle P_V z, w \rangle_{L^2(\Omega; H)}. \\
					\end{aligned}
	\end{equation*}
Similarly, $P_{\hat{U}}$ and $P_{\hat{V}}$ are orthogonal projections. }

{Let us prove the second identity in \eqref{eq: P_UP_u^}. For $f\in L^{2}(\Omega;H)$, we have
	\begin{equation}\label{eq: Pv - Pv^}
		(P_{V}-P_{\hat{V}})f=(P_{\hat{V}}+P_{\hat{V}^{\perp}})(P_{V}-P_{\hat{V}})(P_{V}+P_{V^{\perp}})f=(P_{\hat{V}^{\perp}}P_{V}-P_{\hat{V}}P_{V^{\perp}})f,
	\end{equation}
	As $P_{\hat{V}}$ is an orthogonal projector, $\mathbb{E}[\langle P_{\hat{V}^{\perp}}f,P_{\hat{V}}g\rangle_{H}]=0$  for any $f,g \in L^{2}(\Omega;H)$. 
	Hence, together with \eqref{eq: Pv - Pv^}, for all $f\in L^{2}(\Omega;H)$ one has
	\begin{equation}\label{eq: proj f}
	\|(P_{V}-P_{\hat{V}})f\|_{L^{2}(\Omega;H)}^{2}=\|P_{\hat{V}^{\perp}}P_{V}f\|_{L^{2}(\Omega;H)}^{2}+\|P_{\hat{V}}P_{V^{\perp}}f\|_{L^{2}(\Omega;H)}^{2}.
	\end{equation}
	Moreover, for any $f \in L^{2}(\Omega;H)$ it also holds
	\begin{equation}\label{eq: f de 2}
		\|f\|_{L^{2}(\Omega;H)}^{2}=\|P_{V}f\|_{L^{2}(\Omega;H)}^{2}+\|P_{V^{\perp}}f\|_{L^{2}(\Omega;H)}^{2},
	\end{equation}
	and 
	\begin{equation}\label{eq: f de 1}
	f=\|P_{V}f\|_{L^{2}(\Omega;H)}\tilde{f}_{1}+\|P_{V^{\perp}}f\|_{L^{2}(\Omega;H)}\tilde{f}_{2}
	\end{equation}
	with unit vectors $\tilde{f}_{1}:=P_{V}f/\|P_{V}f\|_{L^{2}(\Omega;H)}$
	and $\tilde{f}_{2}:=P_{V^{\perp}}f/\|P_{V^{\perp}}f\|_{L^{2}(\Omega;H)}$. 
    From \eqref{eq: proj f} , one has
	\begin{align*}
		\|(P_{V}-P_{\hat{V}})f\|_{L^{2}(\Omega;H)}^{2} & =\|P_{V}f\|_{L^{2}(\Omega;H)}^{2}\|P_{\hat{V}^{\perp}}P_{V}\tilde{f}_{1}\|_{L^{2}(\Omega;H)}^{2}+\|P_{V^{\perp}}f\|_{L^{2}(\Omega;H)}^{2}\|P_{\hat{V}}P_{V^{\perp}}\tilde{f}_{2}\|_{L^{2}(\Omega;H)}^{2}\\
		\leq(\|P_{V}f\|_{L^{2}(\Omega;H)}^{2}&  +\|P_{V^{\perp}}f\|_{L^{2}(\Omega;H)}^{2})\max\big\{\|P_{\hat{V}^{\perp}}P_{V}\tilde{f}_{1}\|_{L^{2}(\Omega;H)}^{2},\|P_{\hat{V}}P_{V^{\perp}}\tilde{f}_{2}\|_{L^{2}(\Omega;H)}^{2}\big\}.
	\end{align*}
	Using \eqref{eq: f de 2} and $\|\tilde{f}_{1}\|_{L^{2}(\Omega;H)}=\|\tilde{f}_{2}\|_{L^{2}(\Omega;H)}=1$,
	we obtain
	\[
	\|P_{V}-P_{\hat{V}}\|_{L^{2}(\Omega;H)\to L^{2}(\Omega;H)}\leq\max\big\{\|P_{\hat{V}^{\perp}}P_{V}\|_{L^{2}(\Omega;H)\to L^{2}(\Omega;H)},\|P_{\hat{V}}P_{V^{\perp}}\|_{L^{2}(\Omega;H)\to L^{2}(\Omega;H)}\big\}.
	\]}

{To show the opposite inequality we first notice that 
		\begin{equation*}
			\begin{aligned}
			\|P_{\hat{V}^{\perp}}P_{V}\|^2_{L^{2}(\Omega;H)\to L^{2}(\Omega;H)} &=\sup\limits_{\|f\|_{L^{2}(\Omega;H)}=1}\|P_{\hat{V}^{\perp}}P_{V}f\|_{L^{2}(\Omega;H)}^{2}\\
			&=\sup\limits_{\|f\|_{L^{2}(\Omega;H)}=1,\,P_{V^{\perp}}f=0}\|P_{\hat{V}^{\perp}}P_{V}f\|_{L^{2}(\Omega;H)}^{2}.
		\end{aligned}
		\end{equation*}
		Therefore, from \eqref{eq: proj f} one has
	\begin{align*}
		\|P_{V}-P_{\hat{V}}\|_{L^{2}(\Omega;H)\to L^{2}(\Omega;H)}^{2} & \geq\sup_{\|f\|_{L^{2}(\Omega;H)}=1,\,P_{V^{\perp}}f=0}\|P_{\hat{V}^{\perp}}P_{V}f\|_{L^{2}(\Omega;H)}^{2}\\
		& =\sup_{\|f\|_{L^{2}(\Omega;H)}=1}\|P_{\hat{V}^{\perp}}P_{V}f\|_{L^{2}(\Omega;H)}^{2}=\|P_{\hat{V}^{\perp}}P_{V}\|_{L^{2}(\Omega;H)\to L^{2}(\Omega;H)}^{2}.
	\end{align*}
	Similarly, $\|P_{V}-P_{\hat{V}}\|_{L^{2}(\Omega;H)\to L^{2}(\Omega;H)}\geq\|P_{\hat{V}}P_{V^{\perp}}\|_{L^{2}(\Omega;H)\to L^{2}(\Omega;H)}.$
	Hence,
	\[
	\|P_{V}-P_{\hat{V}}\|_{L^{2}(\Omega;H)\to L^{2}(\Omega;H)}=\max\big\{\|P_{\hat{V}^{\perp}}P_{V}\|_{L^{2}(\Omega;H)\to L^{2}(\Omega;H)},\|P_{\hat{V}}P_{V^{\perp}}\|_{L^{2}(\Omega;H)\to L^{2}(\Omega;H)}\big\}.
	\]}

{Finally, let us prove that $\|P_{\hat{V}^{\perp}}P_{V}\|_{L^{2}(\Omega;H)\to L^{2}(\Omega;H)}=\|P_{\hat{V}}P_{V^{\perp}}\|_{L^{2}(\Omega;H)\to L^{2}(\Omega;H)}$. First define the vector space $H^{R} := [H]^R$ and endow it with the norm $\| z \|_{H^{R}} = \sqrt{\sum_{j=1}^{R} \|z_j\|^2_H}$ for $z=(z_j)_{j=1,\dots, R} \in H^R$. Thanks to the orthogonality of $(V_j)_j$ and the properties of orthogonal projectors, the following chain of equalities holds:
	\begin{equation}\label{eq: long chain}
		\begin{aligned}
			\|P_{\hat{V}^{\perp}}P_{V}\|^2_{L^{2}(\Omega;H)\to L^{2}(\Omega;H)}&= \sup\limits_{\|f\|_{L^{2}(\Omega;H)}=1}\Big\|P_{\hat{V}^{\perp}}\sum_{j=1}^{R}\mathbb{E}[V_{j}f]V_{j}\Big\|^2_{L^{2}(\Omega;H)} \\
			&= \sup\limits_{\|f\|_{L^{2}(\Omega;H)}=1,\,P_{V^{\perp}}f=0}\Big\|P_{\hat{V}^{\perp}}\sum_{j=1}^{R}\mathbb{E}[V_{j}f]V_{j}\Big\|^2_{L^{2}(\Omega;H)} \\
			&=\sup\limits_{x \in
				H^R, \|x\|_{H^{R}} =1}\Big\|P_{\hat{V}^{\perp}}\sum_{j=1}^{R}x_jV_{j}\Big\|^2_{L^{2}(\Omega;H)} \\
			&= \sup\limits_{x \in H^R, \|x\|_{H^{R}} =1} \Big( \Big\|\sum_{j=1}^{R}x_jV_{j}\Big\|^2_{L^{2}(\Omega;H)}  - \Big\|P_{\hat{V}}\Big(\sum_{j=1}^{R}x_jV_{j}\Big)\Big\|^2_{L^{2}(\Omega;H)}\Big)\\
			&= 1 - \inf\limits_{x \in H^R, \|x\|_{H^{R}} =1}  \Big\|\sum_{i=1}^{R}\hat{V}_{i}\mathbb{E}\Big[\hat{V}_{i}\Big(\sum_{j=1}^{R}x_jV_{j}\Big)\Big]\Big\|^2_{L^{2}(\Omega;H)} \\
			&= 1 - \inf\limits_{x \in H^R, \|x\|_{H^{R}} =1}  \sum_{i=1}^{R}\Big\|\sum_{j=1}^{R}x_j\underbrace{\mathbb{E}[\hat{V}_{i}V_{j}]}_{=:(G)_{ij}}\Big\|^2_{H} \\
			& = 1 - \inf\limits_{x \in H^R, \|x\|_{H^{R}} =1 }\left\| Gx\right\|^2_{H^{R}} .
		\end{aligned}
	\end{equation}
    Considering the Stiefel manifold $\mathrm{St}(H;R) = \{ (q_1,\dots, q_R) : q_i \in H, \ \langle q_i, q_j \rangle_H = \delta_{ij}, \ \forall i,j =1,\dots,R\}$ the following equality holds
    \begin{equation}\label{eq: small chain}
    	\begin{aligned}
    	 \inf\limits_{x \in H^R, \|x\|_{H^{R}} =1 }\| Gx\|^2_{H^{R}}& = \inf\limits_{x \in H^R, \|x\|_{H^{R}} =1}  \sum_{i=1}^{R}\Big\|\sum_{j=1}^{R}x_jG_{ij}\Big\|^2_{H}\\
    	 &=\inf\limits_{(q_1,\dots q_R) \in \mathrm{St}(H;R)} \inf\limits_{\substack{b_1,\dots,b_R \in \mathbb{R}^R \\ \sum_{i=1}^{R}\|b_i\|^2_{2} =1 }}\sum_{i=1}^{R}\Big\|\sum_{j=1}^{R}\sum_{k=1}^{R}b^{(k)}_jq_kG_{ij}\Big\|^2_{H}\\
    	 &=\inf\limits_{(q_1,\dots q_R) \in \mathrm{St}(H;R)} \inf\limits_{\substack{b_1,\dots,b_R \in \mathbb{R}^R \\ \sum_{i=1}^{R}\|b_i\|^2_{2} =1 }}\sum_{i=1}^{R}\sum_{j=1}^{R}\sum_{l=1}^{R}\sum_{k=1}^{R}b^{(k)}_lG_{il}G_{ij}b^{(k)}_j\\
    	 &= \inf\limits_{\substack{B \in \mathbb{R}^{R \times R} \\ \|B\|_F =1 }}\sum_{l=1}^{R}\sum_{j=1}^{R}\sum_{k=1}^{R} B_{kl}(G^{\top}G)_{lj}B_{kj}\\
    	 &= \inf\limits_{\substack{B \in \mathbb{R}^{R \times R} \\ \|B\|_F =1 }}\sum_{k=1}^{R} B_{k,:} G^{\top}G B_{k,:}^{\top}\\
    	 &\geq \inf\limits_{\substack{B \in \mathbb{R}^{R \times R} \\ \|B\|_F =1 }}\sum_{k=1}^{R} \lambda_R(G^{\top}G) \sum_{j=1}^R B_{k,j}^{2}= \lambda_R(G^{\top}G),  \\ 	
    	 \end{aligned}
    \end{equation}
where in the fourth line we defined $B= (B_{ij})_{ij} = (b^{(i)}_j)_{ij}\in \mathbb{R}^{R \times R}$, in the second-to-last line $B_{k,:}\in \mathbb{R}^{R}$ indicates the $k$-th row of $B$, and in the last line $\lambda_R(G^{\top}G)$ indicates the $R$-th eigenvalue of $G^{\top}G$. 
The infimum above can be attained  by taking $B \in \mathbb{R}^{R \times R}$ in \eqref{eq: small chain} such that $B^{\top}_{1,:}$ is the unit-norm  eigenvector of $G^{\top}G$ associated with $\lambda_R(G^{\top}G)$, and $B^{\top}_{j,:}= 0 \in\mathbb{R}^R$ for $j >1$. Then, $\inf\limits_{x \in H^R, \|x\|_{H^{R}} =1 }\| Gx\|^2_{H^{R}}=\lambda_R(G^{\top}G)$ independently of  
$(q_1,\dots, q_R) \in \mathrm{St}(H;R)$.  Coming back to \eqref{eq: long chain}, we get
    \begin{equation*}
    	\begin{aligned}
    		\|P_{\hat{V}^{\perp}}P_{V}\|^2_{L^{2}(\Omega;H)\to L^{2}(\Omega;H)}
    		& = 1 - \inf\limits_{x \in H^R, \|x\|_{H^{R}} =1 }\| Gx\|^2_{H^{R}} = 1 - \lambda_R(G^{\top}G)\\
    		& = 1 - \lambda_R(GG^{\top})= 1 - \inf\limits_{b \in \mathbb{R}^R, \|b\|_2=1 }\| G^{\top}b\|^2_2 \\
    		&= \sup\limits_{\|f\|_{L^{2}(\Omega;H)}=1}\|P_{V^{\perp}}P_{\hat{V}}f\|^2_{L^{2}(\Omega;H)} = 		\|P_{V^{\perp}}P_{\hat{V}}\|^2_{L^{2}(\Omega;H)\to L^{2}(\Omega;H)},
    	\end{aligned}
    \end{equation*}
    where in the last line we just follow the same discussion of \eqref{eq: long chain} and \eqref{eq: small chain} in a reverse order. 
Since $P_{V^{\perp}}$ and $P_{\hat{V}}$ are orthogonal projections and thus self-adjoint, we get $\|P_{\hat{V}^{\perp}}P_{V}\|_{L^{2}(\Omega;H)\to L^{2}(\Omega;H)}=\|P_{\hat{V}}P_{V^{\perp}}\|_{L^{2}(\Omega;H)\to L^{2}(\Omega;H)}.$
	The first equality in \eqref{eq:A1} follows similarly.}
\end{proof}
{We now come back to the following Lipschitz continuity of projector-valued mappings of finite rank functions. 
\begin{prop}
	\label{prop: lip of the proj}Suppose that $X,\hat{X}\in L^{2}(\Omega;H)$
	have the representation \eqref{eq:A1}
	with $\sigma_{j},\hat{\sigma}_{j}>0$, and $U_{j},\hat{U}_{j}\in H$,
	$V_{j},\hat{V}_{j}\in L^{2}(\Omega)$ all orthonormal in their respective
	Hilbert spaces, for $j=1,\dots,R$. Here, $\sigma_{j}$ and $\hat{\sigma}_{j}$,
	$j=1,\dots,R$, are ordered in the descending order. Then, the projections
	$P_{U}=\sum_{j=1}^{R}\langle U_{j},\cdot\rangle_{H}U_{j}\colon H\to H$
	and $P_{V}=\sum_{j=1}^{R}\mathbb{E}[V_{j}\,\cdot\,]V_{j}\colon L^{2}(\Omega;H)\to L^{2}(\Omega;H)$
	satisfy 
	\begin{equation}
		\begin{aligned}\|P_{U}-P_{\hat{U}}\|_{H\to H} = \|P_{U}-P_{\hat{U}}\|_{L^{2}(\Omega;H)\to L^{2}(\Omega;H)}&\leq\min \{\frac{1}{\sigma_{R}},\frac{1}{\hat{\sigma}_{R}}\}\|X-\hat{X}\|_{L^{2}(\Omega;H)};\\
			\|P_{V}-P_{\hat{V}}\|_{L^{2}(\Omega;H)\to L^{2}(\Omega;H)}&\leq\min \{\frac{1}{\sigma_{R}},\frac{1}{\hat{\sigma}_{R}}\}\|X-\hat{X}\|_{L^{2}(\Omega;H)}.
		\end{aligned}
		\label{eq: proj U and V}
	\end{equation}
   Moreover, we have 
	\begin{equation}
		\|P_{U}+P_{V}-P_{U}P_{V}-(P_{\hat{U}}+P_{\hat{V}}-P_{\hat{U}}P_{\hat{V}})\|_{L^{2}(\Omega;H)\to L^{2}(\Omega;H)}\leq2\min \{\frac{1}{\sigma_{R}},\frac{1}{\hat{\sigma}_{R}}\}\|X-\hat{X}\|_{L^{2}(\Omega;H)}.\label{eq:tangent-proj-Lip}
	\end{equation}	
\end{prop}
}
\begin{proof}
	{
	First, let us prove that $\|P_{U}-P_{\hat{U}}\|_{H\to H} = \|P_{U}-P_{\hat{U}}\|_{L^{2}(\Omega;H)\to L^{2}(\Omega;H)}$.
	On the one hand, it holds for all $v \in L^{2}(\Omega;H)$ that
    \begin{equation*}
    	\begin{aligned}
    		\|(P_{U}-P_{\hat{U}})v\|^2_{L^{2}(\Omega;H)} =  \mathbb{E}[\|(P_{U}-P_{\hat{U}})v(\omega)\|^2_H] \leq
    		\|P_{U}-P_{\hat{U}}\|^2_{H\to H} \|v\|^2_{L^{2}(\Omega;H)}.
    	\end{aligned}
    \end{equation*}
    On the other hand, 
since $v\in H$ with $\|v\|_H=1$ is trivially in $L^2(\Omega;H)$ with $\|v\|_{L^2(\Omega;H)}=1$ we have
    \begin{equation*}
    	\begin{aligned}
    	\|P_{U}-P_{\hat{U}}\|_{H\to H} = \sup\limits_{v\in H, \|v\|_H=1} \|(P_{U}-P_{\hat{U}})v\|_{L^{2}(\Omega;H)} &\leq  \sup\limits_{f \in L^{2}(\Omega;H), \|f\|_{L^2(\Omega;H)}=1} \|(P_{U}-P_{\hat{U}})f\|_{L^{2}(\Omega;H)}\\
    	& = \|P_{U}-P_{\hat{U}}\|_{L^{2}(\Omega;H)\to L^{2}(\Omega;H)}.
    		\end{aligned}
    \end{equation*}
}

{
	Fix $f\in L^{2}(\Omega;H)$. Using $\langle U_{j},X\rangle_{H}=\sigma_{j}V_{j}$
	in $\|P_{\hat{V}^{\perp}}P_{V}f\|_{L^{2}(\Omega;H)}=\bigl\|(\mathrm{id}-P_{\hat{V}})\sum_{j=1}^{R}\mathbb{E}[V_{j}f]V_{j}\bigr\|_{L^{2}(\Omega;H)}$
	yields 
	\[
	\begin{aligned}\|P_{\hat{V}^{\perp}}P_{V}f\|_{L^{2}(\Omega;H)} & =\Bigl\|(\mathrm{id}-P_{\hat{V}})\sum_{j=1}^{R}\frac{1}{\sigma_{j}}\mathbb{E}[V_{j}f]\langle U_{j},X\rangle_{H}\Bigr\|_{L^{2}(\Omega;H)}.\end{aligned}
	\]	
	We have $\sigma_{j}^{-1}\mathbb{E}[V_{j}f](\mathrm{id}-P_{\hat{V}})\langle U_{j},\hat{X}\rangle_{H}=0$
	for $j\in1,\dots,R$; indeed, 
	\[
	0=\langle U_{j},(\mathrm{id}-P_{\hat{V}})\hat{X}\rangle_{H}=\langle U_{j},\hat{X}\rangle_{H}-\langle U_{j},P_{\hat{V}}\hat{X}\rangle_{H}
	\]
	holds because $(\mathrm{id}-P_{\hat{V}})\hat{X}=0$, but \cite[Theorem 8.13]{Leoni.G_2017_book_Sobolev_2nd}
	implies 
	\begin{align*}
		\langle U_{j},P_{\hat{V}}\hat{X}\rangle_{H} & =\sum_{j=1}^{R}\langle U_{j},\mathbb{E}[\hat{V}_{j}\hat{X}]\rangle_{H}\hat{V}_{j}\\
		& =\sum_{j=1}^{R}\mathbb{E}[\langle U_{j},\hat{V}_{j}\hat{X}\rangle_{H}]\hat{V}_{j}=\sum_{j=1}^{R}\mathbb{E}[\hat{V}_{j}\langle U_{j},\hat{X}\rangle_{H}]\hat{V}_{j}=P_{\hat{V}}(\langle U_{j},\hat{X}\rangle_{H}),
	\end{align*}
	and thus $(\mathrm{id}-P_{\hat{V}})\langle U_{j},\hat{X}\rangle_{H}=0$
	for $j\in1,\dots,R$.	
	Hence, the Cauchy--Schwarz inequality implies 
	\[
	\begin{aligned}\|P_{\hat{V}^{\perp}}P_{V}f\|_{L^{2}(\Omega;H)}= & \Bigl\|(\mathrm{id}-P_{\hat{V}})\sum_{j=1}^{R}\frac{1}{\sigma_{j}}\mathbb{E}[V_{j}f]\langle U_{j},(X-\hat{X})\rangle_{H}\Bigr\|_{L^{2}(\Omega;H)}\\
		\leq & \|\mathrm{id}-P_{\hat{V}}\|_{L^{2}(\Omega;H)\to L^{2}(\Omega;H)}
		 \cdot\left(\sum_{j=1}^{R}\frac{1}{\sigma_{j}}\sqrt{\bigl\|\mathbb{E}[V_{j}f]\bigr\|_{H}^{2}\mathbb{E}\bigl[|\langle U_{j},(X-\hat{X})\rangle_{H}|^{2}\bigr]}\right)\\
		\leq & \frac{1}{\sigma_{R}}\left(\sum_{j=1}^{R}\bigl\|\mathbb{E}[V_{j}f]\bigr\|_{H}^{2}\right)^{1/2}\left(\sum_{j=1}^{R}\mathbb{E}\bigl[|\langle U_{j},(X-\hat{X})\rangle_{H}|^{2}\bigr]\right)^{1/2}.
	\end{aligned}
	\]
	Recalling the orthonormality assumptions on $U_{j}$ and $V_{j}$,
	we use Bessel's inequality to obtain $\sum_{j=1}^{R}\mathbb{E}\bigl[|\langle U_{j},(X-\hat{X})\rangle_{H}|^{2}\bigr]\leq\|X-\hat{X}\|_{L^{2}(\Omega;H)}^{2}$,
	Moreover, one can also prove that $\sum_{j=1}^{R}\|\mathbb{E}[fV_{j}]\|_{H}^{2}\leq\|f\|_{L^{2}(\Omega;H)}^{2}$
	following a similar argument to show  Bessel's inequality. 
	It follows $\|P_{\hat{V}^{\perp}}P_{V}f\|_{L^{2}(\Omega;H)}\le\frac{\|f\|_{L^{2}(\Omega;H)}}{\sigma_{R}}\|X-\hat{X}\|_{L^{2}(\Omega;H)}$.
	Therefore, via using Lemma \ref{lem: PuPu^} we conclude that
	\[
	\|P_{V}-P_{\hat{V}}\|_{L^{2}(\Omega;H)\to L^{2}(\Omega;H)}\leq\frac{1}{\sigma_{R}}\|X-\hat{X}\|_{L^{2}(\Omega;H)}.
	\]
	Likewise, $\|P_{U}-P_{\hat{U}}\|_{L^{2}(\Omega;H)\to L^{2}(\Omega;H)}\leq\frac{1}{\sigma_{R}}\|X-\hat{X}\|_{L^{2}(\Omega;H)}$
	holds. By exchanging the roles of $U$,$V$ with $\hat{U}$,$\hat{V}$, respectively, using the homogeneity property of the norm and \eqref{eq: P_UP_u^}, then the inequalities \eqref{eq: proj U and V} follow. 	
	Finally, the last claim \eqref{eq:tangent-proj-Lip} comes from
	\begin{equation*}
		\begin{aligned}\|P_{U}+P_{V}-P_{U}P_{V} &- (P_{\hat{U}}+P_{\hat{V}}-P_{\hat{U}}P_{\hat{V}})\|_{L^{2}(\Omega;H)\to L^{2}(\Omega;H)}\\
			= & \|(P_{U}-P_{\hat{U}})+(\mathrm{id}-P_{U})P_{V}-(\mathrm{id}-P_{\hat{U}})P_{\hat{V}}\|_{L^{2}(\Omega;H)\to L^{2}(\Omega;H)}\\
			= & \|(P_{U}-P_{\hat{U}})+(\mathrm{id}-P_{U})P_{V}-(\mathrm{id}-P_{U})P_{\hat{V}}\\
			& \qquad\qquad\qquad+(\mathrm{id}-P_{U})P_{\hat{V}}-(\mathrm{id}-P_{\hat{U}})P_{\hat{V}}\|_{L^{2}(\Omega;H)\to L^{2}(\Omega;H)}\\
			\leq & \|(P_{U}-P_{\hat{U}})(\mathrm{id}-P_{\hat{V}})\|_{L^{2}(\Omega;H)\to L^{2}(\Omega;H)}+\|(\mathrm{id}-P_{U})(P_{V}-P_{\hat{V}})\|_{L^{2}(\Omega;H)\to L^{2}(\Omega;H)}\\
			\leq &  2\min \{\frac{1}{\sigma_{R}},\frac{1}{\hat{\sigma}_{R}}\}\|X-\hat{X}\|_{L^{2}(\Omega;H)}.
		\end{aligned}
	\end{equation*}
}
\end{proof} 

We conclude this section giving another useful result inherent to orthogonal projections.

\begin{lem}
\label{lem: lip orth proj} Given $\gamma,{N}>0$ two positive constants,
let 
\[
\mathbb{B}_{\gamma,N}:=\{t\in[0,T]\to\boldsymbol{U}_{t}\in\mathbb{R}^{R\times d}\mid\sup\limits _{s\in[0,T]}\|(\boldsymbol{U}_{{s}}\boldsymbol{U}_{{s}}^{\top})^{-1}\|_{\mathrm{F}}\leq\gamma\text{ and }\sup\limits _{s\in[0,T]}\|\boldsymbol{U}_{s}\|_{\mathrm{F}}<\sqrt{{N}}\},
\]
where the invertibility of  $\boldsymbol{U}_{t}\boldsymbol{U}_{t}^{\top}$ is implicitly assumed.  {Given $\boldsymbol{U},\tilde{\boldsymbol{U}} \in B_{\gamma,N}$,} define the orthogonal projectors $P_{\boldsymbol{U}_{t}}^{\mathrm{row}}:=\boldsymbol{U}_{t}{}^{\top}(\boldsymbol{U}_{t}{}\boldsymbol{U}_{t}^{\top})^{-1}\boldsymbol{U}_{t},$
$P_{\tilde{\boldsymbol{U}_{t}}}^{\mathrm{row}}:=\tilde{\boldsymbol{U}_{t}}{}^{\top}(\tilde{\boldsymbol{U}_{t}}\tilde{\boldsymbol{U}_{t}}^{\top})^{-1}\tilde{\boldsymbol{U}_{t}}$
onto the rows of $\boldsymbol{U}_{t}$ and $\tilde{\boldsymbol{U}}_{t}$,
respectively. Then there exists a constant $C_{1}>0$ such that the
following holds 
\[
\sup\limits _{t\in[0,T]}\|P_{\boldsymbol{U}_{t}}^{\mathrm{row}}-P_{\tilde{\boldsymbol{U}_{t}}}^{\mathrm{row}}\|_{\mathrm{F}}\leq C_{1}\sup_{t\in[0,T]}\|\boldsymbol{U}_{t}-\tilde{\boldsymbol{U}_{t}}\|_{\mathrm{F}}.
\]
Further, given $\tilde{\gamma},M>0$ two positive constants, let 
\[
\mathbb{B}_{\tilde{\gamma},M}:=\{t\in[0,T]\to\boldsymbol{Y}_{t}\in[L^{2}(\Omega)]^{R}\mid\sup\limits _{s\in[0,T]}\|\mathbb{E}[{\boldsymbol{Y}}_{s}{\boldsymbol{Y}}_{s}^{\top}]^{-1}\|_{\mathrm{F}}\leq\tilde{\gamma}\text{ and }\mathbb{E}\bigl[\sup\limits _{s\in[0,T]}|\boldsymbol{Y}_{s}|^{2}\bigr]<M\}.
\]
{Given $\boldsymbol{Y},\tilde{\boldsymbol{Y}}\in\mathbb{B}_{\tilde{\gamma},M}$},
let $P_{\boldsymbol{Y}_{t}}[\cdot]:=\boldsymbol{Y}_{t}^{\top}\mathbb{E}[{\boldsymbol{Y}}_{t}{\boldsymbol{Y}}_{t}^{\top}]^{-1}\mathbb{E}[\boldsymbol{Y}_{t}\cdot]$,
$P_{\tilde{\boldsymbol{Y}}_{t}}[\cdot]:=\tilde{\boldsymbol{Y}}_{t}^{\top}\mathbb{E}[{\tilde{\boldsymbol{Y}}}_{t}{\tilde{\boldsymbol{Y}}}_{t}^{\top}]^{-1}\mathbb{E}[\tilde{\boldsymbol{Y}}_{t}\cdot]$
be the orthogonal projectors onto the subspace spanned by the components
of $\boldsymbol{Y}_{t}$ and $\tilde{\boldsymbol{Y}}_{t}$, respectively.
Then there exists a constant $C_{2}>0$ such that the following holds
\[
\mathbb{E}[\sup\limits _{t\in[0,T]}\|P_{\boldsymbol{Y}_{t}}-P_{\tilde{\boldsymbol{Y}}_{t}}\|_{L^{2}(\Omega)\to L^{2}(\Omega)}^{2}]^{\frac{1}{2}}\leq C_{2}\mathbb{E}[\sup\limits _{t\in[0,T]}|\boldsymbol{Y}_{t}-\tilde{\boldsymbol{Y}}_{t}|^{2}]^{\frac{1}{2}}
\]
\end{lem}
\begin{proof}
For $\boldsymbol{Y}_{t},\tilde{\boldsymbol{Y}_{t}}\in[L^{2}(\Omega)]^{R}$
we have
\begin{align*}
\boldsymbol{Y}_{t}^{\top}\mathbb{E}[{\boldsymbol{Y}}_{t}{\boldsymbol{Y}}_{t}^{\top}]^{-1}\mathbb{E}[\boldsymbol{Y}_{t}g] & -\tilde{\boldsymbol{Y}}_{t}^{\top}\mathbb{E}[\tilde{\boldsymbol{Y}}_{t}\tilde{\boldsymbol{Y}}_{t}^{\top}]^{-1}\mathbb{E}[\tilde{\boldsymbol{Y}}_{t}g]\\
= & (\boldsymbol{Y}_{t}^{\top}-\tilde{\boldsymbol{Y}}_{t}^{\top})\mathbb{E}[{\boldsymbol{Y}}_{t}{\boldsymbol{Y}}_{t}^{\top}]^{-1}\mathbb{E}[\boldsymbol{Y}_{t}g]\\
& +\tilde{\boldsymbol{Y}}_{t}^{\top}(\mathbb{E}[{\boldsymbol{Y}}_{t}{\boldsymbol{Y}}_{t}^{\top}]^{-1}-\mathbb{E}[\tilde{\boldsymbol{Y}}_{t}\tilde{\boldsymbol{Y}}_{t}^{\top}]^{-1})\mathbb{E}[\boldsymbol{Y}_{t}g]\\
& +\tilde{\boldsymbol{Y}}_{t}^{\top}\mathbb{E}[\tilde{\boldsymbol{Y}}_{t}\tilde{\boldsymbol{Y}}_{t}^{\top}]^{-1}(\mathbb{E}[\boldsymbol{Y}_{t}g]-\mathbb{E}[\tilde{\boldsymbol{Y}}_{t}g]).
\end{align*}
{Let $\boldsymbol{Y},\tilde{\boldsymbol{Y}}\in\mathbb{B}_{\tilde{\gamma},M}$}.
From \cite[Lemma 3.5]{kazashi2021existence}, there exists a constant
$C(R,M)>0$ such that 
\begin{align*}
\|\mathbb{E}[{\boldsymbol{Y}}_{t}{\boldsymbol{Y}}_{t}^{\top}]^{-1}-\mathbb{E}[\tilde{\boldsymbol{Y}}_{t}\tilde{\boldsymbol{Y}}_{t}^{\top}]^{-1}\|_{\mathrm{F}} & \leq\tilde{\gamma}C(R,M)\mathbb{E}[|\boldsymbol{Y}_{t}-\tilde{\boldsymbol{Y}}_{t}|^{2}]^{\frac{1}{2}}\\
& \leq \tilde{\gamma}C(R,M)\mathbb{E}[\sup_{t\in[0,T]}|\boldsymbol{Y}_{t}-\tilde{\boldsymbol{Y}}_{t}|^{2}]^{\frac{1}{2}},
\end{align*}
and thus 
\begin{align*}
\sup_{t\in[0,T]}|(P_{\boldsymbol{Y}_{t}}-P_{\tilde{\boldsymbol{Y}}_{t}})g| & \leq\tilde{\gamma}\sup_{t\in[0,T]}|\boldsymbol{Y}_{t}-\tilde{\boldsymbol{Y}}_{t}|\,\mathbb{E}[\sup_{t\in[0,T]}|\boldsymbol{Y}_{t}||g|]\\
& +\tilde{\gamma}C(R,M)\Bigl(\sup_{t\in[0,T]}|\tilde{\boldsymbol{Y}}_{t}|\Bigr)\Bigl(\mathbb{E}[\sup_{t\in[0,T]}|\boldsymbol{Y}_{t}-\tilde{\boldsymbol{Y}}_{t}|^{2}]^{\frac{1}{2}})\Bigr)\mathbb{E}[\sup_{t\in[0,T]}|\boldsymbol{Y}_{t}||g|]\\
& +\Bigl(\sup_{t\in[0,T]}|\tilde{\boldsymbol{Y}}_{t}|\Bigr)\tilde{\gamma}(\mathbb{E}[\sup_{t\in[0,T]}|\boldsymbol{Y}_{t}-\tilde{\boldsymbol{Y}}_{t}||g|]).
\end{align*}
Then, using the Cauchy--Schwarz inequality to $\mathbb{E}[\sup_{t\in[0,T]}|\boldsymbol{Y}_{t}||g|]$
and $\mathbb{E}[\sup_{t\in[0,T]}|\boldsymbol{Y}_{t}-\tilde{\boldsymbol{Y}}_{t}||g|]$,
taking the square of both sides and taking the expectation yields
the result. The result for $P_{\boldsymbol{U}_{t}}^{\mathrm{row}}$
can be analogously shown.
\end{proof}

\printbibliography

\end{document}